\documentclass{amsart}

\usepackage[normalem]{ulem}
\usepackage{amsmath,amsthm, amssymb,mathrsfs,  xypic, url, verbatim,stmaryrd}
\usepackage{xr-hyper}
\usepackage{hyperref}
\usepackage{stackengine,scalerel}
\usepackage{todonotes}
\presetkeys{todonotes}{inline, size=\small}{}
\usepackage[capitalize]{cleveref}
\usepackage{graphicx}
\usepackage{quiver}
\usepackage{tikz}
\usetikzlibrary{matrix, arrows}
\usetikzlibrary{cd}

\crefname{maintheorem}{Theorem}{Theorems}
\crefname{maincorollary}{Corollary}{Corollaries}

\newcommand{\mbb}[1]{\mathbb{#1}}

\makeatletter
\newcommand*{\addFileDependency}[1]{
  \typeout{(#1)}
  \@addtofilelist{#1}
  \IfFileExists{#1}{}{\typeout{No file #1.}}
}
\makeatother

\addFileDependency{main-1.tex}
\addFileDependency{main-1.aux}

\begin{document}
\newcommand{\actsonr}{\mathrel{\reflectbox{$\righttoleftarrow$}}}
\newcommand{\actsonl}{\mathrel{\reflectbox{$\lefttorightarrow$}}}

\newcommand{\floor}[1]{\lfloor #1 \rfloor}

\newcommand{\isoeq}{\cong}
\newcommand{\cech}{\vee}
\newcommand{\dsum}{\mathop{\oplus}}
\newcommand{\End}{\mathrm{End}}

\newcommand{\calQ}{\mathcal{Q}}
\newcommand{\calO}{\mathcal{O}}
\newcommand{\calM}{\mathcal{M}}

\newcommand{\frakm}{\mathfrak{m}}

\newcommand{\bbA}{\mathbb{A}}
\newcommand{\bbB}{\mathbb{B}}
\newcommand{\bbC}{\mathbb{C}}
\newcommand{\bbD}{\mathbb{D}}
\newcommand{\bbE}{\mathbb{E}}
\newcommand{\bbF}{\mathbb{F}}
\newcommand{\bbG}{\mathbb{G}}
\newcommand{\bbH}{\mathbb{H}}
\newcommand{\bbI}{\mathbb{I}}
\newcommand{\bbJ}{\mathbb{J}}
\newcommand{\bbK}{\mathbb{K}}
\newcommand{\bbL}{\mathbb{L}}
\newcommand{\bbM}{\mathbb{M}}
\newcommand{\bbN}{\mathbb{N}}
\newcommand{\bbO}{\mathbb{O}}
\newcommand{\bbP}{\mathbb{P}}
\newcommand{\bbQ}{\mathbb{Q}}
\newcommand{\bbR}{\mathbb{R}}
\newcommand{\bbS}{\mathbb{S}}
\newcommand{\bbT}{\mathbb{T}}
\newcommand{\bbU}{\mathbb{U}}
\newcommand{\bbV}{\mathbb{V}}
\newcommand{\bbW}{\mathbb{W}}
\newcommand{\bbX}{\mathbb{X}}
\newcommand{\bbY}{\mathbb{Y}}
\newcommand{\bbZ}{\mathbb{Z}}

\newcommand{\bc}{\mathbf{c}}

\newcommand{\Spec}{\mathrm{Spec}}\
\newcommand{\Triv}{\mathrm{Triv}}
\newcommand{\Loc}{\mathrm{Loc}}
\newcommand{\Et}{\mathrm{Et}}

\newcommand{\ul}[1]{\underline{#1}}

\newcommand{\Vect}{\mathrm{Vect}}
\newcommand{\FilVect}{\mathrm{FilVect}}
\newcommand{\FilVB}{\mathrm{FilVB}}

\newcommand{\trFil}{\mathrm{trFil}}
\newcommand{\gr}{\mathrm{gr}}
\newcommand{\LattFilt}[1]{\mr{Latt}_{#1}\mr{Vect}}

\newcommand{\dR}{\mathrm{dR}}

\newcommand{\HS}{\mathrm{HS}}

\newcommand{\Hodge}{\mathrm{Hdg}}

\newcommand{\BC}{\mathrm{BC}}
\newcommand{\CM}{\mathrm{CM}}

\newcommand{\goodred}{\textrm{good-red}}

\newcommand{\dash}{\mathrm{-}}
\newcommand\+{\mkern3mu}

\newcommand{\rev}[1]{{\color{teal} #1}}
\newcommand{\strike}[1]{{\color{gray} \sout{#1}}}
\newcommand{\spc}[1]{{\color{blue} \textsf{$\blacktriangle\blacktriangle\blacktriangle$ Comment: [#1]}}}

\newcommand{\Iref}[1]{I.\cref{I.#1}}

\newcommand{\triv}{\mr{triv}}
\newcommand{\crys}{\mathrm{crys}}
\newcommand{\Fil}{\mathrm{Fil}}
\newcommand{\calL}{\mathcal{L}}
\newcommand{\Stab}{\mathrm{Stab}}
\newcommand{\Spd}{\mathrm{Spd}}

\newcommand{\MG}{\mathrm{MG}}
\newcommand{\HT}{\mathrm{HT}}
\newcommand{\GH}{\mathrm{GH}}
\newcommand{\GM}{\mathrm{GM}}
\newcommand{\Gr}{\mathrm{Gr}}
\newcommand{\Fl}{\mathrm{Fl}}
\newcommand{\LT}{\mathrm{LT}}
\newcommand{\colim}{\mathrm{colim}}
\newcommand{\cris}{\mathrm{cris}}
\renewcommand{\l}{\left}
\renewcommand{\r}{\right}
\newcommand{\GL}{\mathrm{GL}}
\newcommand{\qpet}{\mr{qpet}}

\newcommand{\Perf}{\mathrm{Perf}}
\newcommand{\Perfd}{\mathrm{Perfd}}
\newcommand{\mc}[1]{\mathcal{#1}}

\newcommand{\mr}[1]{\mathrm{#1}}
\newcommand{\mf}[1]{\mathfrak{#1}}
\newcommand{\ms}[1]{\mathscr{#1}}

\newcommand{\ol}[1]{\overline{#1}}

\newcommand{\Kt}{\mathrm{Kt}}
\newcommand{\Isoc}{\mr{Isoc}} 

\newcommand{\Spa}{\mathrm{Spa}}
\newcommand{\Spf}{\mathrm{Spf}}
\newcommand{\perf}{\mathrm{perf}}
\newcommand{\Spr}{\mr{Spr}}
\newcommand{\Hom}{\mr{Hom}}
\newcommand{\Ker}{\mr{Ker}}
\newcommand{\adm}{\mr{adm}}
\newcommand{\bs}{\backslash}
\newcommand{\BB}{\mathrm{BB}}
\newcommand{\an}{\mathrm{an}}
\newcommand{\AdmFil}{\mathrm{AdmFil}} 

\newcommand{\QQ}{\mathbb{Q}}
\newcommand{\ZZ}{\mathbb{Z}}

\newcommand{\FF}{\mathrm{FF}}

\newcommand{\proet}{\mr{pro\acute{e}t}}
\newcommand{\et}{\mr{\acute{e}t}}
\newcommand{\fet}{\mr{f\acute{e}t}}

\newcommand{\badm}{{\textrm{b-adm}}}
\newcommand{\Qpbreve}{{\breve{\mbb{Q}}_p}}
\newcommand{\Qpbrevebar}{{\overline{\breve{\mbb{Q}}}_p}}

\newcommand{\Hdg}{\mr{Hdg}}
\newcommand{\Hdggen}{\mr{gen}}
\newcommand{\MT}{\mr{MT}}
\newcommand{\Rep}{\mathrm{Rep}}

\newcommand{\FilPhiMod}{\mathrm{MF}^{\varphi}}
\newcommand{\waFilPhiMod}{\mathrm{MF}^{\varphi, \mr{wa}}}

\newcommand{\Isom}{\mr{Isom}}
\newcommand{\Gal}{\mr{Gal}}
\newcommand{\Aut}{\mr{Aut}}
\newcommand{\AdmPair}{\mr{AdmPair}}
\newcommand{\Forget}{\mr{Forget}}

\newcommand{\Cont}{\mr{Cont}}
\newcommand{\AP}{\mathbf{A}}
\newcommand{\GAP}{\mathbf{A}}
\newcommand{\ev}{\mathrm{ev}}
\newcommand{\std}{\mathrm{std}}
\newcommand{\Lie}{\mr{Lie}}
\newcommand{\Sh}{\mr{Sh}}
\newcommand{\univ}{\mr{univ}}
\newcommand{\Bun}{\mr{Bun}}

\newcommand{\spck}[1]{{\color{teal} \textsf{$\dagger\dagger\dagger$ CK: [#1]}}}
\newcommand{\bilatticed}{\mathrm{bl}}

\numberwithin{equation}{subsubsection}
\theoremstyle{plain}

\newtheorem{maintheorem}{Theorem} 
\renewcommand{\themaintheorem}{\Alph{maintheorem}} 
\newtheorem{maincorollary}[maintheorem]{Corollary}
\newtheorem{mainconjecture}[maintheorem]{Conjecture}

\newtheorem*{theorem*}{Theorem}

\newtheorem{theorem}[subsubsection]{Theorem}
\newtheorem{corollary}[subsubsection]{Corollary}
\newtheorem{conjecture}[subsubsection]{Conjecture}
\newtheorem{proposition}[subsubsection]{Proposition}
\newtheorem{lemma}[subsubsection]{Lemma}

\theoremstyle{definition}

\newtheorem{example}[subsubsection]{Example}
\newtheorem{assumption}[subsubsection]{Assumption}
\newtheorem{definition}[subsubsection]{Definition}
\newtheorem{remark}[subsubsection]{Remark}
\newtheorem{question}[subsubsection]{Question}

\renewcommand{\Vect}{\mathrm{Vect}}

\title[Admissible pairs and $p$-adic Hodge structures II]{Admissible pairs and $p$-adic Hodge structures II: The bi-analytic Ax-Lindemann theorem}

\author{Sean Howe}
\email{sean.howe@utah.edu}
\author{Christian Klevdal}
\email{cklevdal@ucsd.edu}

\begin{abstract} We reinterpret and generalize the construction of local Shimura varieties and their non-minuscule analogs by viewing them as moduli spaces of admissible pairs. Our main application is a bi-analytic Ax-Lindemann theorem comparing, in the basic case, rigid analytic subvarieties for the two distinct analytic structures induced by the Hodge and Hodge-Tate period maps and their lattice refinements. The theorem implies, in particular, that the only bi-analytic subdiamonds are special subvarieties, generalizing the bi-analytic characterization of special points given in Part~I. These results suggest that there is a purely local, $p$-adic theory of bi-analytic geometry that runs in parallel to the global, archimedean theory of bi-algebraic geometry arising in the study of unlikely intersection and functional transcendence for Shimura varieties and more general period domains for variations of Hodge structure. 
\end{abstract}
\maketitle

\tableofcontents

\section{Introduction}

Let $C/\mbb{Q}_p$ be an algebraically closed non-archimedean extension. In Part I of this work \cite{howe-klevdal:ap-pahsI} (which we will reference below using the format I.$\square$), we introduced a linear category of $p$-adic Hodge structures over $C$ and a matching semilinear category of cohomological motives over $C$, the basic admissible pairs. As explained in \Iref{s.introduction}, the relation between these categories is, for certain purposes, analogous to the relation between $\mbb{Q}$-Hodge structures and motives over $\mbb{C}$. In particular, there is a rich transcendence theory of periods of admissible pairs, and the main application of the theory in Part I was a $p$-adic transcendence characterization of CM admissible pairs, \Iref{main.transcendence}, that was essentially conjectured in \cite{howe:transcendence}.  

In this sequel, we study moduli of $p$-adic Hodge structures and admissible pairs. In some sense, this is accomplished already in \cite{scholze:berkeley} via the construction of local Shimura varieties and their non-minuscule analogs, and the methods of loc.\ cit.\ play an important role in our work. There is a conceptual difference in our approach: in \cite{scholze:berkeley}, these spaces are a stepping stone towards more general moduli of mixed characteristic local shtuka, and the treatment is more or less adapted to this perspective. By contrast, our starting point is the category of admissible pairs treated as a toy category of motives over $C$, so that we first develop a relative theory of admissible pairs and then construct the natural moduli spaces that arise from this relative theory. Our approach, in particular, highlights the importance of treating moduli spaces for non-reductive structure groups, which are ubiquitous. Thus we also extend certain constructions from \cite{scholze:berkeley, fargues-scholze:geometrization} to the 
non-reductive case.

Our main result, \cref{main.ax-lindemann} (and more generally \cref{theorem.ax-lindemann-general}), is a bi-analytic Ax-Lindemann theorem that compares rigid analytic subvarieties for two different analytic structures on these moduli spaces. This is a local, $p$-adic analog, of global, archimedean Ax-Lindemann theorems --- these latter are geometric or functional transcendence results that are naturally organized by the theory of complex bi-algebraic geometry and play an important role in the study of unlikely intersections on global Shimura varieties and other complex period domains. This connection is explained in more detail in \cref{ss.an-struc-comparison}, after the statement of the theorem. While our main interest in the present work comes from this result, note that, although local Shimura varieties have played frequent and important roles in the study of the Langlands program since Deligne's 1973 letter to Piateskii-Shapiro, to our knowledge, even some of the basic foundational results about special subvarieties and Hodge/Hodge-Tate tensors that we establish here are new. In particular, we are not aware of another source that systematically handles the existence of special subvarieties for non-reductive subgroups, and these results may prove useful in other aspects of the study of local Shimura varieties and their non-minuscule analogs. 

The moduli spaces that arise are, in general, only (ind-)locally spatial diamonds, thus the theory of diamonds and $v$-sheaves as developed in \cite{scholze:etale-cohomology-of-diamonds} provides the natural geometric context for our work. To make sense of the remainder of the introduction, it will suffice to recall that a diamond is a functor on perfectoid spaces that can be realized as the quotient of a perfectoid space by a pro-\'{e}tale equivalence relation, that any diamond $Y$ admits an underlying topological space $|Y|$, and that any locally closed and generalizing subset $|U| \subseteq |Y|$ underlies a locally closed subdiamond $U \subseteq Y$ determined by the property that $f:Z \rightarrow Y$ factors through $U$ if and only if $|f|:|Z| \rightarrow |Y|$ factors through $|U|$. There is a diamondification functor $Y \mapsto Y^\diamond$ from adic spaces over $\mbb{Q}_p$ to diamonds by taking the functor of points on perfectoid spaces over $\mbb{Q}_p$; in particular, if $L$ is a non-archimedean field, there is a diamond $\Spd\+ L=\Spa(L, \mc{O}_L)^\diamond$. Moreover, if $Y$ is a rigid analytic variety over $L$ (by which we mean an adic space locally topologically of finite type over $\Spa(L, \mc{O}_L)$), then $|Y|=|Y^\diamond|$, diamondification factors over the seminormalization functor, and diamondification induces a fully faithful functor from the category of seminormal rigid analytic varieties over $L$ to diamonds over $\Spd\+ L$. We call a diamond over $\Spd\+ L$ rigid analytic if it is in the essential image of this functor. 

\subsection{The bi-analytic Ax-Lindemann theorem}\label{ss.intro-ax-lind}
Let $L$ be a complete discretely valued extension of $\mbb{Q}_p$ with algebraically closed residue field $\kappa$, fix an algebraic closure $\overline{L}/L$, and let $C$ be the completion of $\overline{L}$. Let $\overline{\mbb{F}}_p$ denote the algebraic closure of $\mbb{F}_p$ in $\kappa$, and let $\breve{\mbb{Q}}_p:=W(\overline{\mbb{F}}_p)[1/p] \subseteq L$. Let $\sigma$ be the automorphism of $\breve{\mbb{Q}}_p$ induced by Frobenius on $\overline{\mbb{F}}_p$. Algebraic closures below will be taken in $C$. 

Let $G / \mbb{Q}_p$ be a connected linear algebraic group (not necessarily reductive), and let $[\mu]$ be a conjugacy class of cocharacters of $G_{\overline{\mbb{Q}}_p}$. We write $B(G)$ for the Kottwitz set of isomorphism classes of $G$-isocrystals (equivalently, $\sigma$-conjugacy classes in $G(\breve{\mbb{Q}}_p)$). For $G$ reductive, attached to the conjugacy class of inverse cocharacters $[\mu^{-1}]$, Kottwitz \cite[\S6]{kottwitz:isocrystals-II} has distinguished a subset $B(G,[\mu^{-1}]) \subseteq B(G)$. We extend this definition to general $G$ following \Iref{ss.pahs-invariants}: for $U$ the unipotent radical of $G$, the projection $G \rightarrow G/U$ induces a natural identification $B(G)=B(G/U)$ as well as a natural identification of conjugacy classes of cocharacters in $G_{\overline{\mbb{Q}}_p}$ and $(G/U)_{\overline{\mbb{Q}}_p}$, and we define $B(G,[\mu^{-1}])=B(G/U, [\mu^{-1}])$ via this identification. 

Generalizing work of Scholze \cite{scholze:berkeley} (treating the case of $G$ reductive), for any $b$ representing a class $[b] \in B(G)$ we define in \cref{ss.infinite-level-moduli} an infinite level moduli space $\mc{M}:=\mc{M}_{b,[\mu]}$ parameterizing admissible pairs with $G$-structure and rigidifications of the associated $G$-local system and $G$-isocrystal (to the trivial $G$-local system and $b$, respectively). This space is non-empty if and only if $[b]\in B(G,[\mu^{-1}])$.

The space $\mc{M}$ is a locally spatial and partially proper diamond over the reflex field $\breve{\mbb{Q}}_p([\mu])$, and admits commuting actions of $G(\mbb{Q}_p)$ and $G_b(\mbb{Q}_p)$ through change of trivialization (here $G_b$ is the $\sigma$-centralizer of $b$ so that $G_b(\mbb{Q}_p)$ acts by automorphisms of the associated $G$-isocrystal; when $b$ is basic, $G_b$ is an inner form of $G$). There are natural period maps over $\Qpbreve([\mu])$ measuring the position of the Hodge and Hodge-Tate filtrations or their refinements, the \'{e}tale and de Rham lattices: 
\begin{equation}
    \label{eq.intro-period-diagram-general}
\begin{tikzcd}
	& {\mc{M}_{[\mu], b}} \\
	\\
	{\Gr_{[\mu]}} && {\Gr_{[\mu^{-1}]}} \\
	\\
	{\Fl_{[\mu^{-1}]}^\diamond} && {\Fl_{[\mu]}^\diamond}
	\arrow["{\pi_{\mc{L}_\et}}"{description}, from=1-2, to=3-1]
	\arrow["{\pi_{\mc{L}_\dR}}"{description}, from=1-2, to=3-3]
	\arrow["{\pi_\Hdg}"{description}, curve={height=-12pt}, from=1-2, to=5-1]
	\arrow["{\pi_\HT}"{description}, curve={height=12pt}, from=1-2, to=5-3]
	\arrow["\BB"{description}, from=3-1, to=5-1]
	\arrow["\BB"{description}, from=3-3, to=5-3]
\end{tikzcd}\end{equation}

The targets of the lattice period maps $\pi_{\mc{L}_\bullet}$ are Schubert cells in the $\mbb{B}^+_\dR$-affine Grassmannian for $G$, $\Gr_G$. The map $\pi_{\mc{L}_\et}$ is a $G(\mbb{Q}_p)$-torsor over the open $b$-admissible locus $\Gr_{[\mu]}^{b-\adm}\subseteq \Gr_{[\mu]}$, and, when $b$ is basic, $\pi_{\mc{L_\dR}}$ is a $G_b(\mbb{Q}_p)$-torsor over an open period domain $X_{[\mu]} \subseteq \Gr_{[\mu^{-1}]}$ (we use  $X$ by analogy with the usual notation for complex period domains in the theory of Shimura varieties). 

\begin{remark}
The following table compares our notation for flag varieties and Schubert cells in the $\mbb{B}^+_\dR$-affine Grassmannian with common references:

\renewcommand{\arraystretch}{2}
\begin{tabular}{c|c|c}
        This paper and Part I \cite{howe-klevdal:ap-pahsI} & $\Gr_{[\mu]}$ & $\Fl_{[\mu]}$ \\ 
       \hline Caraiani-Scholze \cite{caraiani-scholze:non-compact}  & $\Gr_{G,\mu^{-1}}^{B^+_\dR}$ & $\mathscr{F}\ell^\std_{G,\mu^{-1}}$ \textrm{ or } $\mathscr{F}\ell_{G,\mu}$ \\
      \hline Scholze-Weinstein \cite{scholze:berkeley} & $\Gr_{\mu}$  & $\mathscr{F}\ell_{G,\mu^{-1}}$ \\
      \hline Fargues-Scholze \cite{fargues-scholze:geometrization}  & $\Gr_{G,\mu}$  & N/A \\
\end{tabular}
\renewcommand{\arraystretch}{1}

For the comparison with \cite{caraiani-scholze:non-compact}, we note that there is an inverse in the definition of the Schubert cell on the bottom of \cite[p. 683]{caraiani-scholze:non-compact} and we refer to \Iref{remark.cocharacter-filtration} for a comparison with the two normalizations for flag varieties in \cite{caraiani-scholze:non-compact}. In \cite{scholze:berkeley}, see Definition 19.2.2 for the Schubert cell and Definition 19.4.1 for the flag variety (note, in particular, the use of the \emph{opposite} Borel). See \Iref{remark.cocharacter-filtration} for an explanation of why our normalization for flag varieties is a reasonable one from the perspective of moduli of filtrations. 
\end{remark}

We suppose now that $b$ represents a (necessarily unique) basic\footnote{As in \Iref{ss.pahs-invariants}, our definition of basic is that the slope morphism is central. If $G$ is reductive, then there is always a unique basic class in $B(G,[\mu^{-1}])$. In general, if $M$ is the quotient of $G$ by its unipotent radical, realized as a Levi subgroup of $G$, then $B(M,[\mu^{-1}])=B(G,[\mu^{-1}])$, but the slope morphism for the unique basic class for $M$ may no longer be central in $G$, in which case $B(G,[\mu^{-1}])$ has no basic class.} element in $B(G,[\mu^{-1}])$. We consider the base change $\mc{M}_{C} := \mc{M}_{b,[\mu]} \times_{\Spd\+ \breve{\mbb{Q}}_p([\mu])} \Spd\+ C$, which also has an action of $\mf{I}:=\mr{Aut}_{\mr{cont}}(C/L([\mu]))$. Because of the basic assumption, there are \emph{two} natural presentations of $\mc{M}_{C}$ as a limit of finite \'{e}tale maps
\begin{equation}\label{eq.intro-limit-presentations} \mc{M}_{C}=\lim_{\substack{K \leq G(\mbb{Q}_p) \times \mf{I} \\\textrm{compact open}}}\mc{M}_{C}/K \textrm{  and  } \mc{M}_{C}=\lim_{\substack{K \leq G_b(\mbb{Q}_p) \times \mf{I} \\ \textrm{compact open}}} \mc{M}_{C}/K. \end{equation}
The first tower consists of \'{e}tale covers of $\Gr_{[\mu], L([\mu])}^{b-\adm}$ via $\pi_{\mc{L}_\et}$ while the second consists of \'{e}tale covers of $X_{[\mu], L([\mu])}$ via $\pi_{\mc{L}_\dR}$.

In complex bi-algebraic geometry \cite{klingler-ullmo-yafaev:bialgebraic}, an Ax-Lindemann theorem compares closed subvarieties for two different algebraic structures on a complex analytic variety (we recall the statement for Shimura varieties below in \cref{theorem.complex-al-sv}). Our bi-analytic Ax-Lindemann theorem compares closed rigid analytic subvarieties arising in the two different limit presentations \eqref{eq.intro-limit-presentations} of $\mc{M}_{C}$. To simplify notation, in the introduction we will state the result only in the case that $[\mu]$ is minuscule (that is, the adjoint action of $[\mu]$ on $\Lie\+ G$ has weights in $\{-1,0,1\}$); {\bf we impose this minuscule condition for the rest of \cref{ss.intro-ax-lind}} (see \cref{theorem.ax-lindemann-general} for the general case).  This hypothesis is equivalent to requiring that the maps $\BB$ in \eqref{eq.intro-period-diagram-general} are isomorphisms so that $\Gr_{[\mu]}^{b-\adm}=\Fl_{[\mu^{-1}]}^{b-\adm, \diamond}$ for $\Fl_{[\mu^{-1}]}^{b-\adm}$ an open rigid analytic subvariety of $\Fl_{[\mu^{-1}]}/\Qpbreve([\mu])$, and similarly for $X_{[\mu]}$. In the non-minuscule case these period domains may not be rigid analytic, but, by \cref{theorem.bdr-aff-grass-properties}-(4), they always admit many natural rigid analytic subdiamonds that can be described using loci on the flag variety satisfying Griffiths transversality.

Thus the \'{e}tale covers $\mc{M}_C/K$ (resp. $\mc{M}_C/K_b$) of $\Fl_{[\mu^{-1}], L([\mu])}^{b-\adm, \diamond}$ (resp. $X_{[\mu], L([\mu])}$) appearing in \eqref{eq.intro-limit-presentations} are represented by smooth rigid analytic varieties over $L([\mu])$. For $G$ reductive, these are\footnote{By definition when $K \leq G(\mbb{Q}_p) \times \mf{I}$, and by this definition combined with a duality of basic local Shimura data exchanging $G$ and $G_b$ when $K_b \leq G_b(\mbb{Q}_p) \times \mf{I}$ (see \cref{theorem.duality}).} the basic local Shimura varieties of \cite{scholze:berkeley} whose existence was predicted in \cite{rapoport-viehmann}. For our next definition, we note that there is a natural  Zariski topology on any rigid analytic variety $Z$: it is the topology whose closed sets are those $V \subseteq Z$ such that, for any affinoid open $U \subseteq Z$, $V \cap C$ is the vanishing locus of a finite set of functions in $\mathcal{O}_{Z}(U)$; by \cite[9.5.2, Corollary 7]{BoschGuntzerRemmert.NonArchimedeanAnalysis} these are exactly the vanishing sets of the coherent ideal sheaves in $\mc{O}_Z$.

\begin{definition}\label{defn.intro-statement-definitions}\hfill
\begin{itemize}
    \item The \emph{Hodge (resp. Hodge-Tate) analytic Zariski topology} on $\mc{M}_C$ is the inverse limit of the rigid analytic Zariski topologies on $\mc{M}_C/K$ (resp. $\mc{M}_C/K_b$) for $K \leq G(\mbb{Q}_p) \times \mf{I}$ (resp. $K_b \leq G_b(\mbb{Q}_p) \times \mf{I}$) compact open.
    \item an \emph{irreducible Hodge (resp. Hodge-Tate) analytic set} is a connected component of the preimage in $\mc{M}_{C}$ of an irreducible locally closed rigid subvariety of $\mc{M}_{C}/K$ for $K \leq G(\mbb{Q}_p) \times \mf{I}$ compact open (resp.  $\mc{M}_{C}/K_b$ for $K_b \leq G_b(\mbb{Q}_p) \times \mf{I}$ compact open).   
    \item A locally closed subdiamond of $\mc{M}_{C}$ is \emph{bi-analytic} if it is both an irreducible Hodge analytic set and an irreducible Hodge-Tate analytic set.    
    \item A \emph{special} subvariety of $\mc{M}_{C}$ is a connected component of a subdiamond of the form $\mc{M}_{b_H, [\mu_H], C}$ for $H \leq G$ a closed subgroup and $[\mu_H], b_H$ refining $[\mu], b$ (we allow a $\sigma$-conjugation from $b_H$ to $b$; see \cref{defn.special-subvarieties-inf-level}). 
\end{itemize}
\end{definition}

\begin{example}\label{example.algebraic-rigid-subvarieties}
Consider $Y=\mc{M}_{C}/K$  for $K \leq G(\mbb{Q}_p) \times \mf{I}$ over $\Spd\+ E$, $E=C^K$ (so $[E:L([\mu])]<\infty$). We describe two natural \emph{algebraic} sources of irreducible closed rigid analytic subvarieties of $Y$ (and thus of irreducible Hodge analytic sets in $\mc{M}_{C}$). 
\begin{enumerate}
    \item $Y$ is an \'{e}tale cover via the Hodge period map of the open partially proper subset $\Fl^{b-\adm}_{[\mu^{-1}], E} \subseteq \Fl_{[\mu^{-1}], E}$. If $Z \subseteq \Fl_{[\mu^{-1}], E}$ is a closed algebraic subvariety defined over $E$, then any irreducible component of the preimage in $Y$ of $Z$ is a closed rigid subvariety. 
    \item Uniformization can often be used to construct a map from $Y$ to the basic locus of a global Shimura variety $\Sh/E$ (see \cref{s.other-ax-lindemann}). If $Z \subseteq \Sh$ is a closed algebraic subvariety defined over $E$, then any irreducible component of the preimage in $Y$ of $Z$ is a closed rigid subvariety. 
\end{enumerate}
Note that, when both structures are present, they can be viewed as two different algebraic structures underlying the same analytic structure on $Y$. A very different type of bi-\emph{algebraic} $p$-adic Ax-Lindemann theorem compares these two algebraic structures on local Shimura varieties (see also \cref{remark.ax-schan-andre-oort-bialagebraic} and \cref{s.other-ax-lindemann}). 
\end{example}

\begin{example}\label{example.non-red-special-sv}
Even if $G$ is reductive, there are often special subvarieties corresponding to non-reductive subgroups of $G$. For example, if $G=\GL_4$ and $\mu$ is the minuscule cocharacter $\mu(z)=\mr{diag}(z,z,1,1)$, then there is a special variety corresponding to the standard parabolic $H$ with Levi subgroup $\GL_2 \times \GL_2$. In this case, the local Shimura varieties for $G$ are moduli spaces of certain two-dimensional height four $p$-divisible formal groups, and the special subvariety for $H$ classifies those that are constructed as extensions of one-dimensional height two $p$-divisible formal groups (working with filtered isocrystals, it is easy to find non-trivial extensions).  
\end{example}

\begin{maintheorem}[minuscule bi-analytic Ax-Lindemann; see also \cref{theorem.ax-lindemann-general}]\label{main.ax-lindemann}\hfill\\ The Hodge-Tate (resp. Hodge) analytic Zariski closure of an irreducible Hodge (resp. Hodge-Tate) analytic set in $\mc{M}_C$ is a special subvariety. In particular, special subvarieties are the only closed bi-analytic sets in $\mc{M}_{C}$. 
\end{maintheorem}

\begin{example}\label{example.algebraic-condition} \cref{main.ax-lindemann} has the following consequence, which directly mirrors a common formulation of the Ax-Lindemann theorem for complex Shimura varieties (see \cref{theorem.complex-al-sv} below): 
Let $Z'$ be an algebraic subvariety of $\Fl_{[\mu], L([\mu])}$ and let $Z$ be an irreducible component of the rigid analytic variety $Z'_C \cap X_{[\mu], C}$. Then, any connected component $\tilde{Z}$ of $\pi_{\HT}^{-1}(Z)$ is a irreducible Hodge-Tate analytic set, and \cref{main.ax-lindemann} implies that the rigid analytic Zariski closure of the image of $\tilde{Z}$ in the finite level local Shimura variety $\mc{M}_{b,[\mu],C}/K$ is a (finite level) special subvariety. 
\end{example}

\begin{remark}
The characterization of zero dimensional bi-analytic subvarieties in \cref{main.ax-lindemann} (and \cref{theorem.ax-lindemann-general} for the non-minuscule case) is equivalent to  \Iref{main.G-structure-period-maps}. More generally, \cref{main.ax-lindemann} and \cref{theorem.ax-lindemann-general} should be viewed as geometric transcendence results. The minuscule case, \cref{main.ax-lindemann}, says that, given an $\overline{L}$-rigid analytic condition on the Hodge-Tate (resp. Hodge) filtration, the strictest closed $\overline{L}$-rigid analytic condition on the Hodge (resp. Hodge-Tate) filtration that implies it is equivalent to one cutting out a special subvariety or, equivalently, requiring that certain tensors be Hodge (resp. Hodge-Tate) tensors. 
\end{remark}

\begin{remark}A special subvariety is $\overline{\breve{\mbb{Q}}}_p$-bi-analytic, but we obtain a stronger theorem by allowing arbitrary $L$ rather than only finite extensions of $\breve{\mbb{Q}}_p$ since the residue field may be larger than $\overline{\mbb{F}}_p$, in which case there are $\overline{L}$-analytic subvarieties that are not $\overline{\breve{\mbb{Q}}}_p$-analytic (see also \Iref{theorem.transcendence-ht-filt-one-dim} and \Iref{main.transcendence}).
\end{remark}

\subsection{Analytic structures and comparison with other results}\label{ss.an-struc-comparison}
The definition of the Hodge/Hodge-Tate Zariski topology above is convenient for generalization to the non-minuscule case (where we encounter some subtleties in the definitions; see \cref{s.bi-an-ax--lind-theorem}), but in the minuscule case there is another interesting variant. Indeed, we obtain natural sheaves of functions $\mc{O}^{\Hdg-\an}$ and $\mc{O}^{\HT-\an}$  inside the diamond structure sheaf $\mc{O}$ on $\mc{M}_C$ consisting of those that are locally pulled back from a finite level space in one of the two presentations as a limit of rigid analytic varieties. The Hodge (resp. Hodge-Tate) analytic functions on a quasi-compact open subset $U \subseteq \mc{M}_C$ can be characterized as those elements of $\mc{O}(U)$ that are smooth (i.e.\ fixed by an open subgroup) for the action of the open stabilizer of $U$ in $G(\mbb{Q}_p) \times \mf{I}$ (resp. $G_b(\mbb{Q}_p) \times \mf{I}$). Note that, on a general diamond over $C$, the structure sheaf $\mc{O}$ may admit very few sections, but for $[\mu]$ minuscule we expect that $\mc{M}_C$ is a perfectoid space and that for any quasi-compact open $U$ both $\mc{O}^{\Hdg-\an}(U)$ and $\mc{O}^{\HT-\an}(U)$ are dense in the (much larger) $\mc{O}(U)$. We can then define Zariski topologies using the vanishing sets of ideal sheaves of $\mc{O}^{\Hdg-{\an}}$ or $\mc{O}^{\HT-{\an}}$ as the closed sets --- this yields, a priori, more closed sets than in the Hodge/Hodge-Tate Zariski topologies defined above, but the induced subspace topologies agree on any quasi-compact open (in the diamond topology) and for \cref{main.ax-lindemann} the distinction is irrelevant. 

We would like to view the diamond structure sheaf as analogous to the sheaf of continuous functions on a topological space, with the subsheaves $\mc{O}^{\Hdg-\an}$ and $\mc{O}^{\HT-\an}$ providing two distinct ways to promote it to a type of analytic structure (over $L$) in the same way that a complex analytic manifold can be thought of as a topological manifold equipped with a distinguished subsheaf of analytic functions satisfying certain properties. This can be motivated also by simpler examples: 

\begin{example}\label{example.Zp}
    Consider the profinite set $\mbb{Z}_p$. As a diamond over $C$, its underlying topological space is $\mbb{Z}_p$ and its diamond structure sheaf is the sheaf of continuous $C$-valued functions on $\mbb{Z}_p$. For many purposes, however, it is more natural to equip it with the subsheaf of locally analytic functions as in the theory of $p$-adic manifolds or even the subsheaf of locally constant $C$-valued functions, i.e.\ the structure sheaf of $\mbb{Z}_p$ as a scheme over $C$. The former can be obtained via pullback of the structure sheaf on $\mbb{A}^1_C$ under the inclusion $\mbb{Z}_p \hookrightarrow \mbb{A}^1_C$, while the latter can be obtained from the presentation $\mbb{Z}_p=\lim \mbb{Z}/p^n\mbb{Z}$ where each $\mbb{Z}/p^n\mbb{Z}$ is viewed as a finite union of copies of $\Spa\+ C$. The same remarks apply more generally to the description of any $p$-adic Lie group or $p$-adic manifold in the sense of Serre.
\end{example}

\begin{remark}\label{remark.analytic-structures}
    We do not attempt to make a precise definition of analytic structure on a diamond here, but this is a richer notion than might first seem apparent: for example, there is a common refinement to the Hodge-Tate and Hodge analytic structures on $\mc{M}_C$ that comes from pullback along an embedding into a rigid analytic variety --- as explained in \cite{howe.tangent-bundles-of-p-adic-manifold-fibrations}, this is related to the Banach-Colmez tangent spaces of \cite{fargues-scholze:geometrization, ivanov-weinstein, howe.inscription-and-p-adic-periods} and the locally analytic vectors of \cite{pan:loc-an}.
\end{remark}

With this rough idea in place, we recall the Ax-Lindemann theorem for connected Shimura varieties: Let $(G,X^+)$ be a connected Shimura datum, and let
\[ J: X^+ \rightarrow \mr{Sh}_G^\circ(\mbb{C})=\Gamma \backslash X^+ \]
denote the uniformization map for an associated connected Shimura variety $\Sh_{G}^\circ$ over $\mbb{C}$. The period map is a complex analytic embedding $X^+ \hookrightarrow \Fl_G(\mbb{C})$. 

\begin{theorem}[hyperbolic Ax-Lindemann of \cite{klingler-ullmo-yafaev:hyperbolic-al} in the formulation of Theorem 4.28 of \cite{klingler-ullmo-yafaev:bialgebraic}]\label{theorem.complex-al-sv} If $Z \subseteq \Fl_G$ is a complex algebraic subvariety and $Z_0$ is an analytic irreducible component of $Z(\mbb{C}) \cap X^+$, then the algebraic Zariski closure of $J(Z_0)$ in $\Sh_{G}^\circ(\mbb{C})$ is a weakly special subvariety. In particular, the $\mbb{C}$-bi-algebraic subvarieties are exacly the weakly special subvarieties. \end{theorem}

The special subvarieties in this context are defined in essentially the same way as the local $p$-adic case using morphisms of connected Shimura data $(H, X_H^+) \rightarrow (G, X^+)$, and weakly special subvarieties come from allowing translations arising from the product structure of the adjoint quotient of $H$ (\cite[Theorem 3.5]{klingler-ullmo-yafaev:bialgebraic}). This theorem says that, for any complex algebraic condition on the flag variety (i.e.\ a polynomial condition on the algebraic moduli measuring the filtration in the universal variation of Hodge structures), the strictest complex algebraic condition on $\Sh_{G}^\circ$ that implies it is, up to a natural translation, the condition of belonging to a sub-Shimura variety or, equivalently, the condition that certain tensors are Hodge tensors.

This type of Ax-Lindemann theorem is naturally formulated in the context of bi-algebraic geometry \cite{klingler-ullmo-yafaev:bialgebraic, klingler:hodge-loci-atypical-conjectures}. In particular, one reason that \cref{main.ax-lindemann} is interesting is that it shows there is a non-trivial parallel theory of bi-analytic geometry in our local $p$-adic setting --- compare \cref{theorem.complex-al-sv} with the statement of \cref{main.ax-lindemann} and \cref{example.algebraic-condition}. There are, however, two important differences between the minuscule case of \cref{main.ax-lindemann} and the Ax-Lindemann theorem for Shimura varieties: in \cref{main.ax-lindemann} we obtain a characterization of special rather than weakly special subvarieties, and in \cref{main.ax-lindemann} there is a symmetry between the two analytic structures not present in the complex setting. In fact, these can be explained together: whereas the  Ax-Lindemann theorem for Shimura varieties fits into the world of $\mbb{C}$-bi-algebraic geometry, our result lives in $\overline{L}$-bi-analytic geometry and is an analog of a (conjectural) sideways Ax-Lindemann theorem in $\overline{\mbb{Q}}$-bi-algebraic geometry; in fact, we only obtain the statement in a form matching the complex Ax-Lindemann theorem by combining this with a duality of basic local Shimura varieties that swaps the period maps (see \cref{ss.intro-proof-outline} for more on the proof).   

\begin{remark}
For complex Shimura varieties, the special subvarieties are exactly the weakly special subvarieties such that the translation component is by a special point. In particular, by combining the Ax-Lindemann theorem with a transcendence characterization of special points due to Cohen \cite{cohen:transcendence} and Shiga and Wolfart \cite{shiga-wolfart:transcendence} (see \Iref{theorem.complex-av-transcendence} and surrounding discussion), in the abelian type case one finds that the $\overline{\mbb{Q}}$-bi-algebraic subvarieties are precisely the special subvarieties. This characterization of the $\overline{\mbb{Q}}$-bi-algebraic subvarieties is a result of Ullmo-Yafaev \cite{ullmo-yafaev:characterization} predating the Ax-Lindemann theorem in that setting; instead, they used the theorem of Deligne-Andr\'{e} to show that every $\mbb{C}$-bi-algebraic subvariety is weakly special, then invoked the result of Cohen and Shiga-Wolfart. \end{remark}

\begin{remark}\label{remark.non-minuscule-complex-analogy} The complex Ax-Lindemann theorem is known for all admissible variations of Hodge structures \cite{bakker-tsimerman:ax-schanuel} (and even mixed Hodge structures \cite{klingler-gao:ax-schanuel-mixed}) over smooth quasi-projective varieties. However, the analogy between these results and \cref{main.ax-lindemann}/\cref{theorem.ax-lindemann-general} is not as strong outside of the minuscule case --- the issue is that, for a variation of $p$-adic Hodge structures (see Sections \ref{ss.intro-relative-theory} and \ref{ss.vpahs}) over a smooth connected rigid analytic variety over a $p$-adic field, the de Rham lattice period map on a connected trivializing cover will never have open image outside of the local Shimura case. This breaks the step in our argument where we flip between the Hodge and Hodge-Tate period maps, so we only get a ``sideways" Ax-Lindemann (cf. \cref{main.hodge-tate-image}). 
\end{remark}

\begin{question}
By the discussion above, we should view \cref{main.ax-lindemann} as a local, $p$-adic $\overline{L}$-bi-analytic Ax-Lindemann theorem. Is there a $C$-bi-analytic version?
\end{question}

\begin{remark}\label{remark.ax-schan-andre-oort-bialagebraic}
In the global, archimedean setting of complex Hodge theory, Ax-Lindemann type results are consequences of even more general Ax-Schanuel type results (see, e.g. \cite[\S7]{klingler:hodge-loci-atypical-conjectures}). It is not  clear how to formulate a local $p$-adic Ax-Schanuel conjecture even in the case of local Shimura varieties because, with the notation of \cref{main.ax-lindemann}, one needs to define the dimension of the intersection of a rigid analytic subvariety of $\mc{M}_C/K \times_{L([\mu])} \mc{M}_C/K_b$ with the image of the diagonal map from $\mc{M}_C$. It may be possible to formulate a conjecture using the refined analytic structure alluded to in \cref{remark.analytic-structures} and/or the theory of Banach-Colmez tangent bundles developed in \cite{howe.inscription-and-p-adic-periods}.  

In the global, archimedean setting, Ax-Lindemann and Ax-Schanuel theorems are part of a larger framework of unlikely intersection in bi-algebraic geometry, and play an important role in the proof of the Andr\'{e}-Oort conjecture and its generalizations following the Pila-Zannier strategy. It seems unlikely that there is an interesting version of Andr\'{e}-Oort or related unlikely intersection questions in the local $p$-adic setting: simple examples show that the most naive generalizations of Andr\'{e}-Oort to this setting are not true. One issue is that, in certain cases, the flag variety for the Hodge period map is naturally defined over a finite extension $K/\mbb{Q}_p$ such that all of the $K$-points correspond to CM structures, so that the CM points will be Zariski dense in any subvariety defined over $K$ with a smooth $K$-point. 

On the other hand, as we have already alluded to in \cref{example.algebraic-rigid-subvarieties}, there is almost certainly an interesting broader bi-algebraic theory in the global, $p$-adic setting (where the position of the \'{e}tale lattice is measured against a trivialization of a global Kottwitz isocrystal). The bi-algebraic Ax-Lindemann of Chambert-Loir and Loeser \cite{chambert-loir-loeser:nonarchimedean-ax-lindemann} lives in this context (in the cases corresponding to products of Shimura curves uniformized by products of the Drinfeld half plane). As in \cref{example.algebraic-rigid-subvarieties}, the two algebraic structures in \cite{chambert-loir-loeser:nonarchimedean-ax-lindemann}  underlie the same analytic structure (the one coming from the Hodge period map), so this result is completely orthogonal to the bi-analytic results described here. We discuss these points in detail in \cref{s.other-ax-lindemann}; see also \Iref{s.introduction} for more on the relations between the local $p$-adic, global $p$-adic, and global archimedean contexts.
\end{remark}

\subsection{The relative theory of admissible pairs and $p$-adic Hodge structures.}\label{ss.intro-relative-theory}
As indicated above, we view the spaces that arise in \cref{main.ax-lindemann} and its proof as moduli of admissible pairs / $p$-adic Hodge structures. We thus briefly discuss what we mean by an admissible pair or $p$-adic Hodge structure in the relative case. 

Recall from \Iref{def.admpair} (resp. \Iref{def.p-adic-hs}) that, for $C/\mbb{Q}_p$ an algebraically closed non-archimedean extension, an admissible pair (resp. a $p$-adic Hodge structure) over $C$ is a $B^+_\dR$-latticed isocrystal $(W, \mc{L}_\et)$ (resp. a $\mbb{Q}$-graded $B^+_\dR$-latticed $\mbb{Q}_p$-vector space) satisfying an admissibility condition formulated using modifications of vector bundles on the Fargues-Fontaine curve. Here $B^+_\dR$ is the integral de Rham period ring of Fontaine attached to the perfectoid field $C$.  

If $Y$ is a locally spatial diamond then there is a natural period sheaf $\mbb{B}^+_\dR$ on the $v$-site of $Y$ interpolating the Fontaine rings $B^+_\dR$ at all geometric points. A \emph{neutral admissible pair} (resp. \emph{neutral $p$-adic Hodge structure}) on $Y$ is an isocrystal $W$ (resp. $\mbb{Q}$-graded $\mbb{Q}_p$-vector space $V$) equipped with a $\mbb{B}^+_\dR$-lattice on the $v$-sheaf $W \otimes_{\breve{\mbb{Q}}_p} {\mbb{B}_\dR}$ (resp. $V \otimes_{\mbb{Q}_p} {\mbb{B}_\dR}$) such that the restriction of this data to every geometric point of $Y$ is an admissible pair in the sense of \Iref{def.admpair} (resp. a $p$-adic Hodge structure in the sense of \Iref{def.p-adic-hs}). This is a terrible definition from the perspective of descent, but it suffices to study the moduli spaces of admissible pairs (resp. $p$-adic Hodge structures) that arise here because these all include a trivialization of the underlying isocrystal (resp. graded vector space) as part of the data they parameterize. 

\begin{example} If $L$ is a $p$-adic field with algebraically closed residue field, then a neutral admissible pair over $\Spd\+ L$ is the same as an admissible pair with good reduction over $L$ as defined in \Iref{ss.galois-representation}. Indeed, $\Spd\+ L = \Spd\+ \overline{L}^\wedge / \Gal(\overline{L}/L)$, and the lattices on an isocrystal arising from filtrations over $L$ are precisely those preserved by the Galois action on $W \otimes \mbb{B}_\dR$, i.e.\ that descend to $\Spd\+ L$.  \end{example}

In Part III we will view neutral admissible pairs (resp. neutral $p$-adic Hodge structures) as living inside a better behaved category of admissible pairs (resp. $p$-adic Hodge structures) where we allow an arbitrary $v$-local system of isocrystals (resp. $\mbb{Q}$-graded $\mbb{Q}_p$-local system) in order to obtain a $v$-stack. For $Y$ connected, the neutral admissible pairs will be a full Tannakian subcategory of this better behaved category. The price for enforcing $v$-descent is that we obtain many more objects than actually arise from interesting geometric constructions. Thus, after introducing this formalism, a main point of Part III will be to formulate a conjecture stating that the admissible pairs corresponding to $p$-adic Hodge structures arising from de Rham local systems (we call these variations of $p$-adic Hodge structure --- see \cref{ss.vpahs}) are not so far from the neutral admissible pairs that we use here. When $Y=\Spd\+ L$ for a $p$-adic field $L$, this conjecture is a consequence of classical results in Fontaine's crystalline theory. We have postponed this discussion until Part III and worked instead with neutral objects here in order to avoid some technical points about local systems of isocrystals that would create a technical burden in the foundational definitions that is  unnecessary for our main results.

\subsection{Outline of proof \cref{main.ax-lindemann}}\label{ss.intro-proof-outline}
Because of the duality (\cref{theorem.duality}) between basic data that swaps the two period maps, to prove \cref{main.ax-lindemann} it suffices to prove that the Hodge-Tate analytic Zariski closure of an irreducible Hodge analytic set is a special subvariety. To establish this, we consider the universal neutral $G$-admissible pair restricted to the rigid analytic variety at finite level, and the Tannakian formalism  baked in from the beginning allows us to reduce to the case that $G$ is the motivic Galois group. We then need a geometric argument based on analysis of Hodge loci to show that there is a classical point such that the motivic Galois group at that point is already $G$ (\cref{lemma:Hodge-generic-locus-for-rigid-analytic}) --- this is inspired by similar arguments in complex Hodge theory, but the Hodge loci are not as nice in our setting (this is another manifestation of the failure of reductivity for motivic Galois groups in this context). Thus, the argument depends in the end on a delicate application of the Baire category theorem to move between density of weakly Shilov points and classical points. Once we have a Hodge generic classical point, \Iref{corollary.galois-rep-dense-open} implies the image of the Hodge-Tate period map is dense in the flag variety, and this is essentially enough to conclude. 

In fact, what we actually prove using this strategy is the non-minuscule version \cref{theorem.ax-lindemann-general}. The main added difficulty in this non-minuscule version is simply to formulate the correct definitions and basic results (e.g. that special subvarieties are closed) taking into account some subtleties arising for non-reductive subgroups (in particular, it is possible to have a special subdiamond whose image in one period domain is rigid analytic but whose image in the other is not). We note that this generalization is one of the most novel and interesting aspects of our work: it indicates that there is a rich and interesting theory of analytic geometry living within (or on top) of the world of diamonds but going beyond simply rigid analytic varieties; in particular, we are not aware of any earlier work on the analytic geometry of these non-rigid analytic diamonds. We refer the reader to \cref{ss.inf-level-special-sv} and \cref{s.bi-an-ax--lind-theorem} for more details and some explicit examples. 

Outside of the basic case, the same method of proof gives:

\begin{maintheorem}\label{main.hodge-tate-image}
Let $Z$ be a geometrically irreducible rigid analytic variety over a strict $p$-adic field $L$, and let $A=(W, \mc{L}_\et)$ be a neutral admissible pair on $Z^\diamond$. Then $A$ has a generic motivic Galois group $G$ and is good with lattice invariant $[\mu]$, a (not necessarily minuscule!) conjugacy class of cocharacters of $G_{\overline{\mbb{Q}}_p}$ defined over $L$.  For $C=\overline{L}^\wedge$, if $\tilde{Z} \rightarrow Z_C^\diamond$ is a geometric connected component of the associated $G(\mbb{Q}_p)$-torsor over $Z^\diamond$ and $\pi_\HT: \tilde{Z} \rightarrow \Fl_{[\mu], L([\mu])}^\diamond$ is the associated Hodge-Tate period map, then $\pi_\HT(\tilde{Z})$ is rigid Zariski dense in any connected open of $\Fl_{[\mu],L([\mu])}^\diamond$ that it intersects.  
\end{maintheorem}

The argument for \cref{main.hodge-tate-image} is similar to the argument for \Iref{main.G-structure-period-maps}, and just as the argument for that theorem could be transported to the complex setting to give a conditional transcendence characterization of CM motives over $\mbb{C}$, one could also obtain a conditional ``sideways $\overline{\mbb{Q}}$-Ax-Lindemann theorem" (which cannot be flipped into a regular Ax-Lindemann theorem, since there is no duality swapping the two algebraic structures in the global archimedean setting). We leave a precise statement to the interested reader. 

\subsection{Outline}
In \cref{s.preliminaries} we recall some basic facts about $v$-sheaves, diamonds, local systems, and vector bundles that will be used throughout the paper. In \cref{s.filtrations-and-lattices} we extend the discussion of the relation between filtrations and lattices from the case of a point treated in \Iref{s.filtrations-and-lattices} to an arbitrary base; one key point is again the relation between exactness of filtrations and types of lattices, but most of the hard work has already been accomplished in Part I. An essential new feature in the relative setting is the relation between lattices and filtrations satisfiying Griffiths tranversality over rigid analytic varieties coming from the formalism of \cite{scholze:p-adic-ht} and spelled out in detail in \cref{ss.bilatticed-equivalence-rigid-analytic}. 

In \cref{s.G-bundles-ff} we recall some results from \cite{fargues-scholze:geometrization} on the moduli stack $\mr{Bun}_G$ of $G$ bundles from the Fargues-Fontaine curve and extend them to $G$ non-reductive (this extension is straightforward). These results are used to define the moduli spaces $X_{[\mu]}$ and $\Gr_{[\mu]}^{b-\adm}$ and also to prove the duality theorem, \cref{theorem.duality}. 

In \cref{s.aff-grass} we generalize some results from \cite{scholze:berkeley} on the $\mbb{B}^+_\dR$-affine Grassmannian to non-reductive groups; we also apply the results relating filtrations satisfying Griffiths transverality and lattices noted above to describe the maps from a rigid analytic variety. The main difficulty is to show that, for $G$ non-reductive, the Schubert cells have any nice properties. This is not obvious (we illustrate this with some examples), and here again we use some of the tools developed in Part I to understand the relation between lattice types and exactness of filtrations.

In \cref{s.moduli-pahs} we define the moduli space $X_{[\mu]}$ of $p$-adic Hodge structures and its special subdiamonds/subvarieties. This is the  simplest of the moduli spaces to study, and this section is essentially a warm-up for \cref{s.moduli-admissible-pairs} where we study the admissible locus and its special subdiamonds/subvarieties. One of the key points in both sections is to show that special subvarieties are closed; this is not at all obvious when the special subvariety corresponds to a non-reductive group, and indeed this kind of statement is not true on the affine Grassmannian itself. To prove it we study the Hodge-Tate/Hodge loci of a tensor and use the Tannakian formalism to obtain reductions of structure group. The flexibility to make this kind of argument is one of our main motivations to interpret local Shimura varieties and their non-minuscule analogs as moduli spaces of objects in a Tannakian category; this is crucial also in the proof of our Ax-Lindemann theorem.

In \cref{s.moduli-infinite-level} we define the infinite level moduli spaces and prove the duality theorem. At this point most of the work is formal, but it is important to formulate here two distinct notions of special subvariety in the basic case (the issue being that in the non-minuscule case only one of the two period domains may be rigid analytic). We note that the duality theorem \cref{theorem.duality} and its proof by twisting bundles is well-known to experts, but does not seem to appear in the literature. 

In \cref{s.bi-an-ax--lind-theorem} we state the non-minuscule version of \cref{main.ax-lindemann} and prove it along with \cref{main.hodge-tate-image}.

Finally, in \cref{s.other-ax-lindemann}, we discuss how one can deduce results for global Shimura varieties and other spaces carrying natural admissible pairs from \cref{main.ax-lindemann}, and compare our bi-\emph{analytic} Ax-Lindemann theorem to potential bi-\emph{algebraic} Ax-Lindemann theorems in this context. In particular, in \cref{ss.vpahs} we introduce a notion of variations of $p$-adic Hodge structures and explain how it furnishes a natural setting for generalizations of our methods to other families, and in \cref{ss.products-of-Shimura-curves} we give a detailed discussion of these ideas in the case of products of Shimura curves over $\mathbb{Q}$ and a comparison with the bi-algebraic Ax-Lindemann theorem  established in \cite{chambert-loir-loeser:nonarchimedean-ax-lindemann}. 

\subsection{Notation}\label{ss.notation}
Throughout this paper $p > 0$ is a prime number. A non-archimedean field is a field that that is complete with respect to a non-archimedean absolute value such that the residue field is of characteristic $p$. A $p$-adic field is a discretely valued non-archimedean field of characteristic zero whose residue field is perfect, and is strict if the residue field is algebraically closed.

If $L$ is a non-archimedean field, a rigid analytic variety $X$ over $L$ is an adic space over $\Spa(L, \mc{O}_L)$ locally of the form $\Spa\+A = \Spa(A, A^\circ)$ for $A$ a topologically finite type $L$ algebra (i.e.\ $A$ is isomorphic to a quotient of a Tate algebra $L\langle T_1,\ldots, T_n \rangle$), and the (rigid) analytic topology of $X$ is the topology of the adic space. A rigid analytic subvariety is by definition locally closed. 

If $G$ is a linear algebraic group over a field $L$, we write $\Rep\+ G$ for the category of algebraic representations of $G$ on finite dimensional $L$ vector spaces, and $\omega_\std$ the standard fiber functor on $\Rep\+ G$ sending a representation to its underlying vector space. We refer to \cite{deligne-milne:tannakian} and \Iref{ss.tannakian} for definitions involving Tannakian categories.  Tensor functors between Tannakian categories over a field $k$ are assumed to be $k$-linear (and the field $k$ will often be left implicit).

\subsection{Acknowledgements} During the preparation of this work, Sean Howe was supported by the NSF through grants DMS-1704005 and DMS-2201112,  by an AMS-Simons travel grant, and as a visitor at the 2023 Hausdorff Trimester on The Arithmetic of the Langlands Program supported by the Deutsche Forschungsgemeinschaft (DFG, German Research Foundation) under Germany's Excellence Strategy – EXC-2047/1 – 390685813.  Christian Klevdal was supported by Samsung Science and Technology Foundation under Project Number SSTF-BA2001-02, and by NSF grant DMS-1840190. We thank Banff International Research Station for providing excellent working conditions during revisions and the preparation of \cref{s.other-ax-lindemann}. We thank Laurent Fargues, Ian Gleason, Pol van Hoften, Fran\c{c}ois Loeser, Jackson Morrow, Marc-Hubert Nicole, and Giovanni Rosso for helpful conversations, and Haohao Liu for pointing out an omission in the original proof of \cref{prop:hodge-exceptional-countable-union}. We thank an anonymous referee for a careful reading and detailed feedback.

\section{Preliminaries}\label{s.preliminaries}

In this work, we need to make geometric arguments that simultaneously involve perfectoid spaces, rigid analytic varieties, and more general objects like the $\mbb{B}^+_\dR$-affine Grassmannian. A natural category to work in is thus the category of locally spatial diamonds. In this section we recall the main points of this theory and describe some constructions in relative $p$-adic Hodge theory in this language -- for diamonds the basic reference is \cite{scholze:etale-cohomology-of-diamonds}. For relative $p$-adic Hodge theory we draw from \cite{fargues-scholze:geometrization,scholze:berkeley,scholze:p-adic-ht, kedlaya-liu:relative-foundations}.

We note that, in order to use the results in relative $p$-adic Hodge theory from \cite{scholze:p-adic-ht, kedlaya-liu:relative-foundations}, one needs to work on a site larger than the \'{e}tale site. One option is the quasi pro-\'{e}tale site, which is not the same as the pro-\'{e}tale site of \cite{scholze:p-adic-ht}, but is close enough to carry over the results (see \cite[Paragraph following Theorem 10.4.2]{scholze:berkeley}). Here instead we use the $v$-topology of \cite{scholze:etale-cohomology-of-diamonds}; it is not essential to our arguments, but we find this presentation to be the cleanest.

\subsection{Small $v$-sheaves and diamonds}
Let $\Perfd$ denote the category of perfectoid spaces and let $\Perf \subseteq \Perfd$ denote the category of perfectoid spaces over $\mbb{F}_p$. The coverings for the $v$-topology on $\Perfd$ are the families of maps $\{f_j: Y_i \rightarrow X\}_{i \in I}$ such that for any quasi-compact open $V \subseteq X$ there is a finite subset $J\subseteq I$ and quasi-compact opens $U_j \subseteq Y_j$ such that $V=\bigcup_{j \in J} f_j(U_j)$ (\cite[Definition 8.1]{scholze:etale-cohomology-of-diamonds}). The $v$-topology is subcanonical \cite[Theorem 8.7]{scholze:etale-cohomology-of-diamonds}, that is, every $X\in \Perfd$ represents a $v$-sheaf on $\Perfd$. We note that the definition and topological theory of diamonds will be formulated below using $\Perf$ --- later we will focus on diamonds over a $p$-adic field $L$, at which point we will use an alternative interpretation via the category $\Perfd_L$ of perfectoid spaces over $L$. 

A $v$-sheaf $Y$ on $\Perf$ is \emph{small} if it admits a surjection from a perfectoid space $X \in \Perf$ (\cite[Definition 12.1]{scholze:etale-cohomology-of-diamonds}). A small $v$-sheaf is a \emph{diamond} if it can be presented as a quotient $Y=X/R$ for $X \in \Perf$ and $R \subset X\times X$ a pro-\'{e}tale equivalence relation, and it suffices to verify this as a pro-\'{e}tale sheaf (\cite[Definition 11.1 and Proposition 11.9]{scholze:etale-cohomology-of-diamonds}).  

Any small $v$-sheaf $Y$ has an underlying topological space $|Y|$, functorial in $Y$, whose points are equivalence classes of geometric points, i.e.\ maps $\Spa(C,C^+) \rightarrow Y$ where $C$ is an algebraically closed perfectoid field (\cite[Proposition 12.7, Definition 12.8]{scholze:etale-cohomology-of-diamonds}). If $Y$ is a diamond, then any open subfunctor $U \subseteq Y$ is a diamond with $|U| \subseteq |Y|$ an open set, and this is a bijection between open sets in $|Y|$ and open subfunctors of $Y$ which we refer to as open subdiamonds (\cite[Proposition 11.15]{scholze:etale-cohomology-of-diamonds}). A diamond $Y$ is spatial if $Y$ is qcqs and $|Y|$ admits a basis of quasi-compact opens, and it is locally spatial if it admits an open cover by spatial subdiamonds (\cite[Definition 11.17]{scholze:etale-cohomology-of-diamonds}). If $Y$ is a spatial diamond, then $|Y|$ is spectral, thus if $Y$ is locally spatial then $|Y|$ is locally spectral (\cite[Proposition 11.18, 11.19]{scholze:etale-cohomology-of-diamonds}). Moreover, maps of locally spatial diamonds induce locally spectral and generalizing maps on topological spaces \cite[Proposition 11.19-(iv)]{scholze:etale-cohomology-of-diamonds}. The category of locally spatial diamonds has good stability properties: for example, it is closed under fiber products (\cite[Corollary 11.29]{scholze:etale-cohomology-of-diamonds}).

For $Y$ a small $v$-sheaf, $Y_v$ is the site whose underlying category is all small $v$-sheaves over $Y$ with covering families the jointly surjective families (\cite[Definition 14.1-(iii)]{scholze:etale-cohomology-of-diamonds}). For $Y$ a locally spatial diamond, inside of the $v$-site we also have the quasi pro-\'{e}tale, \'{e}tale  and finite \'{e}tale sites $Y_v \supset Y_\qpet \supset Y_\et \supset Y_{\fet}$ given by taking only objects $Y' \rightarrow Y$ that are quasi pro-\'{e}tale, \'{e}tale or finite \'{e}tale, respectively, over $Y$. The first two are defined in \cite[Definition 14.1]{scholze:etale-cohomology-of-diamonds}, while the finite \'{e}tale site does not appear to be explicitly defined in loc.\ cit.\ but is invoked, e.g., in \cite[Lemma 15.6]{scholze:etale-cohomology-of-diamonds} (the notation for the \emph{category} of finite \'{e}tale morphisms to $Y$ is introduced in \cite[Proposition 11.23]{scholze:etale-cohomology-of-diamonds}).  Note that for any $Y'\rightarrow Y \in Y_{\qpet}$,  $Y'$ is a locally spatial diamond by \cite[Corollary 11.28]{scholze:etale-cohomology-of-diamonds}. The analytic site $Y_\an \subset Y_\et$ is the site corresponding to the topological space $|Y|$ under the identification above of open subsets of $|Y|$ with open subdiamonds.

We note that, for $Y$ a locally spatial diamond, there is a good notion of a locally closed subdiamond: suppose $D \subseteq |Y|$ is locally closed and generalizing.  Then, we may form the $v$-sheaf $Z:=D\times_{|Y|} Y$, i.e.\ the subsheaf of $Y$ consisting of morphisms $f:Y' \rightarrow Y$ such that $|f|$ factors through $D$, and we claim $Z$ is a locally spatial diamond whose underlying topological space is $D$ (with the subspace topology) and we refer to such a $Z$ as a locally closed subdiamond of $Y$. The claim appears to be well-known to experts, and can be deduced from the results of \cite{scholze:etale-cohomology-of-diamonds} as follows:
Replacing $Y$ with an open subdiamond we may assume $D$ is closed. Then, since $D \rightarrow |Y|$ is quasi-compact, we find $Z \rightarrow Y$ is a quasi-compact injection of $v$-sheaves, so $Z$ is a locally spatial diamond by \cite[Proposition 11.20]{scholze:etale-cohomology-of-diamonds}. We can make this more explicit to see that $|Z|=D$: if we replace $Y$ with a quasi-compact open, then, choosing a pro-\'{e}tale surjection $P \rightarrow Y$ from a totally disconnected perfectoid, we have $P \times_Y Z = P \times_{|Y|} D = P \times_{|P|} (|P| \times_{|Y|}D).$  Since $|P| \times_{|Y|} D$ is a closed and generalizing subset of $|P|$, applying \cite[Lemma 7.3]{scholze:etale-cohomology-of-diamonds}, we find it is the pre-image of a closed subset of $\pi_0(P)$, thus $P \times_Y Z \rightarrow P$ is the associated closed immersion of totally disconnected affinoid perfectoids. It follows, in particular, that $|Z|=D$ with the subspace topology, as claimed.

\subsection{Diamonds attached to analytic adic spaces}

For $Y$ an analytic adic space over $\mbb{Z}_p$, the sheaf $Y^\diamond$ on $\Perf$ is defined by
\[ Y^{\diamond}(X) = \{ (X^\sharp,\iota), f: X^\sharp \rightarrow Y \}/\sim \]
where $X^\sharp$ is a perfectoid space over $\mbb{Z}_p$, $\iota: X^{\sharp, \flat} \cong X$, and $f$ is a map of adic spaces over $\mbb{Z}_p$. By \cite[Lemma 15.6]{scholze:etale-cohomology-of-diamonds}, the assignment $Y \rightarrow Y^\diamond$ is a functor from analytic adic spaces over $\mbb{Z}_p$ to diamonds and satisfies
\[ |Y|=|Y^\diamond|, Y_\et=Y^\diamond_\et, \textrm{ and } Y_{\fet}=Y^\diamond_{\fet}. \]

Fix now a non-archimedean field $L/\mbb{Q}_p$. We write $\Spd\+ L = \Spa(L, \mc{O}_L)^\diamond$. Let $\Perfd_L$ denote the category of perfectoid spaces over $L$. Then, we can and do identify a diamond over $\Spd\+ L$ with a small $v$-sheaf on $\Perfd_L$. In particular, this applies to $Y^\diamond$ for $Y$ a locally Noetherian adic space over $\Spa(L, \mc{O}_L)$.

Recall from \cref{ss.notation} that a rigid analytic variety over $L$ is an adic space over $\Spa\+ L$ locally of the form $\Spa\+A = \Spa(A, A^\circ)$ for $A$ a topologically finite type $L$ algebra (i.e.\ $A$ is isomorphic to a quotient of a Tate algebra $L\langle T_1,\ldots, T_n \rangle$). By \cite[Proposition 10.2.3]{scholze:berkeley}
the functor $Y \rightarrow Y^\diamond$ induces a fully faithful functor from seminormal rigid analytic varieties over $L$ to diamonds over $\Spd\+ L$. 
\begin{definition}\label{defn:rigid-analytic-diamond} 
Let $L$ be a non-archimedean field over $\QQ_p$, and let $S$ be a diamond over $\Spd\+ L$. 
\begin{enumerate}
    \item A \emph{rigid analytic point} of $S$ over $L$ is a morphism $s \colon \Spd\+ L' \to S$ over $L$ with $L'$ a finite extension of $L$.
    \item A \emph{rigid analytic diamond} over $\Spd\+ L$ is a diamond over $\Spd\+ L$ isomorphic to $Y^\diamond/\Spd\+ L$ for $Y$ a seminormal rigid analytic variety over $L$.
    \item A \emph{rigid analytic subdiamond} of $S$ over $\Spd\+ L$ is a locally closed subdiamond of $S$ that is a rigid analytic diamond over $\Spd\+ L$.  
\end{enumerate} 
\end{definition}
\begin{remark}\label{remark:rigid-analytic-dense-classical-points}
For $Y$ a rigid analytic variety over $L$, the image of a rigid analytic point of $Y^\diamond/\Spd\+L$ is a classical point of $Y$ under the identification $|Y^\diamond|=|Y|$, and via this construction the classical points of $Y$ are in bijection with equivalence classes of rigid analytic points of $Y^\diamond$. We sometimes refer to such a point also as a classical point of $Y^\diamond$. Note that the set of classical points is dense in $|Y|=|Y^\diamond|$ --- indeed, for $Y \supseteq \Spa(A,A^\circ)$ an open affinoid, the classical points in $\Spa(A,A^\circ)$ correspond to the maximal ideals of $A$ and thus there is always at least one. 
\end{remark}
We write $\mc{O}$ for the structure sheaf on $\Perfd_L$ defined by $\mc{O}(X)=\mc{O}_X(X)=\Hom_L(X,\mbb{A}^1_L)$. For $S/\Spd L$ a small $v$-sheaf, it extends uniquely to a structure sheaf $\mc{O}$ on $S_v$. By the full-faithfulness of $Y \rightarrow Y^\diamond$ invoked above, if $Y$ is a seminormal rigid analytic variety over $L$, then $\mc{O}_{Y_\et}$ is identified with ${\mc{O}}|_{Y^\diamond_\et}$ under the identification $Y_\et=Y_\et^\diamond$. 

\begin{lemma}\label{lemma.closed-rigid-analytic-subdiamonds-zariski-closed}
    Let $Y$ be a seminormal rigid analytic variety over $L$. Then, under the identification $|Y|=|Y^\diamond|$, the closed sets in $|Y^\diamond|$ underlying closed rigid analytic subdiamonds are exactly the Zariski closed sets in $|Y|$. 
\end{lemma}
\begin{proof}
    If $Z \subseteq Y$ is a Zariski closed subvariety, then $Z^\diamond \rightarrow Y^\diamond$ is the closed rigid analytic subdiamond of $Y$ corresponding to the closed set $|Z| \subseteq |Y|$ (it is rigid analytic in the sense of \cref{defn:rigid-analytic-diamond} because $Z^\diamond$ is equal to $(Z^{\mathrm{sn}})^\diamond$ for $\mathrm{sn}$ the seminormalization as in \cite[\S10.2]{scholze:berkeley}). Indeed, since any perfectoid space $P/\Spa\+ L$ is uniform, if a map $P \rightarrow Y$ factors through $|Z|$ at the level of topological spaces then it factors through $Z^\mathrm{red}$ and thus through $Z$. 

    Conversely, suppose that $Z^\diamond$ is a closed rigid analytic subdiamond of $Y^\diamond$, and let $Z \rightarrow Y$ be the associated map of seminormal rigid analytic varieties. Then $|Z| \rightarrow |Y|$ is identified with $|Z^\diamond| \rightarrow |Y^\diamond|$, so is a homeomorphism onto its closed image. It follows that $Z \rightarrow Y$ is quasi-compact. Evaluating on perfectoid points we see moreover that $Z \rightarrow Y$ is partially proper, and thus $Z \rightarrow Y$ is proper. Invoking \cite[Proposition 3 of \S9.6]{BoschGuntzerRemmert.NonArchimedeanAnalysis}, we conclude that its image is Zariski closed. 
\end{proof}

\subsection{Local systems and torsors}\label{ss.local-systems}
We first recall the category of $\Lambda$-local systems on a small $v$-sheaf, where $\Lambda$ is a complete Huber ring --- we will almost exclusively use the case $\Lambda = \QQ_p$ in this paper. First, the constant sheaf $\ul{\Lambda}$ on $\Perf$ is defined by 
\[ \ul{\Lambda}(X)=\Cont(|X|, \Lambda). \]
It satisfies more generally $\ul{\Lambda}(X)=\Cont(|X|,\Lambda)$ for any small $v$-sheaf $X$ (see \cite[\S3, Corollary 3.15]{mann-werner} for this result and a more general discussion). If $X$ is a small $v$-sheaf, a $\Lambda$-local system on $X$ is a sheaf of $\ul{\Lambda}$-modules on $X_v$ that is locally (in the $v$-topology) isomorphic to $\ul{\Lambda}^n$ for some $n$. We write $\Loc_{\Lambda}(X)$ for the full subcategory of $\Lambda$-local systems in the category of sheaves of $\ul{\Lambda}$-modules on $X_v$.

We now recall the construction of the sheaves $\mbb{B}^+_\dR$ and $\mbb{B}_\dR$ on $\Perfd_{\mbb{Q}_p}$. Suppose $\Spa(R,R^+) \in \Perfd_{\mbb{Q}_p}$ is affinoid perfectoid with tilt $(R^\flat, R^{+\flat})$. As in \cite[Lecture 15]{scholze:berkeley}, let $\varpi^\flat$ be a pseudo-uniformizer of $R^{+,\flat}$ such that $\varpi:=(\varpi^\flat)^\sharp$ satisfies $\varpi^p | p$. There is a surjective Fontaine map
\[ \theta: W(R^{+\flat})  \rightarrow R^+ \]
with principal kernel generated by a non-zero divisor, and the ring $\mbb{B}^+_\dR(R)$ is defined to be the $\mr{Ker}\+ \theta$-adic completion of $W(R^+)[1/[\varpi^\flat]]$. In particular, $\theta$ extends to natural surjection $\mbb{B}^+_\dR(R) \rightarrow R$. The ring $\mbb{B}_\dR(R)$ is obtained from $\mbb{B}^+_\dR(R)$ by inverting any generator of $\Ker\+ \theta$ and, as suggested by the notation, these rings are independent of the choice of $R^+$ --- indeed, this is already true of $W(R^{+,\flat})[1/[\varpi^\flat]]$, which can be equivalently characterized as the bounded Witt vectors of $R^\flat$, i.e.\ the sub-ring consisting of $\sum [r_i] p^i \in W(R^\flat)$ such that the set $\{r_i\}_{i \in \mathbb{Z}_{\geq 0}}$ is a bounded subset of the topological ring $R^\flat$. 

There is unique $v$-sheaf on $\Perfd_{\mbb{Q}_p}$ sending an affinoid perfectoid $\Spa(R,R^+)$ to $\mbb{B}^+_\dR(R)$ (resp. $\mbb{B}_\dR(R)$) --- this follows from \cite[Example 15.2.9-5.]{scholze:berkeley}. We write this sheaf also as $\mbb{B}^+_\dR$ (resp. $\mbb{B}_\dR$). By construction, there is a natural map $\theta: \mbb{B}^+_\dR \rightarrow \mc{O}$ that is a surjection in the $v$-topology. If $L$ is a $p$-adic field, then, after restricting $\mbb{B}^+_\dR$ to $\Perfd_L$, it is canonically a sheaf of algebras over the topological constant sheaf $\ul{L}$ such that $\theta$ is a map of $\ul{L}$-algebras. 

Suppose $X/\Spd\+ \mbb{Q}_p$ is a locally spatial diamond and $\mbb{B}=\mbb{B}^+_\dR, \mbb{B}_\dR,$ or $\mc{O}$. A sheaf of $\mbb{B}$-modules on $X_v$ is constant if it is isomorphic to $\mbb{B}^r$. A $\mbb{B}$-local system on $X$ is a locally constant $\mbb{B}$-module on $X_v$. We write $\Loc_{\mbb{B}}(X)$ for the category of $\mbb{B}$-local systems on $X_v$. 

\begin{proposition}[\cite{scholze:berkeley}]\label{prop.aff-perfectoid-loc-system-finite-proj-module} If $X=\Spa(A,A^+)$ is affinoid perfectoid over $\mbb{Q}_p$, then evaluation on $X$ is an equivalence between $\Loc_{\mbb{B}}(X)$ and finite projective $\mbb{B}(A)$-modules. 
\end{proposition}
\begin{proof}
This follows from the $v$-stack property of \cite[Corollary 17.1.9]{scholze:berkeley}
\end{proof}

Note that base change of coefficients along $\theta$ induces a functor 
\[ \Loc_{\mbb{B}^+_\dR}(X) \rightarrow \Loc_{\mc{O}}(X). \]

We recall an important lemma that will be used in later sections. It is extracted essentially from the proof of \cite[Lemma 19.1.5]{scholze:berkeley}. 
\begin{lemma}\label{lemma.zariski-closed-affine}
Let $L$ be a $p$-adic field and let $S=\Spa(R, R^+)$ be affinoid perfectoid over $\Spa\+ L$. Let $X$ be an affine scheme of finite type over $L$, and let $s \in X(\mbb{B}_\dR(R))$. Then the locus where $s \in X(\mbb{B}_\dR^+(R))$ is a Zariski closed (thus affinoid perfectoid) subdiamond of $S$. 
\end{lemma}
\begin{proof}
Let $Z_s$ be this locus. More specifically, $Z_s$ is the functor 
	\[ T \mapsto \{ \phi \colon T \to S \colon s \circ \mbb{B}_\dR(\phi) \in X(\mbb{B}_\dR^+(T)) \}. \]
Let $X \subseteq \mbb{A}^n_L$. Since $\mbb{B}_\dR^+ \to \mbb{B}_\dR$ is injective, $X(\mbb{B}_\dR^+) = X(\mbb{B}_\dR) \cap \mbb{A}^n_L(\mbb{B}_\dR^+)$, so we can consider $s \in \mbb{A}^n_L(\mbb{B}_\dR(R))$, and we reduce to the case where $X = \mbb{A}^n_L$.
But the locus $Z_s$ is the intersection of the loci where each coordinate of $s$ is in $\mbb{B}_\dR^+$, so we reduce to the case where $X = \mbb{A}^1_L$. This case is explicitly handled in \cite{scholze:berkeley} as follows: suppose $f \in \Fil^{-i}$ and fix a generator $\xi$ for $\Ker\+ \theta$. Then the locus where $f \in \Fil^{-i+1}$ is the vanishing set of $\theta( \xi^i f) \in R$, a Zariski closed subset, and we conclude that $Z_s$ is Zariski closed by induction. 

As $Z_s \subseteq S$ is Zariski closed, it is affinoid perfectoid \cite[Theorem 7.4, Remark 7.5]{bhatt-scholze:prismatic}.
\end{proof}

\begin{remark}\label{remark.strong-zariski}
In the proof we implicitly used the strong statement that, if $Z_2 \subseteq Z_1 \subseteq Z_0$ are affinoid perfectoid with $Z_1$ Zariski closed in $Z_0$ and $Z_2$ Zariski closed in $Z_1$, then $Z_2$ is Zariski closed in $Z_0$ (this is a strong statement because we need to know that an ideal $I$ in $\mc{O}(Z_1)$ such that $Z_2=V(I)$ can be generated by elements in the image of $\mc{O}(Z_0)$, so we are really using surjectivity of $\mc{O}(Z_0) \rightarrow \mc{O}(Z_1)$ --- cf. \cite[Remarks after Definition 5.7]{scholze:etale-cohomology-of-diamonds}). The full strength of this statement is not really necessary for our purposes, and one could argue instead as in \cite[Lemma 19.1.5]{scholze:berkeley} to deduce just that the locus is closed and represented by an affinoid perfectoid. \end{remark}

\begin{definition}[$G$-local systems and $G$-torsors]\label{defn:G-torsor}\hfill
\begin{enumerate}
    \item Let $G$ be a linear algebraic group over a $p$-adic field $L$, let $X$ be a locally spatial diamond over $\Spd\+ L$, and let $\mbb{B}$ be one of  $\mc{O}$, $\mbb{B}_\dR^+$, or $\mbb{B}_\dR$ (so that, in particular, $\mbb{B}$ is a sheaf of $L$-algebras on $X_v$).  A \emph{$G(\mbb{B})$-local system} on $X$ is an exact ($L$-linear) tensor functor $\Rep\+ G \to \Loc_{\mbb{B}}(X)$. The trivial $G(\mbb{B})$-local system is $\omega_\std \otimes_{L} \mbb{B}$. 
    \item Let $G$ be a linear algebraic group over $\QQ_p$, let $X$ be a locally spatial diamond, and let $\Lambda$ be a complete Huber $\QQ_p$-algebra. 
    \begin{enumerate}
        \item A \emph{$G(\Lambda)$-local system} on $X$ is an exact tensor functor $\Rep\+ G \to \Loc_{\Lambda}(X)$. The trivial $G(\Lambda)$-local system is $\omega_\std \otimes_{\mbb{Q}_p} \ul{\Lambda}$. 
        \item A $G(\Lambda)$-torsor on $X$ is a sheaf on $X_v$ with a (right) action of $G(\ul{\Lambda})$ that is locally on $X_v$ isomorphic to $G(\ul{\Lambda})$ with the action by right multiplication. 
    \end{enumerate}
\end{enumerate}
\end{definition}

\begin{remark}
With the obvious notion of morphisms, the groupoids of $G(\Lambda)$-local systems and $G(\Lambda)$-torsors are equivalent. We can explicitly describe quasi-inverse functors: if $\rho$ is a $G(\Lambda)$-local system then the sheaf of isomorphisms of $\omega_\std$ with $\rho$ is a $G(\Lambda)$-torsor; conversely, if $S$ is a $G(\Lambda)$-torsor, then the push-out $V \mapsto S \times^{G(\ul{\Lambda})} \ul{V}$ defines a $G(\Lambda)$-local system. While these notions are categorically equivalent, we find it concenptually useful to keep them distinct; $G(\Lambda)$-local systems have a more linear algebraic flavor which is useful for describing lattices/filtrations while keeping track of Tannakian structure groups in the construction of $p$-adic Hodge structure/admissible pairs in \cref{s.moduli-pahs} and \cref{s.moduli-admissible-pairs}. On the other hand, $G(\Lambda)$-torsors are more geometric (e.g. a $G(\QQ_p)$-torsor is naturally a diamond \cite[Lemma 10.13]{scholze:etale-cohomology-of-diamonds}) and arise naturally in the context of period maps in \cref{ss.period-maps}.
\end{remark}

\subsection{$\mc{O}$-local systems vs.\ vector bundles}\label{ss.vector-bundles-local-systems}
For $X$ a locally spatial diamond, above we defined an $\mc{O}$-local system on $X$ to be a locally free of finite rank $\mc{O}$-module on $X_v$. We define a vector bundle to be a locally free of finite rank $\mc{O}$-module on $X_\et$ and  write $\Vect(X)$ for the category of vector bundles. Pullback along the natural map of sites $v \colon X_v \to X_\et$ is an equivalence of $\Vect(X)$ with the full subcategory of $\Loc_{\mc{O}}(X)$ consisting of local systems that can be trivialized on an \'{e}tale cover; we often implicitly make this identification.

If $X$ is perfectoid then \cref{prop.aff-perfectoid-loc-system-finite-proj-module} implies any element of $\Loc_{\mc{O}}(X)$ can be trivialized even on an analytic covering, so $\Vect(X)=\Loc_{\mc{O}}(X)$. On the other hand, if $L/\mbb{Q}_p$ is a non-archimedean field and $X/L$ is a rigid analytic diamond, then vector bundles on $X$ are equivalent to vector bundles on the corresponding seminormal rigid analytic variety over $L$ and it can be shown that $\mc{O}$-local systems are equivalent to locally free modules over the completed structure sheaf $\hat{\mc{O}}$ on the pro-\'{e}tale site of the rigid analytic variety in the sense of \cite{scholze:p-adic-ht} --- the terminology of local systems vs. vector bundles that we use here is motivated by the Hodge-Tate theory in this setting; in particular, using Tate twists we see easily that there are many $\mc{O}$-local systems that are not vector bundles.

\subsection{Filtered vector bundles with connection and $\mbb{B}^+_{\dR}$-local systems}\label{ss.filtered-vector-bundles-scholzes-functor-M}
Suppose $L$ is a $p$-adic field and $S/L$ is a smooth rigid analytic variety. Following \cite[Definition 7.4]{scholze:p-adic-ht}, a filtered vector bundle with integrable connection on $S$ satisfying Griffiths transversality\footnote{This is called a filtered $\mc{O}_S$-module with integrable connection on $S$ in \cite[Definition 7.4]{scholze:p-adic-ht} --- note, in particular, that in loc. cit. Griffiths transversality is assumed even though it is not explicitly included in the terminology.} is a locally free $\mc{O}_S$-module $\mc{V}$ on $S$ (in the analytic or equivalently \'{e}tale topology \cite[Lemma 7.3]{scholze:p-adic-ht}) equipped with an integrable connection $\nabla$ and a separated and exhaustive decreasing filtration $\Fil^\bullet \mc{V}$ by locally direct summands such that $\Fil^\bullet \mc{V}$ satisfies Griffiths transversality for $\nabla$, i.e., for each $i$,
\[ \nabla(\Fil^i \mc{V}) \subseteq \Fil^{i-1} \otimes_{\mc{O}_{S}} \Omega_{S/L}. \]

We consider the pro-\'{e}tale site $S_\proet$ of $S$ as in \cite[Definition 3.9]{scholze:p-adic-ht}, and write $\nu: S_\proet \rightarrow S_\et$ for the natural map. In \cite[\S6]{scholze:p-adic-ht} and \cite{scholze.p-adic-ht-corrigendum}, there is defined a period sheaf $\mc{O}\mbb{B}_\dR$ that is an algebra over $\nu^{-1}\mc{O}_{S_\et}$ equipped with a filtration and a connection $\mc{O}\mbb{B}_\dR \rightarrow \mc{O}\mbb{B}_\dR \otimes _{\nu^{-1}\mc{O}_{S_\et}}\Omega_S$. It satisfies $(\mc{O}\mbb{B}_\dR)^{\nabla=0}=\mbb{B}_\dR$, compatibly with the filtrations on both sides, where here $\mbb{B}_\dR$ can be viewed as a sheaf on $S_\proet$ by evaluation on the basis of perfectoid objects as in \cite[Definition 4.3]{scholze:p-adic-ht}.

The construction of \cite[\S 7]{scholze:p-adic-ht} gives a fully faithful functor 
\[ (\mc{V}, \nabla, \Fil^\bullet \mc{V}) \rightarrow \mbb{M}(\mc{V}, \nabla, \Fil^\bullet \mc{V}) = \Fil^0(\nu^{-1}\mc{V} \otimes_{\nu^{-1}\mc{O}_{S_\et}} \mc{O} \mbb{B}_\dR)^{\nabla = 0} \]
from filtered vector bundles with integrable connection satisfying Griffiths transversality on $S$ to $\mbb{B}^+_\dR$-local systems on $S_\proet$. The category of $\mbb{B}^+_\dR$-local systems on $S_\proet$ is equivalent to the category of $\mbb{B}^+_\dR$-local systems on $S^\diamond$ in the sense considered in \cref{ss.local-systems}: indeed, using local torus coordinates we obtain an \'{e}tale cover $\{S_i\}_{i \in I} \rightarrow S$ and Galois profinite \'{e}tale coverings $\tilde{S}_i \rightarrow S_i$ such that each $\tilde{S}_i=\Spa(A_i,A^+_i)$ is affinoid perfectoid. By evaluation on $\tilde{S}_i$, the category of $\mbb{B}^+_\dR$-local systems on $\tilde{S}_i^\diamond$ is equivalent to category of finite projective $\mbb{B}^+_\dR(A_i)$-modules \cite[Cor. 17.1.9]{scholze:berkeley}, and the same holds for the category of $\mbb{B}^+_\dR$-local systems on $\tilde{S}_i$ viewed as an object of $S_\proet$. We then conclude by comparing the descent data on both sides (which in both cases is given by a continuous semi-linear action of $(\mbb{Z}_p(1))^n \rtimes \mbb{Z}_p^\times$ to descend from each $\tilde{S}_i$ to $S_i$ and then equivalent \'{e}tale descent data to go from $(S_i)_{i \in I}$ to $S$).

\section{Filtrations and lattices over diamonds}\label{s.filtrations-and-lattices}
In this section we define latticed $\mbb{B}^+_\dR$-local systems over locally spatial diamonds, and associate to any such a trace filtration by $\mc{O}$-modules on the $\mc{O}$-local system obtained by base change of coefficients. We also treat the symmetric case of bilatticed $\mbb{B}_\dR$-local systems. These definitions build on those used in \Iref{s.filtrations-and-lattices}, and we will use the language and definitions introduced there.  

Let $L$ be a $p$-adic field and let $X/\Spd\+L$ be a locally spatial diamond. In \cref{ss.filtrations-and-lattices}, we define latticed $\mbb{B}^+_\dR$-local systems and bilatticed $\mbb{B}_\dR$-local systems on $X$, and explain how to obtain associated filtered $\mc{O}$-local systems on $X$  by taking the trace filtration from one lattice on the special fiber of another. For $X$ a geometric point, these notions were studied in \Iref{ss.from-filtrations-to-lattices-and-back} and \Iref{ss.bilatticed-and-exactness}.

The main accomplishment in \Iref{s.filtrations-and-lattices} is \Iref{theorem.good-equivalence-bilatticed}, which relates exactness properties of this trace filtration construction over a geometric point to Schubert cells in affine Grassmannians. In \cref{ss.filtrations-lattices-G-torsors} we extend this to the relative setting: let $G/L$ be a connected linear algebraic group. For $X/ \Spd\+L$ a locally spatial diamond, we define a notion of latticed $G(\mbb{B}^+_\dR)$-local systems on $X$, extending the case of a geometric point treated in \Iref{def.filt-lattice-G-torsor}. For $[\mu]$ a conjugacy class of cocharacters of $G_{\overline{L}}$, we say a latticed $G(\mbb{B}^+_\dR)$-local system has type $[\mu]$ if it does after pullback to every geometric point: for $\Spd\+ C \rightarrow X$ a geometric point, a latticed $G(\mbb{B}_\dR^+)$-local system has a classifying double coset in
\[ G(\mbb{B}^+_\dR(C)) \backslash G(\mbb{B}_\dR(C)) / G(\mbb{B}^+_\dR(C)) \]
and it is of type $[\mu]$ if this is the coset containing $\mu(\xi)$ for $\xi$ any generator of $\Fil^1\mbb{B}^+_\dR(C)$ and some $\mu \in [\mu]$ (when $G$ is reductive, these cosets are exhaustive, but they are not in general).

\begin{example}\label{example:GL_n-type-of-latticed-torsor}
If $G=\GL_n$, a latticed $G(\mbb{B}^+_\dR)$-local system on $X$ is equivalent to a $\mbb{B}^+_\dR$-local system $M$ of rank $n$ equipped with a second $\mbb{B}^+_\dR$-local system of rank $n$, $\mc{L} \subseteq M \otimes_{\mbb{B}^+_\dR}\mbb{B}_\dR$, and it is of type $[\mu]$ for $\mu$ the diagonal cocharacter $(t^{i_1}, \ldots, t^{i_n})$ if and only if, after specializing to any geometric point, there is a $\mbb{B}_\dR^+$-basis $e_1, \ldots, e_n$ of $M$ such that $\xi^{i_1}e_1, \ldots, \xi^{i_n}e_n$ is a basis for $\mc{L}$. In other words, $[\mu]$ describes the elementary divisors of $\mc{L}$ relative to $M$ at any geometric point.  
\end{example}

In \cref{theorem.filtered-functor-iff-lattice-type} we estabish the relative version of \Iref{theorem.good-equivalence-bilatticed}, which says that a latticed $G(\mbb{B}_\dR^+)$-local system has type $[\mu]$ 
if and only if the associated filtration defines a filtered $G(\mc{O})$-local system of type $[\mu^{-1}]$ --- here to have a filtered $G(\mc{O})$-local system we need that the filtration on the $\mc{O}$-local systems attached to representations of $G$ is by local direct summands, and moreover that one obtains from this an exact functor from representations of $G$ to filtered $\mc{O}$-local systems. To be of type $[\mu^{-1}]$ means that, at every geometric point, the filtration is isomorphic to the standard decreasing filtration by weight spaces attached to $\mu^{-1}$. 

\begin{example}\label{example:GL_n-type-of-filtered-vb}
    If $G=\GL_n$, a filtered $G(\mc{O})$-local system on $X$ is equivalent to a $\mc{O}$-local system $\mc{V}$ of rank $n$ equipped with a (separated and exhaustive) filtration $\Fil^\bullet \mc{V}$ by $\mc{O}$-local systems such that the associated graded pieces $\gr^i\mc{V}=\Fil^i \mc{V}/\Fil^{i+1}\mc{V}$ are also $\mc{O}$-local systems. For $\mu$ the diagonal character $(t^{i_1}, \ldots, t^{i_n})$, the filtration is of type $[\mu]$ if and only if the rank of $\gr^{i}\mc{V}$ is equal to the multiplicity of $i$ in the multiset of weights of $\mu$, $\{i_1, \ldots, i_n\}$. 
\end{example}

To prove \cref{theorem.filtered-functor-iff-lattice-type}, we combine the case of a geometric point treated in \Iref{theorem.good-equivalence-bilatticed} with a supplementary argument from \cite{scholze:berkeley} to deduce that the trace filtration is actually by local direct summands. Note that, while the type of a filtered $G(\mc{O})$-local system always exists and is locally constant, even in the reductive case the type of a $G(\mbb{B}^+_\dR)$-local system can vary non-trivially over a connected base, in which case the trace filtration will not be by local direct summands. This is related to the closure relations between Schubert cells in affine Grassmannians, to be recalled in \cref{s.aff-grass}, and thus also to the case of a point treated in \Iref{theorem.good-equivalence-bilatticed} (recall the proof in that case depended on studying the behavior of types in extensions).

After setting up the general relative theory, we conclude in \cref{ss.bilatticed-equivalence-rigid-analytic} by specializing to smooth rigid analytic diamonds over a $p$-adic field, where we reinterpret Scholze's functor $\mbb{M}$ recalled earlier in \cref{ss.filtered-vector-bundles-scholzes-functor-M} as a fully faithful functor from filtered vector bundles with integrable connection satisfying Griffiths transversality to latticed $\mbb{B}^+_\dR$-local systems (\cref{theorem.smooth-scholze-filtrations-lattices}). 

\subsection{Filtrations and lattices}\label{ss.filtrations-and-lattices}
Let $X/\Spd\+ \mbb{Q}_p$ be a locally spatial diamond. A \emph{filtered $\mc{O}$-local system on $X$} is a pair $(\mc{V}, F^\bullet \mc{V})$ where $\mc{V}$ is an $\mc{O}$-local system on $X$ and $F^\bullet \mc{V}$ is a decreasing filtration of $\mc{V}$ by $\mc{O}$-local subsystems such that 
\begin{enumerate}
    \item The associated graded $\mr{gr}^\bullet \mc{V}:=\bigoplus_i F^i \mc{V}/F^{i+1}\mc{V}$ is an $\mc{O}$-local system
    \item For each geometric point $x:\Spd(C,C^+) \rightarrow X$, $x^* F^i \mc{V} = 0$ for $i \gg 0$ and $x^* F^i \mc{V} = x^*\mc{V}$ for $i \ll 0$.
\end{enumerate}
Note that since the dimension of each graded piece is locally constant, the second condition implies locally the filtration is finite and thus the first condition could be replaced by the requirement that each $F^i \mc{V}$ is a local direct summand. It follows also that the type of a filtered $\mc{O}$-local system is a locally constant function from $|X|$ to finite multisets of integers (equipped with the discrete topology), where the type of a filtration is as in \cref{example:GL_n-type-of-filtered-vb}.

If $V$ is a $\mbb{B}_\dR$-local system on $X$, a $\mbb{B}^+_\dR$-submodule $\mc{L} \subset V$ is a \emph{lattice on $V$} if $\mc{L}$ is a $\mbb{B}^+_\dR$-local system and $\mc{L} \otimes_{\mbb{B}^+_\dR} \mbb{B}_\dR=V$. A latticed $\mbb{B}^+_\dR$-local system on $X$ is a pair $(M, \mc{L})$ where $M$ is an $\mbb{B}^+_\dR$-local system on $X$ and $\mc{L} \subset M_{\mbb{B}_\dR}$ is a lattice. A bilatticed $\mbb{B}_\dR$-local system on $X$ is a triple $(V, \mc{L}_1, \mc{L}_2)$ where $V$ is a $\mbb{B}_\dR$-local system and each $\mc{L}_i$ is a lattice in $V$. There is an equivalence between latticed $\mbb{B}^+_\dR$-local systems on $X$ and bilatticed $\mbb{B}_\dR$-local systems given by 
\begin{equation}\label{eq.equiv-latt-bplusls-bilatt} M \otimes_{\mbb{B}^+_\dR} {\mbb{B}_\dR}=\mc{L} \otimes_{\mbb{B}^+_\dR} {\mbb{B}_\dR}  \leftrightarrow V, M \leftrightarrow \mc{L}_1, \mc{L} \leftrightarrow \mc{L}_2.\end{equation}
We write $\Loc_{\mbb{B}_\dR}^{\bilatticed}(X)$ for the category of bilatticed $\mbb{B}_\dR$-local systems on $X$.

If $X=\Spa(C,C^+)$ for $C/\mbb{Q}_p$ a perfectoid field, then $\mbb{B}^+_\dR(C)$ is a complete discrete valuation ring with residue field $C$ and fraction field $\mbb{B}_\dR(C)$. If $C$ is furthermore algebraically closed, then a $\mbb{B}_\dR$-local system on $X$ is a $\mbb{B}_\dR(C)$-vector space and, as in \Iref{ss.bilatticed-and-exactness}, we define the type of a bilatticed $\mbb{B}_\dR(C)$-vector space to be the multiset of integers giving the principal divisors of $\mc{L}_2$ relative to $\mc{L}_1$ (under the equivalence of \cref{eq.equiv-latt-bplusls-bilatt}, this agrees with the definition of \cref{example:GL_n-type-of-latticed-torsor}).  For a bilatticed $\mbb{B}_\dR$-local system over a general $X$, the type at geometric points defines a function on $|X|$. Unlike the type of a filtered $\mc{O}$-local system, however, the type of a bilatticed $\mbb{B}_\dR$-local system is not always locally constant. 

\begin{example}\label{example:non-constant-type-family}
Let $\mbb{C}_p$ be the completion of an algebraic closure of $\mbb{Q}_p$, and fix a generator $\xi$ for $\Ker\+ \theta_{\mbb{C}_p}$. On the closed unit disc $\mbb{D}_{\mbb{C}_p} = \Spa\+ \mbb{C}_p\langle z \rangle$ we may consider the bilatticed $\mbb{B}_\dR$-module
\[ (\mbb{B}_\dR^2, (\mbb{B}^+_\dR)^2,  \mbb{B}^+_\dR \cdot e_1 + \mbb{B}^+_\dR \cdot (e_2 + \frac{z}{\xi}e_1))  \]
where $z$ is the coordinate on $\mbb{D}_{\mbb{C}_p}$. 
For $z \neq 0$ the type is $\{1,-1\}$ (a $\mbb{B}_\dR^+$-basis of the second lattice is $\xi e_2, \xi^{-1}(e_1 + \frac{\xi}{z} e_2)$), while for $z=0$ the type is $\{0,0\}$. 

We note that we have abused notation slightly in our definition of the second lattice. One way to define $\mc{L}_z := \mbb{B}_\dR^+\cdot e_1 + \mbb{B}^+_\dR \cdot (e_2 + \frac{z}{\xi}e_1)$ explicitly is as follows: on the cover $\tilde{\mbb{D}}_{\mbb{C}_p} = \Spa(\mbb{C}_p\langle z^{1/p^\infty}\rangle)$, we have the element $\tilde{z} \in \mbb{B}_\dR^+(\tilde{\mbb{D}}_{\mbb{C}_p})$ that is the image of $[(z, z^{1/p}, z^{1/p^2},\ldots)] \in \mbb{A}_{\mr{inf}}(\tilde{\mbb{D}}_{\mbb{C}_p})$, and thus the lattice $\tilde{\mc{L}}_z := \mbb{B}^+_\dR \cdot e_1 + \mbb{B}^+_\dR \cdot (e_2 + \frac{\tilde{z}}{\xi}e_1)$. This descends uniquely to a lattice $\mc{L}_z$ over $\mbb{D}_{\mbb{C}_p}$ once we check that $\tilde{\mc{L}}_z$ is $\mr{Aut}(\tilde{\mbb{D}}_{\mbb{C}_p}/\mbb{D}_{\mbb{C}_p})$-invariant. If $\tau \in \mr{Aut}(\tilde{\mbb{D}}_{\mbb{C}_p}/\mbb{D}_{\mbb{C}_p})$ then $\tau(\tilde{z}) = \varepsilon\tilde{z}$ for some $\varepsilon \in \mbb{B}_\dR^+(\tilde{\mbb{D}}_{\mbb{C}_p})$ that is the image of $[(1, \zeta_p, \zeta_{p^2}, \ldots)] \in \mbb{A}_{\mr{inf}}(\tilde{\mbb{D}}_{\mbb{C}_p})$ for $(\zeta_{p^n})$ a sequence of compatible $p$-power roots of unity (not necessarily primitive). Now, note that       
\[ \tau(e_2 + \frac{\tilde{z}}{\xi}e_1) = e_2 + \frac{\varepsilon\tilde{z}}{\xi}e_1 = e_2 + \frac{\tilde{z}}{\xi}e_1 - \tilde{z}\left(\frac{1 - \varepsilon}{\xi}\right) e_1. \] 
Because $\theta_{\mbb{C}_p}(1 - \varepsilon) = 0$ and $\ker(\theta_{\mbb{C}_p}) = (\xi)$, $\frac{1 - \varepsilon}{\xi} \in \mbb{B}_\dR^+(\mbb{C}_p)$, and thus we find that $\tau$ sends the basis $e_1, e_2+\frac{\tilde{z}}{\xi} e_1$ of $\tilde{\mc{L}}_z$ to another basis of $\tilde{\mc{L}}_z$, as desired. 
\end{example}

If $(M, \mc{L})$ is a latticed $\mbb{B}_\dR^+$-local system, the trace filtration of $\mc{L}$ on $M$
\begin{align}\label{eqn:trace-filtration}
    \mr{trFil}_{\mc{L}}^{i} (M \otimes_{\mbb{B}^+_\dR} \mc{O}):&= \text{ Image of }(\Fil^i \mbb{B}_\dR \cdot \mc{L}) \cap M \text{ in } M \otimes_{\mbb{B}^+_\dR} \mc{O} \\
    &\cong ((\Fil^i \mbb{B}_\dR  \cdot \mc{L}) \cap M)/((\Fil^i\mbb{B}_\dR\cdot \mc{L}) \cap (\Fil^1\mbb{B}_\dR \cdot M)) \nonumber
\end{align}
is a filtration on $M \otimes_{\mbb{B}^+_\dR} \mc{O}$ by $\mc{O}$-submodules. As in the algebraic case treated in \Iref{ss.bilatticed-and-exactness}, the trace filtration induces two natural Bialynicki-Birula functors $\BB_1(V, \mc{L}_1,\mc{L}_2) = (\mc{L}_1 \otimes_{\mbb{B}^+_\dR} \mc{O}, \mr{trFil}_{\mc{L}_2})$ and $\BB_2(V, \mc{L}_1,\mc{L}_2) = (\mc{L}_2 \otimes_{\mbb{B}^+_\dR} \mc{O}, \mr{trFil}_{\mc{L}_1})$ from $\Loc_{\mbb{B}_\dR}^{\bilatticed}(X)$ to the category of $\mc{O}$-local systems equipped with a filtration by $\mc{O}$-submodules. The main difference is that here we are not guaranteed that the filtration is by local direct summands: indeed, \cref{example:non-constant-type-family} shows this need not occur, since the associated graded is not a local system because the ranks of its graded components differ between $\{0\}$ and its complement. However, the argument in the proof of \cite[Proposition 19.4.2]{scholze:berkeley} shows this is the only obstruction: that is, if the type of the lattice is constant then the filtration is by local direct summands so that in this case $\BB_i$ factors through the category of filtered $\mc{O}$-local systems. In this case it is again clear that $\BB_1$ inverts the type while $\BB_2$ preserves the type.

We have the following relative version of \Iref{lemma.associated-gradeds-twist}, giving a more precise relationship between the associated graded objects.
\begin{lemma}\label{lem:filtrations-bilatticed-relation} Suppose $(\mbb{V}, \mc{L}_1, \mc{L}_2) \in \Loc_{\mbb{B}_\dR}^{\bilatticed}(X)$, and let $\mc{V}_i=\mc{L}_i \otimes_{\mbb{B}^+_\dR} \mc{O}$. Then,
for each $i \in \mbb{Z}$
\[ \mr{gr}^i_{\mc{L}_2} \mc{V}_1 = \mr{gr}^{-i}_{\mc{L}_1} \mc{V}_2 (i) \]
where $(i)$ denotes a Tate twist (tensor with $\mr{gr}^i \mbb{B}_\dR = \mc{O} \otimes_{\mbb{Q}_p} \mbb{Q}_p(i)$). 
\end{lemma}
\begin{proof}
The map from right to left is induced by multiplication $\mbb{V} \otimes_{\mbb{B}^+_\dR} \Fil^{i} \mbb{B}_\dR \rightarrow \mbb{V}$. To see it is well-defined and an isomorphism, it suffices to verify it on each affinoid perfectoid $\Spa(R,R^+) \rightarrow X$ so that we are working everywhere with $\mbb{B}^+_\dR(R)$-modules. If we fix a generator $\xi$ for $\Ker\+ \theta_R$, then
\begin{align*} \mr{gr}^i_{\mc{L}_2} \mc{V}_1 &= (\xi^i \mc{L}_2) \cap \mc{L}_1 /\left( \xi^{i+1}\mc{L}_2 \cap \mc{L}_1 + \xi^i\mc{L}_2 \cap  t \mc{L}_1 \right)\\
\mr{gr}^{-i}_{\mc{L}_1} \mc{V}_2&= \mc{L}_2 \cap (\xi^{-i} \mc{L}_1) / \left( \xi \mc{L}_2  \cap \xi^{-i}\mc{L}_1  + \mc{L}_2 \cap  \xi^{-i+1} \mc{L}_1\right)
\end{align*}
and thus multiplying by $\xi^{-i}$ is an isomorphism from $\mr{gr}^i_{\mc{L}_2}\mc{V}_1$ to $\mr{gr}^{-i}_{\mc{L}_1}\mc{V}_2$. 
\end{proof}

\subsection{Filtrations and lattices on $G$-local systems}\label{ss.filtrations-lattices-G-torsors}
Fix a $p$-adic field $L$ and an algebraic closure $\overline{L}/L$. If $G/L$ is a connected linear-algebraic group, $X/\Spd\+ L$ is a locally spatial diamond, and $\mbb{B}=\mbb{B}^+_\dR, \mbb{B}_\dR,$ or $\mc{O}$, recall from \cref{defn:G-torsor} that a $G(\mbb{B})$-local system on $X$ is an exact tensor functor $\Rep\+ G \rightarrow \Loc_{\mbb{B}}(X)$. If $X=\Spa(R,R^+)$ is an affinoid perfectoid space over $L$, \cref{prop.aff-perfectoid-loc-system-finite-proj-module} implies this is equivalent to an \'etale $G$-bundle (in the usual sense of algebraic geometry) on $\Spec\+ \mbb{B}^+_\dR(R)$, $\Spec\+ \mbb{B}_\dR(R)$, or $\Spec\+ R$, respectively.

A filtered $G(\mc{O})$-local system is an exact tensor functor from $\Rep\+ G$ to filtered $\mc{O}$-local systems on $X$, where exact means it sends exact sequences to strict exact sequences. We define similarly a latticed $G(\mbb{B}^+_\dR)$-local system or bilatticed $G(\mbb{B}_\dR)$-local system. The latter two notions are equivalent, and can be interpreted as the data of two $G(\mbb{B}^+_\dR)$-local systems $\mc{G}_1$, $\mc{G}_2$ equipped with an isomorphism of their push-outs to $G(\mbb{B}_\dR)$ (i.e.\ the composition with the base-change functor from $\mbb{B}^+_\dR$-local systems to $\mbb{B}_\dR$-local systems).

If $\mc{G}^f$ is a filtered ${G}(\mc{O})$-local system on $X/\Spd\+ L$, then the type of $\mc{G}^f$ after pull-back to geometric points induces a locally constant function $[\mu]_{\mc{G}^f}$ from $|X|$ to the set of $\Gal(\overline{L}/L)$-orbits of conjugacy classes of cocharacters of $G_{\overline{L}}$. If $[\nu]$ is a conjugacy class defined over the reflex field $L([\nu])\subset \overline{L}$, we say a filtered $G(\mc{O})$-local system $\mc{G}^f$ on $X/\Spd(L[\nu])$ is of type $[\nu]$ if $[\mu]_{\mc{G}^f}$ is constant equal to $[\nu]$. 

Similarly, if $G$ is reductive and $\mc{G}^\bilatticed$ is a bilatticed $G(\mbb{B}_\dR)$-local system on $X/\Spd\+ L$, then the type of $\mc{G}^\bilatticed$ after pull-back to geometric points induces a function $[\mu]_{\mc{G}^\bilatticed}$ from $|X|$ to the set of $\Gal(\overline{L}/L)$-orbits of conjugacy classes of cocharacters of $G_{\overline{L}}$. If $G$ is not reductive this still holds once we add a point $\emptyset$ corresponding to bilatticed local systems with no type. As \cref{example:non-constant-type-family} illustrates for $G=\GL_2$, the type may not be locally constant. If $[\nu]$ is a conjugacy class defined over the reflex field $L([\nu])\subset \overline{L}$, we say a bilatticed $\mc{G}(\mc{O})$-local system $\mc{G}^\bilatticed$ on $X/\Spd(L[\nu])$ is \emph{good} of type $[\nu]$ if $[\mu]_{\mc{G}^\bilatticed}$ is constant equal to $[\nu]$. 

\begin{theorem}\label{theorem.filtered-functor-iff-lattice-type}
Let $G/L$ be a connected linear algebraic group, let $[\mu]$ be a conjugacy class of cocharacters of $G_{\overline{L}}$. Suppose $X/\Spd(L[\mu])$ is a diamond and $\mc{G}^\bilatticed$ is a bilatticed $G(\mbb{B}_\dR)$-local system on $X$. The following are equivalent. 
\begin{enumerate}
\item $\mc{G}^\bilatticed$ is good of type $[\mu]$. 
\item $\BB_1 \circ \mc{G}^\bilatticed$ is a filtered $G(\mc{O})$-local system of type $[\mu^{-1}]$
\item $\BB_2 \circ \mc{G}^\bilatticed$ is a filtered $G(\mc{O})$-local system of type $[\mu]$. 
\end{enumerate}
\end{theorem}
\begin{proof}
The case that $X=\Spa(C,C^+)$ for $C/L([\mu])$ algebraically closed perfectoid is \Iref{theorem.good-equivalence-bilatticed}. For the general case, it is now immediate that (2) or (3) implies (1) in general, since (1) can be checked at geometric points. The other direction follows also from the case of geometric points by working locally on an affinoid perfectoid $\Spa(R,R^+)$ and arguing as in \cite[Proposition 19.4.2]{scholze:berkeley} (the key statement to verify being that the associated graded pieces for any representation of $V$ are $\mc{O}$-local systems rather than just $\mc{O}$-modules). 
\end{proof}

\begin{remark}\label{remark.no-type-bruhat-order}
    When $G$ is reductive, \cite[Proposition 19.2.3]{scholze:berkeley} implies, by passing to a cover trivializing one of the lattices, that the type of a bilatticed $G(\mbb{B}_\dR)$-local system is lower semi-continuous for the Bruhat order (this is illustrated concretely in \cref{example:non-constant-type-family}). For general $G$, \cref{theorem.bdr-aff-grass-properties}-(3) will imply that, if we fix the type of the push-out to the Levi quotient, then the bilatticed $G(\mbb{B}_\dR)$-local systems with a type are a closed set, so that the type $\emptyset$ corresponds to an open point. However, in general it is not true that the type is lower semi-continuous if we declare $\emptyset$ to be a maximum: in \cref{example.non-closed-special-in-schubert-cell}, we exhibit a family that has a type over an open locus but has no type over the complementary closed locus (the push-out of this family to the Levi quotient gives a specialization of types similar to \cref{example:non-constant-type-family}). 
\end{remark}

\subsection{An equivalence for rigid analytic diamonds}\label{ss.bilatticed-equivalence-rigid-analytic}

Suppose $L$ is a $p$-adic field and $X/L$ is a smooth rigid analytic variety. Recall from \cref{ss.filtered-vector-bundles-scholzes-functor-M} that there is a fully faithful functor $\mbb{M}$ from filtered vector bundles with integrable connection satisfying Griffiths transversality on $X$ to  $\Loc_{\mbb{B}^+_\dR}(X^\diamond)$, defined by taking
\[ \mbb{M}(\mc{V}, \nabla, \Fil) = \Fil^0(\nu^{-1}\mc{V} \otimes_{\nu^{-1}\mc{O}_{X_\et}} \mc{O}\mbb{B}_\dR)^{\nabla = 0} \]
on $X_\proet$ and then extending to the $v$-site of $X^\diamond$ as explained in \cref{ss.filtered-vector-bundles-scholzes-functor-M}.
We will now reinterpret this functor in light of the relation between latticed $\mbb{B}^+_\dR$-local systems and filtrations. 

To that end, for $(\mc{V}, \nabla)$ a filtered vector bundle with integrable connection on $X$, let 
\[ \mbb{M}_0(\mc{V}, \nabla):=\mbb{M}(\mc{V},\nabla,\Fil_{\triv}^\bullet), \]
which is the extension to the $v$-site of 
\[ (\nu^{-1}\mc{V} \otimes_{\nu^{-1}\mc{O}_{X_\et}} \Fil^0\mc{O}\mbb{B}_\dR)^{\nabla = 0} = (\nu^{-1}\mc{V} \otimes_{\nu^{-1}\mc{O}_{X_\et}} \mc{O}\mbb{B}^+_\dR)^{\nabla = 0}\]
where $\Fil_{\triv}^\bullet$ is the trivial filtration $\Fil^i_{\triv}{\mc{V}}= \mc{V}$ for $i \leq 0$ and $\Fil^i_{\triv}{\mc{V}}=0$ for $i>0$. Note that \emph{a priori} we have only an inclusion 
\[ (\nu^{-1}\mc{V} \otimes_{\nu^{-1}\mc{O}_{X_\et}} \mc{O}\mbb{B}_\dR^+)^{\nabla = 0} \subseteq (\nu^{-1}\mc{V} \otimes_{\nu^{-1}\mc{O}_{X_\et}} \Fil^0\mc{O}\mbb{B}_\dR)^{\nabla = 0},\]
but in fact this is an equality: looking pro-\'etale locally reduces to the case when $\mc{V} \otimes_{\nu^{-1}\mc{O}_{X_\et}} \mc{O}\mbb{B}_\dR^+$ is a trivial $\mc{O}\mbb{B}_\dR^+$-local system, and we conclude from the equality $(\mc{O}\mbb{B}_\dR^+)^{\nabla = 0} = \mbb{B}_\dR^+=\Fil^0 \mbb{B}_\dR=\Fil^0\mc{O}\mbb{B}_\dR^{\nabla = 0} $, which can be seen easily from \cite[Proposition 6.10]{scholze:p-adic-ht}. In \cite[\S7]{scholze:p-adic-ht}, it is shown that $\mbb{M}(\mc{V}, \nabla, \Fil^\bullet)$ is a $\mbb{B}^+_\dR$-lattice in $\mbb{M}_0(\mc{V}, \nabla) \otimes_{\mbb{B}^+_\dR} \mbb{B}_\dR$. Recall we write $v \colon X^\diamond_v \to X^\diamond_\et$ for the projection to the \'etale site.

\begin{theorem}[\cite{scholze:p-adic-ht}]\label{theorem.smooth-scholze-filtrations-lattices} Let  $X/\Spa\+ L$ be a smooth rigid analytic variety.
\begin{enumerate}
    \item The assignment $(\mc{V}, \nabla) \mapsto \mbb{M}_0(\mc{V}, \nabla)$ is a fully faithful functor from vector bundles with integrable connection on $X$ to $\mbb{B}^+_\dR$-local systems on $X^\diamond$, and $v_*(\mbb{M}_0(\mc{V}, \nabla) \otimes_{\mbb{B}^+_\dR} \mc{O})= \mc{V}$ under the identification $(X^\diamond_\et, \mathcal{O}|_{X^\diamond_\et})=(X_\et, \mathcal{O}_{X_\et}).$ 
    \item The assignment 
    \[ (\mc{V}, \nabla, \Fil^\bullet) \rightarrow (\mbb{M}_0(\mc{V}, \nabla), \mc{L}_{\Fil^\bullet}), \mc{L}_{\Fil^\bullet}:=\mbb{M}(\mc{V}, \nabla, \Fil^\bullet) \subset \mbb{M}_0(\mc{V}, \nabla) \otimes_{\mbb{B}^+_\dR} \mbb{B}_\dR \]
    is a fully faithful tensor functor from filtered vector bundles with integrable connection satisyfing Griffiths transversality on $X$ to latticed  $\mbb{B}^+_\dR$-local systems on $X^\diamond$. Under the identification $v_\ast(\mbb{M}_0(\mc{V}, \nabla) \otimes_{\mbb{B}^+_\dR} \mc{O})=\mc{V}$ of (1), the pushforward along $v$ of the trace filtration for $\mc{L}_{\Fil^\bullet}$ is $\Fil^\bullet$ and $\mc{L}_{\Fil^\bullet}$ is uniquely determined by this property. 
    \item  A latticed $\mbb{B}_\dR^+$-local system $(M, \mc{L})$ on $X^\diamond$ is in the essential image of the functor from (2) if and only if $M$ is in the essential image of the functor from (1) and the trace filtration (\ref{eqn:trace-filtration}) of $\mc{L}$ on $M \otimes_{\mbb{B}_\dR^+} \mc{O}$ is by local direct summands (equivalently, the type of the lattice $\mc{L}$ is constant).
\end{enumerate}
\end{theorem}
\begin{proof}
The functor $\mbb{M}_0$ is the composition of the fully faithful functor $\mbb{M}$ with the fully faithful embedding, via the trivial filtration, of vector bundles with integrable connection inside filtered vector bundles with integrable connection satisfying Griffiths transversality. This gives (1).  

We use notation and definitions of \cite{scholze:p-adic-ht}, in particular we have the structure sheaf $\hat{\mc{O}}_X$ on $X_\proet$ and the projection $\nu \colon X_\proet \to X_\et$ satisfies $\nu_\ast \hat{\mc{O}}_X = \mc{O}_{X_\et}$. If $\lambda \colon X_v \to X_\proet$ is induced by $\{U_i\}_{i \in I} \in X_\proet \mapsto U = \varprojlim_i U_i^\diamond$, then $\lambda^\ast$ induces an equivalence $\Loc_{\mbb{B}_\dR^+}(X_\proet) \cong \Loc_{\mbb{B}_\dR^+}(X^\diamond)$ as in \cref{ss.filtered-vector-bundles-scholzes-functor-M}, and the functor $\mbb{M}$ factors through $\Loc_{\mbb{B}_\dR^+}(X_\proet)$ so (2) and (3) will follow from the results of \cite{scholze:p-adic-ht}. We describe this now, working everywhere below on $X_\proet$. 

We now prove (2). Let $(\mc{V}, \nabla, \Fil^\bullet)$ be a filtered vector bundle with integrable connection satisfying Griffiths transversality on $X$, and let $(M, \mc{L}) =  (\mbb{M}_0(\mc{V}, \nabla), \mc{L}_{\Fil^\bullet})$. By \cite[Proposition 7.9]{scholze:p-adic-ht} and \cref{lem:filtrations-bilatticed-relation} (applied to $(M \otimes_{\mbb{B}_\dR^+} \mbb{B}_\dR, M, \mc{L})$), there are natural isomorphisms 
    \[ \Fil^i\mc{V} \otimes_{\nu^{-1}\mc{O}_{X_\et}} \hat{\mc{O}}_X \cong (\Fil^i\mc{L} \cap M)/ (\Fil^i \mc{L} \cap \Fil^1 M), \]
(where we write $\Fil^i\mc{L} = \Fil^i\mbb{B}_\dR \cdot \mc{L}$ and likewise for $\Fil^1M$) and $\mc{L}$ is uniquely determined by this property. Moreover, $\Fil^\bullet \mc{V}$ can be recovered functorially from $(M,\mc{L})$ by applying $\nu_\ast$, and the connection $\nabla$ can be recovered as in \cite[Lemma 7.8]{scholze:p-adic-ht}. Full faithfulness of the functor in (2) follows since 
    \[ \Hom((\mc{V}_1, \nabla_1, \Fil^\bullet_1),(\mc{V}_2, \nabla_2, \Fil^\bullet_2)) = \Fil^0\ul{\Hom}((\mc{V}_1, \nabla_1, \Fil^\bullet_1),(\mc{V}_2, \nabla_2, \Fil^\bullet_2))^{\nabla = 0}. \]

We now prove (3). Given a latticed $\mbb{B}_\dR^+$-local system $(M, \mc{L}) = (\mbb{M}_0(\mc{V}, \nabla), \mc{L})$ as in (3), we have already seen that $(\mc{V}, \nabla)$ can be recovered from $(M, \mc{L})$. If the type of the lattice is constant, the trace filtration on $\mc{V}$
    \[ \trFil_{\mc{L}}^i \mc{V} = \nu_\ast( \Fil^i\mc{L} \cap M/ \Fil^i \mc{L} \cap \Fil^1 M) \]
is a filtration by local direct summands \cite[Proposition 3.4.3]{caraiani-scholze:cohomology-compact-shimura}.  It remains to show that $\nabla$ satisfies Griffiths transversality with respect to $\trFil_{\mc{L}}^\bullet$, and that the natural map $\mc{L} \hookrightarrow M_{\mbb{B}_\dR} = (\nu^{-1}\mc{V} \otimes_{\nu^{-1}\mc{O}_{X_\et}} \mc{O}\mbb{B}_\dR)^{\nabla = 0}$ identifies $\mc{L}$ with $\Fil^0(\nu^{-1}\mc{V} \otimes_{\nu^{-1}\mc{O}_{X_\et}} \mc{O}\mbb{B}_\dR)^{\nabla = 0} = \mbb{M}(\mc{V}, \nabla, \trFil_{\mc{L}}^\bullet)$. If the type of $\mc{L}$ is $r_1 \geq r_2 \geq \cdots \geq r_n$, we prove the statement by induction on $r_1 - r_n$. Recall from our conventions on types of lattices and filtrations (c.f.\ \cref{example:GL_n-type-of-latticed-torsor}, \cref{example:GL_n-type-of-filtered-vb}) this means that, after restricting to any geometric point, $M$ has a basis $e_1, \ldots, e_n$ such that $\xi^{r_1}e_1, \ldots, \xi^{r_n}e_n$ is a basis of $\mc{L}$; consequently, the images of $e_i, r_i \leq j$ form a basis of $\trFil_{\mc{L}}^{-j}\mc{V}$ and the rank of $\gr^{-j}_\mc{L}\mc{V} = \trFil_{\mc{L}}^{-j}\mc{V}/\trFil_{\mc{L}}^{-j + 1}\mc{V}$ is the multiplicity of $j$ in the multiset $\{r_1,\ldots, r_n\}$. 

In case $r_1 = r_n = r$, the result is trivial since $\mc{L} = \Fil^r \mbb{B}_\dR \cdot M$ and $\trFil_\mc{L}$ is a single step filtration in degree $-r$. In general, we can twist $\mc{L}$ so that the type is $r_1 \geq \cdots \geq r_n = 0$, and consequently $\mc{L} \subseteq M$. For ease of notation, in what follows we write $\trFil_{\mc{L}}^\bullet$ for $\trFil_{\mc{L}}^\bullet \mc{V}$, unadorned tensor products are over $\mc{O}_{X_\et}$ except for $(-) \otimes \hat{\mc{O}}_X$, which is the composition of $\nu^{-1}$ and $(-)\otimes_{\nu^{-1}\mc{O}_{X_\et}} \hat{\mc{O}}_X$. 

Define a sublattice $\mc{L}' = \mc{L} \cap \Fil^1 M$. From construction, $\trFil_\mc{L'}^0 = 0$ and $\trFil_{\mc{L}'}^i = \trFil_{\mc{L}}^i$ for $i < 0$. Consequently by the inductive hypothesis for $\mc{L}'$, we have $\nabla(\trFil^i_{\mc{L}}) \subset \trFil^{i-1}_{\mc{L}} \otimes \Omega^1_X$ for $i < 0$, so we only need to check that $\nabla(\trFil^0_{\mc{L}}) \subset \trFil^{-1}_{\mc{L}} \otimes \Omega^1_X$. Noticing that $\mc{L}/\mc{L}' \cong \trFil^0_{\mc{L}} \otimes \hat{\mc{O}}_X$ and $\mc{L}'/(\mc{L}'\cap \Fil^2 M) \cong \trFil^{-1}_{\mc{L}} \otimes \hat{\mc{O}}_X(1)$ yields the following commutative diagram with exact rows and columns:
\[\begin{tikzcd}[column sep = small, row sep = small]
	& 0 & 0 & 0 \\
	0 & {\trFil^{-1}_{\mc{L}} \otimes \hat{\mc{O}}_X(1)} & {\mc{L}/(\mc{L} \cap \Fil^2M)} & {\trFil^0_{\mc{L}} \otimes \hat{\mc{O}}_X} & 0 \\
	0 & {\mc{V} \otimes \hat{\mc{O}}_X(1)} & {M/\Fil^2 M} & {\mc{V} \otimes \hat{\mc{O}}_X} & 0
	\arrow[from=2-1, to=2-2]
	\arrow[from=3-1, to=3-2]
	\arrow[from=3-2, to=3-3]
	\arrow[from=2-2, to=2-3]
	\arrow[from=2-3, to=2-4]
	\arrow[from=3-3, to=3-4]
	\arrow[from=2-4, to=2-5]
	\arrow[from=3-4, to=3-5]
	\arrow[from=2-4, to=3-4]
	\arrow[from=2-2, to=3-2]
	\arrow[from=2-3, to=3-3]
	\arrow[from=1-2, to=2-2]
	\arrow[from=1-3, to=2-3]
	\arrow[from=1-4, to=2-4]
\end{tikzcd}\]
Applying the long exact sequence for $\nu_\ast$ and the isomorphism $R^1\nu_\ast(\hat{\mc{O}}_X(1)) \cong \Omega^1_X$ of \cite[Corollary 6.19]{scholze:p-adic-ht} gives the following  commutative diagram \cite[Lemma 7.8]{scholze:p-adic-ht}
\[\begin{tikzcd}
	{\trFil^0_{\mc{L}}} & {\trFil^{-1}_{\mc{L}} \otimes \Omega^1_X } \\
	{\mc{V}} & {\mc{V} \otimes \Omega^1_X} & {}
	\arrow[from=1-1, to=1-2]
	\arrow["{-\nabla}", from=2-1, to=2-2]
	\arrow[hook, from=1-1, to=2-1]
	\arrow[hook, from=1-2, to=2-2]
\end{tikzcd}\]
Consequently $\nabla$ satisfies Griffiths transversality for $\trFil_{\mc{L}}$. The equality $\mc{L} = \Fil^0(\mc{V} \otimes \mc{O}\mbb{B}_\dR)^{\nabla = 0}$ follows from the proof of \cite[Proposition 7.9]{scholze:p-adic-ht} as both lattices sit in a short exact sequence $\mc{L}' \to \bullet \to \trFil^0_{\mc{L}} \otimes \hat{\mc{O}}_X$ that is compatible with $\Fil^1 M \to M \to \mc{V} \otimes \hat{\mc{O}}_X$. 

\end{proof}

\begin{example}\label{example.lattices-over-rig-an-point}
When $X=\Spd\+ L$, the connection is always trivial and Scholze's functor $(V, \Fil^\bullet) \mapsto \mbb{M}(V,\Fil^\bullet)$ can be interpreted as mapping filtered $L$-vector spaces to semi-linear representations of $\Gal(\overline{L}/L)$ on finite free $\mbb{B}^+_\dR(C)$-modules (by evaluating on $\Spd\+ C$ and using the descent data to obtain the Galois action). The essential image consists of those modules $M$ such that $(M \otimes \mbb{B}_\dR(C))^{\Gal(\overline{L}/L)}$ is an $L$-vector space of the same rank as $M$. In this case, $M_0$ is $(M \otimes \mbb{B}_\dR(C))^{\Gal(\overline{L}/L)} \otimes \mbb{B}^+_\dR(C)$, so that $(M,M_0)$ can be recovered from $M$ alone. \end{example}

\section{$G$-bundles on the relative Fargues-Fontaine curve}\label{s.G-bundles-ff}

Recall \cite[Definition II.1.15]{fargues-scholze:geometrization} that for any $S \in \Perf_{\overline{\mbb{F}}_p}$ we have an associated adic relative Fargues-Fontaine curve $\FF_S$ over $\Spa(\mbb{Q}_p)$. For $G/\mathbb{Q}_p$ a connected reductive group, the moduli stack of $G$-bundles on relative Fargues-Fontaine curves and the theory of modifications of $G$-bundles on relative Fargues-Fontaine curves are both studied in \cite{fargues-scholze:geometrization}. In this section, we recall some basic definitions and results from that study and extend them to general connected linear algebraic groups $G/\mathbb{Q}_p$. The main result is \cref{theorem.bun-g-b-structure}, which says that the stratum in the moduli stack of $G$-bundles associated to a basic $G$-isocrystal is an open substack isomorphic to the classifying space for the automorphism group of the $G$-isocrystal. 

\subsection{$G-$Isocrystals}\label{ss.G-isocrystals}
Let $G/\mbb{Q}_p$ be a connected linear-algebraic group. We recall some results on $G$-isocrystals and, in particular, the definition of a basic isocrystal in the non-reductive setting; this material is treated in more detail in \Iref{sss.isocrystals-G-structure}. 

We write $\sigma$ for the lift of Frobenius on $\Qpbreve$, and we write $\Isoc$ for the category of isocrystals, i.e. the category of finite dimensional $\Qpbreve$-vector spaces equipped with a $\sigma$-linear isomorphism. Recall from \cite{kottwitz:isocrystals} that a $G$-isocrystal is an exact tensor functor $\Rep\+G \rightarrow \Isoc$. For any $b \in G(\Qpbreve)$, we obtain an associated $G$-isocrystal by sending $(V,\rho)$ to $V \otimes_{\mbb{Q}_p} \Qpbreve$ with the $\sigma$-linear automorphism $\rho(b) \circ (\mathrm{Id}_V \otimes \sigma)$. This construction induces a bijection between $B(G)$, the set of $\sigma$-conjugacy classes in $G(\Qpbreve)$, and the set of isomorphism classes of $G$-isocrystals. 

An element $b \in G(\Qpbreve)$ (or the associated class $[b] \in B(G)$) is called basic if the the slope morphism $\nu_b$ (defined in \cite[\S 4.2]{kottwitz:isocrystals}) factors through the center of $G$. Note that, in this case, the slope morphism is independent of the choice of representative and is defined over $\mathbb{Q}_p$ (as in \cite[\S5.1]{kottwitz:isocrystals}). The center of $G$ is the kernel of the adjoint representation \cite[Theorem 13.4]{humphreys:linear-algebraic-groups}, and so $b$ is basic if and only if the isocrystal associated to $b$ by the adjoint action of $G$ on $\Lie\+ G$ has trivial slope grading or, equivalently, is the trivial isocrystal. We write $B(G)_\mr{basic} \subseteq B(G)$ for the set of basic classes. 

\begin{remark}\label{remark.Levi-decomp-BG-BM-and-basic}
If we fix a Levi decomposition $G=MU$, then, as explained in \Iref{sss.isocrystals-G-structure}, the inclusion $M \rightarrow G$ induces an identification $B(G)=B(M)$ with inverse induced by the quotient $G \rightarrow M$. A basic element of $B(G)$ gives rise to a basic element of $B(M)$, but a basic element of $B(M)$ may not give rise to a basic element of $B(G)$: in general, there can exist a basic $M$-isocrystal whose push-out to a $G$-isocrystal along the inclusion $M \rightarrow G$ is not basic (see \Iref{example.basic-levi-non-basic-G}). In particular, such a push-out then also gives a non-basic $G$-isocrystal whose push-out to an $M$-isocrystal along the quotient $G \rightarrow M$ is basic.  
\end{remark}

For $b \in G(\Qpbreve)$ we write $G_b$ for the functor on $\mathbb{Q}_p$-algebras sending $R$ to the automorphism group of the exact tensor functor from $\Rep\+ G$ to $\Isoc \otimes_{\mbb{Q}_p} R$ obtained from the above tensor functor associated to $b$ by extension of scalars. By \cite[Proposition 1.12]{rapoport-zink:period-spaces}, $G_b$ is a linear algebraic group over $\mathbb{Q}_p$ and $G_b(\mbb{Q}_p)$ is identified with the $\sigma$-centralizer of $b$ in $G(\Qpbreve)$. When $b$ is basic, $G_b$ is an inner form of $G$, and in this case we have a canonical identification 
\begin{equation}\label{eq.Qpbreve-points-of-G-and-twist} G(\breve{\mbb{Q}}_p) = G_b(\breve{\mbb{Q}}_p). \end{equation}

\subsection{$\Bun_G$ for $G$ non-reductive}\label{ss.BunG}
Let $G/\mathbb{Q}_p$ be a connected linear algebraic group. For $S \in \Perf_{\overline{\mbb{F}}_p}$, we write $\Bun_G$ for the prestack on $\Perf_{\overline{\mbb{F}}_p}$ sending $S$ to the groupoid of $G$-bundles on $\FF_S$ (interpreted via any of the equivalent formulations in \cite[Definition/Proposition III.1.1]{fargues-scholze:geometrization} --- note that the reference in loc.\ cit.\ to \cite{scholze:berkeley} does not use that $G$ is reductive). Then $\Bun_G$ is a small $v$-stack --- note that the argument for $G$ reductive in \cite[following III.1.1, Proposition III.I.3]{fargues-scholze:geometrization} does not depend on $G$ being reductive. 

As in the reductive case treated in \cite{fargues-scholze:geometrization}, for any $b \in G(\breve{\mbb{Q}}_p)$, there is a natural
\[ \mc{E}_b: \ast=\Spd\+ \overline{\mbb{F}}_p \rightarrow \Bun_G\]
that sends a perfectoid space $S/\overline{\mbb{F}}_p$ to the $G$-bundle $\mc{E}_b$ on $\FF_S$ obtained by descending the trivial $G$-bundle on the cover $Y_S$ of $\FF_S$ (as defined in \cite[II.1.1]{fargues-scholze:geometrization}) via the Frobenius action $b\cdot\sigma$. From the Tannakian perspective, this is given for any $S$ by composing the associated isocrystal $\Rep\+ G \rightarrow \Isoc$ with the natural map $\Isoc \rightarrow \Vect(\FF_S)$ constructed by using the $\sigma$-linear isomorphism of an isocrystal to descend the associated trivial vector bundle on $Y_S$. In particular, the automorphisms of the $G$-isocrystal associated to $b$ induce automorphisms of $\mc{E}_b$. 

In fact, for any $S \in \Perf_{\overline{\mathbb{F}}_p}$, arguing as in \cite[Prop III.4.2]{fargues-scholze:geometrization} we find that the sheaf of sections of $G_b \times_{\mathbb{Q}_p} \FF_S$ on the \'{e}tale site of $\FF_S$ maps naturally into the sheaf of automorphisms of $\mc{E}_b$ on the \'{e}tale site of $\FF_S$, and that this is an isomorphism for $b$ basic. In particular, since the global sections of $G_b \times_{\mathbb{Q}_p} \FF_S$ over $\FF_S$ are naturally identified with the continuous maps from the underlying topological space $|S|$ of $S$ to $G_b(\mathbb{Q}_p)$, i.e. the $S$-points of the $v$-sheaf of groups associated to the topological group $G_b(\mathbb{Q}_p)$, we find that $\mc{E}_b$ induces a faithful map 
\begin{equation}\label{eq.class-stack-map-to-BunG} [\ast / G_b(\mbb{Q}_p) ] \rightarrow \Bun_G \end{equation}
that is full when $b$ is basic. When $b$ is not basic, the functor of  \cref{eq.class-stack-map-to-BunG} may not be full since the $v$-sheaf associated to $G_b(\mbb{Q}_p)$ may not be the entire $v$-sheaf of automorphisms of $\mc{E}_b$ (a necessary, but not sufficient, condition for this is that the image of $b$ in the quotient of $G$ by its unipotent radical be basic). 

\begin{remark}We do not discuss the structure of the full automorphism group in detail outside the basic case here since we will not need it anywhere, but see \cite[III.5.1]{fargues-scholze:geometrization} for the reductive case. In the general case, we note that the reductive case treated in \cite[III.5.1]{fargues-scholze:geometrization} allows one to understand the quotient of the automorphism group by a group of unipotent automorphisms associated to the unipotent radical $U$ of $G$; to understand the structure of this unipotent part of the automorphism group, one can argue as in the proof of \cref{theorem.bun-g-b-structure} below to see that it is an iterated extension of the Banach-Colmez spaces of global sections associated to the isocrystals obtained by the adjoint action of $b$ on the quotients in the derived series for $\Lie\+ U$. \end{remark}

For $[b] \in B(G)_\mr{basic}$, we write $\Bun_G^{[b]}$ for the substack of $\Bun_G$ whose objects are those $G$-bundles that, after pullback to any geometric point, are isomorphic to $\mc{E}_b$. Note that the map of \cref{eq.class-stack-map-to-BunG} factors through $\Bun_G^{[b]}$. 

\begin{theorem}\label{theorem.bun-g-b-structure} 
Let $G/\mathbb{Q}_p$ be a connected linear algebraic group, and let $M$ be the reductive quotient of $G$ by its unipotent radical $U$. 
For each $[b] \in B(G)_{\mr{basic}}$, $\Bun_G^{[b]}$ is an open substack of $\Bun_G$, and, for any $b \in [b]$, the map of  \cref{eq.class-stack-map-to-BunG} is an isomorphism 
\begin{equation}\label{eq.isomorphism-b-loc-classifying-stack} [\ast/G_b(\mbb{Q}_p)] \cong \Bun_G^{[b]}. \end{equation}
Moreover, for $[b_M]$ the image of $[b]$ in $B(M)$, 
\begin{equation}\label{eq.bun-G-over-Levi-fiber-prod} \Bun_G^{[b]}\cong\Bun_G \times_{\Bun_M} \Bun_M^{[b_M]}, \end{equation}
where the map from $\Bun_G$ to $\Bun_M$ in the formation of the fiber product on the right-hand side is given by push-out of bundles from $G$ to $M$. 
\end{theorem}
\begin{proof}
    When $G$ is reductive, \cref{eq.bun-G-over-Levi-fiber-prod} is trivial and the rest of the result is part of \cite[Theorem III.4.5]{fargues-scholze:geometrization}. 

    We consider now the case of general $G$. We fix a Levi decomposition $G=MU$. By \cref{remark.Levi-decomp-BG-BM-and-basic}, the inclusion $M \rightarrow G$ induces a bijection $B(M)=B(G)$ so we may assume $b \in M(\Qpbreve) \subseteq G(\Qpbreve)$; we will write $b$ if we view it as an element of $G(\Qpbreve)$ and $b_M$ if we view it as an element of $M(\Qpbreve)$. 
    
    We first show simultaneously that \cref{eq.bun-G-over-Levi-fiber-prod} holds and that $\mc{E}_b: \ast \rightarrow \Bun_G^{[b]}$ is a $v$-cover. Since the map $\Bun_G^{[b]} \rightarrow \Bun_G \times_{\Bun_M} \Bun_M^{[b_M]}$ is fully faithful, to establish \cref{eq.bun-G-over-Levi-fiber-prod} it suffices to show that this map is also essentially surjective. Thus, to establish both claims, it suffices to show the composition of $\mc{E}_b$ with the map of \cref{eq.bun-G-over-Levi-fiber-prod} is a $v$-cover. In other words, we need to show that any object of $\Bun_G \times_{\Bun_M} \Bun_M^{[b_M]}$ is $v$-locally isomorphic to $\mc{E}_b$.
    
    To that end, suppose $S\in \Perf_{\overline{\mathbb{F}}_p}$ and $\mc{E}$ in $\Bun_G \times_{\Bun_M} \Bun_M^{[b_M]}(S)$. Since $b_M$ is basic, by the reductive case of the theorem we find that we can replace $S$ with a $v$-cover and choose an isomorphism $ \mc{E}_{b_M}=\mc{E} \times^G M$, where $\mc{E} \times^G M$ is the push-out of $\mc{E}$ to $M$ along the quotient map $G \rightarrow M$. We then consider the sheaf $\mc{T}$ on the \'{e}tale site of $\FF_S$ parameterizing isomorphisms $\mc{E}_b \xrightarrow{\sim} \mc{E}$ that induce this fixed isomorphism $\mc{E}_{b_M}=\mc{E} \times^G M$ after push-out to $M$ and composition with the canonical isomorphism $\mc{E}_{b_M}=\mc{E}_b \times^G M$. The sheaf $\mc{T}$ is a torsor for the \'{e}tale sheaf of groups $\mc{U}$ on $\FF_S$ parameterizing automorphisms of the $G$-torsor $\mc{E}_b$ that induce the identity on $\mc{E}_{b_M}=\mc{E}_b \times^G M$. This group $\mc{U}$ is an inner form of $U$ over $\FF_S$, and the isomorphism class of the torsor $\mc{T}$ is given by a class in $H^1_\et(\FF_S, \mc{U})$. For $S' \rightarrow S$ a map in $\Perf_{\overline{\mathbb{F}}_p}$, the formation of $\mc{U}$ and this class in $H^1_\et(\FF_S, \mc{U})$ is compatible with pullback along $\FF_{S'} \rightarrow \FF_S$, and it suffices to show that this class can be annihilated after passing to a $v$-cover. Indeed, $\mc{T}_{\FF_{S'}}$ is the trivial torsor if and only if it admits a section, which gives an isomorphism $\mc{E}_b \cong \mc{E}$ over $\FF_{S'}$. 

    To obtain this annihilation, we first consider the case that the unipotent radical $U$ of $G$ is abelian. Then $U\cong \Lie\+ U$ as a vector group, compatibly with the conjugation action of $G$, and $\mc{U}$ is the vector bundle on $\FF_S$ associated to the isocrystal given by the adjoint action of $b$ on $\Lie\+ U$. Because $b$ is basic, this is the trivial vector bundle  $\Lie\+ U \otimes_{\mbb{Q}_p} \mc{O}_{\FF_S}$. The result then follows from the equivalence $H^1_\et(\FF_S, \mc{O}_S)=H^1(\FF_S, \mc{O}_S)$ (which follows from \cite[Theorem 8.2.22-(c)]{kedlaya-liu:relative-foundations}) and the pro-\'{e}tale local annihilation of classes in $H^1(\FF_S, \mc{O}_S)$ \cite[Proposition II.2.5-(ii)]{fargues-scholze:geometrization}. For the case of general $U$, we can filter $U$ by its derived series and argue by induction to reduce to this same pro-\'{e}tale local annihilation. Thus we have established the isomorphism \cref{eq.bun-G-over-Levi-fiber-prod} and that $\mc{E}_b$ is a $v$-cover of $\Bun_G^{[b]}$. 
    
    Since the reductive case of the entire theorem has already been established, $\Bun_M^{[b]}$ is an open substack of $\Bun_M$, so the isomorphism \cref{eq.bun-G-over-Levi-fiber-prod} implies, in particular, that $\Bun_G^{[b]}$ is an open substack of $\Bun_G$.  Finally, because the map \cref{eq.class-stack-map-to-BunG} is fully faithful for $b$ basic, since we have also shown $\mc{E}_b$ is a $v$-cover of $\Bun_G^{[b]}$, 
    we find that \cref{eq.class-stack-map-to-BunG} induces an equivalence with $\Bun_G^{[b]}$, establishing \cref{eq.isomorphism-b-loc-classifying-stack}
    \end{proof}

We also record here a computation relating basic strata of $\Bun_G$ for different groups that will be used in the duality theorem for infinite level spaces, \cref{theorem.duality}.  Let $b \in G(\Qpbreve)$ be basic. Arguing exactly as in \cite[III.4.1]{fargues-scholze:geometrization}, we find that  $\mc{E} \mapsto \ul{\Isom}_{\FF_S}(\mc{E}_b, \mc{E})$ induces an equivalence $F_b:\Bun_G \cong \Bun_{G_b}$. Under the identification of \cref{eq.Qpbreve-points-of-G-and-twist}, in the following for any $a \in G(\Qpbreve)$ we write $a_G$ for $a$ viewed as an element of $G(\Qpbreve)$ and $a_{G_b}$ for $a$ viewed as an element of $G_b(\Qpbreve)$. 

\begin{lemma}\label{lemma.BunG-twisting}
There are natural isomorphisms in $\Bun_{G_b}(\Spd\+ \overline{\mbb{F}}_p)$
\[ F_b(\mc{E}_{b_G}) \cong \mc{E}_{1_{G_b}} \textrm{ and } F_b(\mc{E}_{1_G}) \cong\mc{E}_{b^{-1}_{G_b}}.\]
Note here that $\mc{E}_{1_{G_b}}$ is the trivial $G_b$-bundle and $\mc{E}_{1_G}$ is the trivial $G$-bundle. 
\end{lemma}
\begin{proof}
This first isomorphism follows since $F_b(\mc{E}_{b_G})=\ul{\Isom}(\mc{E}_b, \mc{E}_b)$ and we have already seen that $G_b$ is the automorphism group of $\mc{E}_b$ (cf. \cite[Corollary III.4.3 of Proposition III.4.2]{fargues-scholze:geometrization}). The second isomorphism is obtained by computing with the trivialization on the covers $Y_S$ of $\FF_S$ as in the proof of \cite[Proposition III.4.2]{fargues-scholze:geometrization}.
\end{proof}

\subsection{Modifications}\label{ss.modifications}
We recall, from the Tannakian perspective, the notion of modifications of $G$-bundles as in \cite[III.3]{fargues-scholze:geometrization}. Suppose $S/\Spd\+ \breve{\mbb{Q}}_p$, with $S$ in $\Perf$. Note that $\Spd\+\breve{\mbb{Q}}_p = \Spd\+ \mbb{Q}_p \times \Spd\+ \overline{\mbb{F}}_p$, so in particular $S/\Spd\+ \overline{\mbb{F}}_p$ and the discussions of the previous section apply. The structure map to $\Spd\+ \breve{\mbb{Q}}_p$ gives an untilt $\infty: S^\sharp \hookrightarrow \FF_S$ which is the inclusion of a closed Cartier divisor \cite[Proposition 11.3.1]{scholze:berkeley} and the completion of $\mc{O}_{\FF_S}$ at $\mc{I}_\infty$ (the ideal sheaf of $\infty$) is the sheaf $\mbb{B}_\dR^+$. Thus completion along $\mc{I}_\infty$ gives rise to an exact tensor functor $\Vect(\FF_S) \to \Loc_{\mbb{B}_\dR^+}(S)$ (the category of $\mbb{B}_\dR^+$-local systems on $S$), which should be thought of geometrically as restricting to the formal neighborhood of $\infty \in \FF_S$.   Thus, given $\mc{E} \in \Bun_G(S)$ we obtain a $G(\mbb{B}^+_\dR)$-local system\footnote{Recall that we take the Tannakian viewpoint, i.e.\ that $\mc{E}$ is an exact tensor functor $\Rep\+ G \to \Vect(\mc{O}_{\FF_S})$ so that $\mc{E}|_{\mbb{B}_\dR^+(S)}$ is the exact tensor functor obtained by composition with $\Vect(\FF_S) \to \Loc_{\mbb{B}_\dR^+}(S)$.} $\mc{E}|_{\mbb{B}^+_\dR(S)}$ and by push-out a $G(\mbb{B}_\dR)$-local system $\mc{E}|_{\mbb{B}_\dR}$ on $S$. Moreover, given a lattice $\mc{L}$ on this $G(\mbb{B}_\dR)$-local system, we may form the modification $\mc{E}_\mc{L}$ of $\mc{E}$ by $\mc{L}$ to obtain a new object in $\Bun_G(S)$: by the Tannakian formalism, we reduce to modifications of vector bundles on $\FF_S$ by $\mbb{B}_\dR^+$-lattices, which can be constructed via the Beauville-Laszlo gluing \cite[Lemma 5.2.9]{scholze:berkeley} (using the equivalence of vector bundles and finite projective modules on affinoids $U \subseteq \FF_S$ \cite[Theorem 5.2.8]{scholze:berkeley}). This construction is functorial and interpolates the construction of modifications of bundles on the algebraic Fargues-Fontaine curve at geometric points as described in \Iref{ss.modifications}.

\section{The $\mbb{B}^+_\dR$-affine Grassmannian}\label{s.aff-grass}

Let $L$ be a $p$-adic field and fix an algebraic closure $\overline{L}/L$. For $G/L$ a connected linear algebraic group, the $\mbb{B}^+_\dR$-affine Grassmannian parameterizes lattices on the trivial $G(\mbb{B}_\dR)$-local system. It plays a central role in our study, similar to the role of flag varieties in classical Hodge theory. The case of $G$ reductive was handled in \cite{scholze:berkeley}; here we develop the basic structure for a general connected linear algebraic group. 

Below we use the notion of types for filtrations and lattices that was defined in \cref{s.filtrations-and-lattices} (see also \Iref{s.filtrations-and-lattices}). This is also made explicit for the moduli spaces that appear in the following definition immediately afterwards in \cref{remark.bdr-gr-filtration}. 

\begin{definition}
 Let $G/L$ be a connected linear algebraic group. 
\begin{enumerate}
    \item The $\mbb{B}^+_\dR$-affine Grassmannian $\Gr_G / \Spd\+ L$ is the functor on $\Perfd_L$
    \begin{align*} X & \mapsto \{\textrm{Lattices on the trivial $G(\mbb{B}_\dR)$-local system over $X$}\}/\sim \\
    & = \{\textrm{Lattices on the trivial $G(\mbb{B}^+_\dR)$-local system over $X$}\}/\sim
    \end{align*}
    \item The flag variety $\Fl_G / \Spa\+ L$ 
    is the rigid analytic variety over $\Spa\+ L$ whose functor of points on adic spaces over $\Spa\+ L$ is
    \[ X \mapsto \{\textrm{Filtrations on the trivial $G$-bundle on $X_\et$}\}/\sim. \]
    In particular, $\Fl_G^\diamond / \Spd\+ L$ is identified with the functor on $\Perfd_L$ 
    \[ X \mapsto \{\textrm{Filtrations on the trivial $G(\mc{O})$-local system over $X$}\}/\sim. \]
\end{enumerate}
Let  $[\mu]$ be a conjugacy class of cocharacters of $G_{\overline{L}}$ with reflex field $L([\mu])$.   
\begin{enumerate}
    \item The affine Schubert cell $\Gr_{[\mu]}$ is the subfunctor of $\Gr_{G} \times_{\Spd\+ L} \Spd\+ L([\mu])$ parameterizing lattices of type $[\mu]$ on the trivial $G(\mbb{B}^+_\dR)$-local system. 
    \item The flag variety $\Fl_{[\mu]}$ is the rigid analytic subvariety of $\Fl_{G} \times_{\Spa\+ L} \Spa\+ L([\mu])$ parameterizing filtrations of type $[\mu]$.
    \item The Bialynicki-Birula map $\BB: \Gr_{[\mu]} \rightarrow \Fl_{[\mu^{-1}]}^\diamond$ sends a lattice on the trivial $G(\mbb{B}^+_\dR)$-local system to its trace filtration on the trivial $G(\mc{O})$-local system (this is well defined by \cref{theorem.filtered-functor-iff-lattice-type}).
\end{enumerate}
\end{definition}

\begin{remark}\label{remark.bdr-gr-filtration}
For the convenience of the reader, we recall our conventions on types of filtrations and lattices (\cref{s.filtrations-and-lattices} and \Iref{s.filtrations-and-lattices}) by giving an explicit description of the $\Spa(C, C^+)$-points of the above functors for $C/L([\mu])$ an algebraically closed non-archimedean extension. To that end, we first note that, since we are given an inclusion $L([\mu]) \rightarrow C$, it makes sense to talk about $[\mu]_C$, the cocharacters of $G_C$ conjugate by an element of $G(C)$ to an element in $[\mu]$, and $[\mu]_{\mbb{B}^+_\dR(C)}$, the cocharacters of $G_{\mbb{B}^+_\dR(C)}$ conjugate by an element of $G(\mbb{B}^+_\dR(C))$ to an element in $[\mu]$: indeed, the inclusion $L([\mu]) \rightarrow C$ can be extended to $\overline L \rightarrow C$ and then, by Hensel's lemma, to $\overline{L} \rightarrow \mbb{B}^+_\dR(C)$, and the definition of the reflex field $L([\mu])$ is such that the resulting conjugacy class depends only on the given map $L([\mu]) \rightarrow C$ and not on the choice of its extension. Now, the data of $\mc{L} \in \Gr_G(C) = \Gr_G(C, C^+)$ is equivalent to giving, for each $V$ in $\Rep\+ G$, a $\mbb{B}_\dR^+(C)$-lattice $\mc{L}(V) \subseteq V_{\mbb{B}_\dR(C)}$ that is functorial and exact in $V$ and compatible with tensor products. Then, $\mc{L} \in \Gr_{[\mu]}(C)$ if there is a cocharacter $\mu \in [\mu]_{\mbb{B}^+_\dR(C)}$ such that for one (equivalently, any) generator $\xi$ of $\Fil^1 \mbb{B}_\dR^+(C)$  and each $V \in \Rep\+ G$, $\mc{L}(V) = \mu(\xi)V_{\mbb{B}_\dR^+(C)}$. The data of $\Fil \in \Fl_G(C) = \Fl_G(C,C^+)$ is equivalent to a decreasing, complete and exhaustive filtration $\Fil^\bullet V_C$ that is functorial and exact in $V$ and compatible with tensor products; $\Fil \in \Fl_{[\mu]}(C)$ if there exists a cocharacter $\mu \in [\mu]_C$ such that for each $V \in \Rep\+ G$, $\Fil^i V_C = \bigoplus_{j \geq i} V_C[j]$ (with $V_C[j]$ the eigenspace where $\mu(t)$ acts by $t^j$ for $t \in \mbb{G}_m(C)$).
\end{remark}

Note that $\Gr_G$ admits a natural action of $G(\mbb{B}_\dR)$ by change of trivialization. The action of $G(\mbb{B}^+_\dR) \leq G(\mbb{B}_\dR)$ preserves $\Gr_{[\mu]}$ (it is the stabilizer of the trivial lattice) and the map $\BB$ is equivariant for the action of $G(\mbb{B}^+_\dR)$ on $\Fl_{[\mu^{-1}]}^\diamond$ by projection to $G(\mc{O})$ and change of trivialization. The following proposition compiles some alternate interpretations and formal consequences of the definitions and results that have already been stated: 
\begin{proposition}\label{prop.alternative-characterizations}\hfill
\begin{enumerate}
    \item If $X=\Spa(A,A^+) \in \Perfd_L$ is affinoid perfectoid, then 
\[ \Gr_{G}(X)= \{\textrm{G-torsors on $\Spec\+ \mbb{B}^+_\dR(A)$ with a trivialization over $\Spec\+ \mbb{B}_\dR(A)$} \}/\sim \]
\item The action of $G(\mbb{B}_\dR)$ on the trivial lattice induces an identification of $\Gr_G$ with the \'{e}tale (or $v$) sheafification of 
\[ X \mapsto G(\mbb{B}_\dR(X))/G(\mbb{B}^+_\dR(X)). \]
In particular for any algebraically closed perfectoid field $C/L$ 
\[ \Gr_G(\Spa(C,C^+))= G(\mbb{B}_\dR(C))/G(\mbb{B}^+_\dR(C)) \]
and if $C/L([\mu])$, this identifies 
\[ \Gr_{[\mu]}(\Spa(C,C^+)) = G(\mbb{B}^+_\dR(C)) \xi^\mu G(\mbb{B}^+_\dR(C))/G(\mbb{B}^+_\dR(C)) \]
for any generator $\xi$ of $\Ker\+ \theta_C$ and $\mu \in [\mu]$.
\item $\Gr_{G}$ is a $v$-sheaf, and for any diamond $X$, 
   \begin{align*} \Gr_{G}(X) & = \{\textrm{Lattices on the trivial $G(\mbb{B}_\dR)$-local system over $X_v$}\}/\sim \\
    & = \{\textrm{Lattices on the trivial $G(\mbb{B}^+_\dR)$-local system over $X_v$}\}/\sim
    \end{align*}
    
\item If $S/\Spa\+ L$ is a seminormal rigid analytic variety, 
\[ \Hom(S^\diamond, \Fl_G^\diamond) = \{ \textrm{Filtrations on the trivial $G$-bundle on $S_\et$} \} /\sim. \]
\item $\Fl_{[\mu]} \subseteq \Fl_G \times_L \Spa\+ L([\mu])$ is a clopen subvariety.  
\end{enumerate}
\end{proposition}
\begin{proof} 
We omit the proof --- most of the statements follow immediately from the definitions and \cref{prop.aff-perfectoid-loc-system-finite-proj-module}. For (3), note that a lattice on the trivial local system has no non-trivial automorphisms so the $v$-sheaf property follows from $v$-descent for $\mbb{B}^+_\dR$-local systems. 
\end{proof}

The following theorem is the main result of the section. 

\begin{theorem}\label{theorem.bdr-aff-grass-properties} Let $G/L$ be a connected linear algebraic group with unipotent radical $U$ and reductive quotient $M=G/U$. Let $[\mu]$ be a conjugacy class of cocharacters of $G_{\overline{L}}$ with reflex field $L([\mu])$; we write $[\mu]_M$ for the induced conjugacy class of cocharacters of $M_{\overline{L}}$. Then,
\begin{enumerate}
    \item $\Gr_G$ is separated and partially proper over $\Spd\+ L$ and an increasing union of closed locally spatial subdiamonds. 
    \item $\Gr_{[\mu]} \subseteq \Gr_{G} \times_{\Spd\+ L} \Spd\+ L([\mu])$ is a locally closed locally spatial subdiamond, partially proper over $\Spd\+ L([\mu])$. If $[\mu]_M$ is minuscule (i.e.\ the adjoint action of any representative $\mu$ on $\Lie\+ M$ has weights in $\{-1,0,1\}$),  it is a closed subdiamond. 
    \item $\Gr_{[\mu]}$ is closed in $\Gr_G \times_{\Gr_{M}} \Gr_{[\mu]_M}$. 
    \item If $S/\Spa\+ L([\mu])$ is a seminormal rigid analytic variety, then
    \begin{equation} \Gr_G \times_{\Gr_M} \Gr_{[\mu]_M} (S^\diamond) = \Gr_{[\mu]}(S^\diamond), \label{eq.rigid-points-bounded-affgr}\end{equation}
    and $\BB$ induces a bijection between this set and the subset of $\Fl_{[\mu^{-1}]}^\diamond(S^\diamond)=\Fl_{[\mu^{-1}]}(S)$ consisting of filtrations satisfying Griffiths transversality over the smooth locus $S^{\mr{sm}}$ (for the trivial connection on the trivial $G$-bundle on $S^{\mr{sm}}$). 
    \item If the weights of $[\mu]$ for the adjoint action on $\Lie\+ G$ are $\leq 1$, then $\BB$ is an isomorphism $\Gr_{[\mu]}= \Fl_{[\mu^{-1}]}^\diamond$. This is the only case where $\BB$ is an isomorphism; more precisely, in all other cases the universal filtration on $\Fl_{[\mu^{-1}]}$ does not satisfy Griffiths transversality for the trivial connection over any open (so that even a local section to $\BB$ would violate (4)). 
\end{enumerate}    
\end{theorem}

For $G$ reductive (3) is trivial, (1), (2), and (5) were established in \cite{scholze:berkeley}, and we will establish (4) as a consequence of the results of \cite[\S7]{scholze:p-adic-ht} plus a short argument using resolution of singularities and the $v$-sheaf property to reduce to the smooth case. 

For $G$ reductive, part (1) is established by reducing to the case of $G=\GL_n$ by using a faithful representation to realize $\Gr_G$ as a closed sub-functor of $\Gr_{\GL_n}$. When $G$ is not reductive it is no longer the case that a faithful representation realizes $\Gr_G$ as a closed sub-functor; however, we can adapt the argument to show that $\Gr_{G}$ is locally closed in $\Gr_{\GL_n}$ if we choose $\rho$ so that it realizes $G$ as an observable subgroup, and this suffices for (1) --- the details are given in \cref{ss.aff-grass-locally-spatial} below. 

For the rest of the theorem, the key statement is (3), which allows us to bootstrap from results in the reductive case to prove (2), (4) and (5) in general. The main difficulty is that, for $G$ non-reductive, the $\Gr_{[\mu]}$ no longer give a decomposition of $\Gr_G$, as illustrated already in the first non-trivial case: 
\begin{example}\label{example.schubert-cells-additive-group} Suppose $G=\mbb{G}_a$. Then the only co-character is the trivial co-character $\mu_\triv$, and $\Gr_{[\mu_\triv]}=\Spd\+ L$ is a single point parameterizing the trivial lattice. On the other hand, it is easy to see that $\Gr_G=\mbb{B}_\dR/\mbb{B}^+_\dR=\bigcup_{n \geq 0} t^{-n} \mbb{B}^+_\dR/\mbb{B}^+_\dR$. This is enormous, but any map from a rigid analytic diamond is the constant map to $0$ --- indeed, the only rigid analytic point of $\Gr_G$ over $\Spd\+ L$ is $0$ (recall from \cref{defn:rigid-analytic-diamond} that a rigid analytic point is a map $\Spd\+L' \rightarrow \Gr_G$ for $L'/L$ a finite extension), but classical points are dense in any rigid analytic variety over $L$. 
\end{example}
Part (4) of \cref{theorem.bdr-aff-grass-properties} says that \cref{example.schubert-cells-additive-group} captures the essential behavior of maps from rigid analytic varieties --- although the unipotent radical always contributes an enormous part to the affine Grassmannian, this part is mostly invisible to rigid analytic varieties, and the maps from rigid analytic varieties admit a completely classical description in terms of the flag variety.  The example also illustrates the general method of proof -- after (3) is established, to prove the first equality in (4) it suffices to check it on rigid analytic points since classical points are dense in any rigid analytic variety. To prove (3) we will use \Iref{thm.specialization-goes-up-when-not-good}.

In the remainder of this section we prove \cref{theorem.bdr-aff-grass-properties}. Part (1) is treated in \cref{ss.aff-grass-locally-spatial}. Parts (2) and (3) are treated in \cref{ss.aff-grass-schubert-cells}. Parts (4) and (5) are treated in \cref{ss.aff-grass-rigid-analytic}. 

\subsection{The $\mbb{B}^+_\dR$-affine Grassmannian is an ind-locally spatial diamond}\label{ss.aff-grass-locally-spatial}

In this subsection we will prove part (1) of \cref{theorem.bdr-aff-grass-properties}. That $\Gr_G$ is partially proper is immediate: for $C$ an algebraically closed perfectoid field, the $\Spa(C,C^+)$-points can be identified with $G(\mbb{B}_\dR(C))/G(\mbb{B}^+_\dR(C))$ by \cref{prop.alternative-characterizations}-(2), and this set does not depend on $C^+$. 

To show it is an increasing union of closed locally spatial subdiamonds and separated, it suffices to identify $\Gr_G$ with a subdiamond of $\Gr_{G'}$ for some $G'$ reductive, which is locally spatial and separated by the results of \cite{scholze:berkeley}. The new difficulty in the non-reductive case is that if we take a faithful representation $G \rightarrow \GL_n$, the quotient $\GL_n/G$ will not be affine. However, we claim that there is an embedding $G \hookrightarrow G'$ for $G'$ reductive such that $G/G'$ is quasi-affine. Indeed,  by \cite[Theorem 9]{thang-bac:some-rationality-properties-of-observable-groups} it suffices to find an embedding that realizes $G$ as an observable subgroup of $G'$, i.e.\ as the stabilizer of a vector in some representation of $G'$. If we start with an arbitrary faithful representation $\rho: G \rightarrow \GL_n$, then $G$ is the stabilizer of a line in some representation of $\GL_n$. It thus acts by a character $\chi$ on that line, so we may embed $G$ in $G'=\GL_{n} \times \mbb{G}_m$ as an observable subgroup via $\rho \times \chi^{-1}$. It then remains to apply the following lemma. 

\begin{lemma}\label{lemma.quasi-affine}
Suppose $G'/L$ is a connected linear algebraic group and $G \leq G'$ is a closed subgroup. The push-out map $\Gr_G \rightarrow  \Gr_{G'}$ is an injection. If $G'/G$ is quasi-affine, it identifies $\Gr_G$ with a Zariski locally closed subdiamond of $\Gr_{G'}$. If $G'/G$ is affine, it is Zariski closed.
\end{lemma}
\begin{proof}
For an affinoid perfectoid $\QQ_p$-algebra $R$, the map $G(\mbb{B}_\dR(R))/G(\mbb{B}_\dR^+(R)) \to G'(\mbb{B}_\dR(R))/G'(\mbb{B}_\dR^+(R))$ is injective and gives an injection $\Gr_G \to \Gr_{G'}$ upon applying \'etale sheafification. 

If $x \in \Gr_{G'}((R,R^+))$, then after passing to an \'{e}tale cover, we may assume $x$ is in the image of the natural map $G'(\mbb{B}_\dR(R)) \rightarrow   \Gr_{G'}((R,R+))$ coming from the action on the trivial bundle. The locus where $x$ factors through $\Gr_G$ is thus the locus where the induced point of $(G'/G)(\mbb{B}_\dR(R))$ lies in $(G'/G)(\mbb{B}^+_\dR(R))$. If $G'/G$ is affine, this is Zariski closed by \cref{lemma.zariski-closed-affine}. If $G'/G$ is quasi-affine, let $X$ be an affine scheme containing $G'/G$ as an open subscheme, say $G'/G=X \backslash V(I)$ for an ideal $I$. Then the locus where the induced point of $X(\mbb{B}_\dR)$ comes from $X(\mbb{B}^+_\dR)$ is closed by \cref{lemma.zariski-closed-affine}. Restricting to this closed locus, the map factors through $(G'/G)(\mbb{B}^+_\dR(R))$ precisely when $I$ generates the unit ideal in $\mbb{B}^+_\dR(R)$, which is equivalent to it generating the unit ideal in $R$, thus gives a Zariski open set. 
\end{proof}

\begin{remark} If $G \leq G'$ is a closed subgroup and both $G$ and $G'$ are reductive then $G'/G$ is affine; thus only the affine case of \cref{lemma.quasi-affine} was needed in \cite[Lemma 19.1.5]{scholze:berkeley} to reduce the result for reductive groups to the case of $\GL_n$. \end{remark}

\subsection{Schubert cells}\label{ss.aff-grass-schubert-cells}\newcommand{\Cl}{\mr{Cl}}
In this section we will prove parts (2) and (3) of \cref{theorem.bdr-aff-grass-properties}. We first recall the Bruhat order: let $G/L$ be a connected linear algebraic group, let $U$ denote the unipotent radical of $G$ and let $M=G/U$ be the reductive quotient. We identify the conjugacy classes of cocharacters of $G_{\overline{L}}$ with those of $M_{\overline{L}}$ via the natural map; we write $[\mu]$ for a conjugacy class of cocharacters of $G_{\overline{L}}$ and $[\mu]_M$ for the associated conjugacy class of cocharacters of $M_{\overline{L}}$. The Bruhat order on these conjugacy classes is defined as follows: if we fix a maximal torus $T \subseteq M_{\overline{L}}$ and a Borel subgroup $T \subseteq B \subseteq M_{\overline{L}}$, then, each conjugacy class of cocharacters $[\mu]_M$ of $M_{\overline{L}}$ admits a unique representative factoring through $T$ that is $B$-dominant (i.e.\ $\mbb{G}_m$ acts on $\Lie\+ B$ through $\mr{Ad}_B \circ \mu$ via non-negative powers). For two $B$-dominant cocharacters $\mu_1$ and $\mu_2$ of $T$ in this set, we say $[\mu_1] \geq [\mu_2]$ or $[\mu_1]_M\geq [\mu_2]_M$ if $\mu_1 - \mu_2$ is a non-negative integral combination of $B$-positive coroots. This definition is independent of the choice of $T$ and $B$. 

Recall the sub-functor $\Gr_{[\mu]} \subseteq \Gr_G$ parameterizes lattices of type $[\mu]$ on the trivial $G(\mbb{B}^+_\dR)$-local system. Equivalently, a lattice on the trivial $G(\mbb{B}^+_\dR)$-local system over $X$ is parameterized by a map to $\Gr_{[\mu]}$ if and only if for every geometric point $x: \Spd(C,C^+) \rightarrow X$, the corresponding lattice lies in
\[ G(\mbb{B}^+_\dR(C)) \xi^\mu G(\mbb{B}^+_\dR(C))/G(\mbb{B}^+_\dR(C)) \subseteq G(\mbb{B}_\dR(C)) / G(\mbb{B}^+_\dR(C))\] 
under the identification of \cref{prop.alternative-characterizations}-(2).

We can similarly define $\Gr_{\leq [\mu]}$ to parameterize those lattices whose restriction to every geometric point is of type $[\mu']$ for some $[\mu'] \leq [\mu]$. If $G$ is reductive then, then by \cite[Proposition 19.2.3]{scholze:berkeley}, this is a closed sub-diamond of $\Gr_G$, proper over $\Spd\+ L([\mu])$. As an immediate corollary, $\Gr_{[\mu]}$ is an open sub-diamond of $\Gr_{\leq [\mu]}$ (the complement of the finite union of $\Gr_{ \leq[\mu']}$ for all $[\mu'] < [\mu]$), thus a subdiamond of $\Gr_G$. When $[\mu]$ is minuscule there are no conjugacy classes that lie below $[\mu]$, thus in that case $\Gr_{[\mu]}=\Gr_{\leq[\mu]}$ is a closed sub-diamond of $\Gr_{G}$. 

For $G$ non-reductive, it is no longer the case that every lattice has a type, and the situation is more complicated. It follows immediately from the reductive case that $\Gr_G \times_{\Gr_M} \Gr_{[\mu]_M} \subseteq \Gr_G$ is a locally closed subdiamond, and even a closed subdiamond if $[\mu]_M$ is minuscule, but the locus  
\[ \Gr_{[\mu]} \subseteq \Gr_G \times_{\Gr_M} \Gr_{[\mu]_M} \subseteq \Gr_G \]
is much smaller (see \cref{example.schubert-cells-additive-group}) so we need another argument to control it. The following proposition shows that, within this locus, $\Gr_{[\mu]}$ is cut out by a natural specialization condition on the cocharacter under faithful representations and, in particular, is closed. This will prove part (3) of \cref{theorem.bdr-aff-grass-properties}. As a consequence of this statement and the results described above, $\Gr_{[\mu]}$ is a locally closed subdiamond of $\Gr_G$ and is closed if $[\mu]$ is minuscule in $M$, thus giving also part (2) of \cref{theorem.bdr-aff-grass-properties}.  

\begin{proposition}\label{prop.closed-fiber-specialization} Notation as above, $\Gr_{[\mu]}$ is a closed locally spatial sub-diamond of $\Gr_G \times_{\Gr_M} \Gr_{[\mu]_M}$. More precisely, for any representation $\rho: G \rightarrow \GL_n$, the restriction of the induced map $\Gr_G \rightarrow \Gr_{\GL_n}$ to ${\Gr_G \times_{\Gr_M} \Gr_{[\mu]_M}}$ factors through the locus of lattices of type $\geq [\rho \circ \mu]$ in $\Gr_{\GL_n}$, and if $\rho$ is faithful, 
\begin{equation}\label{eq.intersection-cuts-out-good-locus} \Gr_{[\mu]} = \Gr_G \times_{\Gr_M} \Gr_{[\mu]_M} \cap \Gr_G \times_{\Gr_{\GL_n}} \Gr_{\leq[\rho\circ\mu]}. \end{equation}
\end{proposition}
\begin{proof}
The claim that \cref{eq.intersection-cuts-out-good-locus} holds when $\rho$ is faithful will imply, by picking any faithful representation $\rho$, that $\Gr_{[\mu]}$ is closed in $\Gr_G \times_{\Gr_M} \Gr_{[\mu]_M}$ because $\Gr_{\leq[\rho \circ \mu]}$ is closed in $\Gr_{\GL_n}$. To verify \cref{eq.intersection-cuts-out-good-locus} and the rest of the statement, it suffices to check on geometric points. For $(C,C^+)$-points with $C$ algebraically closed, the statement is immediate from \Iref{thm.specialization-goes-up-when-not-good} after choosing an identification $\mbb{B}^+_\dR(C)\simeq C[[t]]$ as in \Iref{example.filtration-lattice-situations}. 
\end{proof}

\subsection{Rigid analytic subdiamonds}\label{ss.aff-grass-rigid-analytic}

We now prove part (4) of \cref{theorem.bdr-aff-grass-properties}. We first consider the case where $S=\Spa\+ L'$ for $L'/L([\mu])$ a finite extension. In this case, by descent and \cref{prop.aff-perfectoid-loc-system-finite-proj-module}, a point of $\Gr_G(S^\diamond)$ corresponds to a $\Gal(\overline{L}/L')$-invariant $\mbb{B}^+_\dR(C)$-lattice on the trivial $G$-bundle on $\Spec\+ \mbb{B}^+_\dR(C)$ where $C=\overline{L}^\wedge$; this amounts to giving, for each $V \in \Rep\+ G$, a $\mbb{B}_\dR^+(C)$-lattice in $V \otimes_{\QQ_p} \mbb{B}_\dR(C)$ that is functorial, compatible with tensor products, and stable under $\mr{Gal}(\ol{L}/L')$. The category of $\Gal(\overline{L}/L')$-invariant $\mbb{B}^+_\dR(C)$-latticed $L'$-vector spaces is equivalent via the trace filtration to the category of filtered $L'$-vector spaces (for $C=\overline{L}^\wedge$) --- see \Iref{example.galois-action-lattice}. On the other hand, any filtration on the trivial $G$-bundle over $L'$ is of the form $\Fil_{\mu^{-1}}$ for some cocharacter $\mu$ of $G_{L'}$, and tracing through the construction we see that the corresponding lattice on the trivial $G$-bundle over $\mbb{B}^+_\dR(C)$ is $\mc{L}_\mu$. We deduce that the Bialynicki-Birula map gives a bijection $\Gr_{[\mu]}(S^\diamond) = \Fl_{[\mu^{-1}]}^\diamond(S^\diamond)=\Fl_{[\mu^{-1}]}(S)$, and moreover that every element of $\Gr_{G}(S^\diamond)$ lies in $\Gr_{[\mu]}(S^\diamond)$ for some $[\mu]$. This proves part (4) in this case. 

We now treat the case of a general seminormal rigid analytic variety $S/\Spa\+ L([\mu])$. To obtain the equality \cref{eq.rigid-points-bounded-affgr}, we must show that any map $S^\diamond \rightarrow \Gr_G \times_{\Gr_M} \Gr_{[\mu]_M}$ factors through $\Gr_{[\mu]}$. This holds because the classical points are dense in $S$ (see \cref{remark:rigid-analytic-dense-classical-points}),  $\Gr_{[\mu]}$ is closed in $\Gr_G \times_{\Gr_M} \Gr_{[\mu]_M}$ (by part (3) of \cref{theorem.bdr-aff-grass-properties}), and, by the case we have already treated, all classical points in $S$ factor through $\Gr_{[\mu]}$.

We now verify that the maps $S^\diamond \rightarrow \Gr_{[\mu]}$ are as claimed. When $S$ is smooth, this is an immediate consequence of \cref{theorem.smooth-scholze-filtrations-lattices}: the Bialynicki-Birula map induces a bijection between isomorphism classes of lattices of type $[\mu]$ on the trivial $G(\mbb{B}^+_\dR)$-local system and filtrations of type $[\mu^{-1}]$ satisfying Griffiths transversality for the  trivial connection on the trivial $G$-bundle on $S$. 

\newcommand{\sm}{\mr{sm}}
In the general case, given a map $S^\diamond \rightarrow \Gr_{[\mu]}$, we obtain by composition with $\BB$ a map $S^\diamond \rightarrow \Fl_{[\mu^{-1}]}^\diamond$. The corresponding map $S \rightarrow \Fl_{[\mu^{-1}]}$, when restricted to the smooth locus $S^\sm$, satisfies Griffiths transversality by the previous case. 

Conversely, suppose given a map $S \rightarrow \Fl_{[\mu^{-1}]}$ that satisfies Griffiths transversality on $S^\sm$. We must show the associated map $S^\diamond \rightarrow \Fl_{[\mu^{-1}]}^\diamond$ admits a unique lift along $\BB$ to a map $S^\diamond \rightarrow \Gr_{[\mu]}$. To see this, choose a resolution of singularities $\tilde{S} \rightarrow S$ (see \cite[{Theorem 5.3.2}]{temkin:functorial-nonembedded-case}). Then, $\tilde{S}$ contains a dense open isomorphic to $S^\sm$, so the composed map $\tilde{S} \rightarrow \Fl_{[\mu^{-1}]}$ satisfies Griffiths transversality (this is a closed condition). Thus, by the previous case, the associated map of diamonds admits a unique lift $\tilde{S}^\diamond \rightarrow \Gr_{[\mu]}$. Because $\tilde{S}^\diamond \rightarrow S^\diamond$ is a $v$-cover, to verify this comes from a map $S^\diamond \rightarrow \Gr_{[\mu]}$, it suffices to show the two pullbacks to $\tilde{S}^\diamond \times_{S^\diamond} \tilde{S}^\diamond$ agree. This is clear by taking a resolution of singularities $\tilde{\tilde{S}}$ of $\tilde{S} \times_S \tilde{S}$ --- the associated map of diamonds is again a $v$-cover, and the two resulting maps from $\tilde{\tilde{S}}$ to $\Fl_{[\mu^{-1}]}$ are the same (they factor through $S$), so by uniqueness of the lift in the smooth case, the two maps from $\tilde{\tilde{S}}^\diamond$ to $\Gr_{[\mu]}$ which both lift this map to the flag variety must also be the same (cf. the following diagram).  
\[\begin{tikzcd}
	&&&& {\Gr_{[\mu]}} \\
	{\tilde{\tilde{S}}^\diamond} & {(\tilde{S}\times_S \tilde{S})^\diamond=\tilde{S}^\diamond\times_{S^\diamond} \tilde{S}^{\diamond}} & {\tilde{S}^{\diamond}} & {S^{\diamond}} & {\Fl_{[\mu^{-1}]}^{\diamond}}
	\arrow[from=1-5, to=2-5]
	\arrow[curve={height=-12pt}, from=2-1, to=1-5]
	\arrow[from=2-1, to=2-2]
	\arrow["{\pi_\bullet}", from=2-2, to=2-3]
	\arrow[from=2-3, to=1-5]
	\arrow[from=2-3, to=2-4]
	\arrow[from=2-4, to=2-5]
\end{tikzcd}\]

It remains to prove part (5) of \cref{theorem.bdr-aff-grass-properties}. If the weights are $\leq 1$, then the universal filtration over $\Fl_{[\mu^{-1}]}$ satisfies Griffiths transversality, so there is a section $\Fl_{[\mu^{-1}]}^\diamond \rightarrow \Gr_{[\mu]}$ of $\BB$. It is a bijection on geometric points by \Iref{prop.filt-det-latt-criteria}, and because $\Gr_{[\mu]}$ is separated, the section is a closed immersion. Thus it is an isomorphism. The remainder of part (5) is a standard computation --- if we consider the action of $G$ on the point $\Fil^\bullet_{\mu^{-1}}$ in $\Fl_{[\mu^{-1}]}$, then the action of any root subgroup with weight $>1$ for $\mu$ gives a direction where Griffiths transversality fails.

\section{Moduli of neutral $p$-adic Hodge structures}\label{s.moduli-pahs}

\begin{definition}\label{def:neutral-p-adic-HS}
Let $Y$ be a locally spatial diamond over $\breve{\QQ}_p$. A \emph{neutral $p$-adic Hodge structure} on $Y$ is a  $\mbb{Q}$-graded $\mbb{Q}_p$-vector space $V=\bigoplus_{\lambda \in \mbb{Q}} V_\lambda$, and for each $\lambda \in \mbb{Q}$, a $\mbb{B}_\dR^+$-lattice $\mc{L}_\lambda$ on $V_\lambda \otimes_{\mbb{Q}_p} \mbb{B}^+_\dR$ over $Y_v$ such that the restriction of $(V_\lambda, \mc{L}_\lambda)$ along each geometric point $\Spd\+(C,C^+)\rightarrow Y$ is a $p$-adic Hodge structure of weight $\lambda$ in the sense of \Iref{ss.p-adic-hodge-definition} (equivalently, the modified vector bundle $(V_{\lambda} \otimes_{\QQ_p} \mc{O}_\FF)_{\mc{L}_\lambda}$ as in \cref{ss.modifications} is semistable of slope $-\lambda/2$ at each geometric point).
\end{definition}

\begin{example}\label{example.Tate-p-adic-HS}
The pair $\QQ_p(k) = (\QQ_p, \Fil^{-k}\mbb{B}_\dR)$ is the constant Tate $p$-adic Hodge structure of weight $-2k$. 
\end{example}

As in the case of a point in treated in \Iref{theorem.pahs-connected-tannakian}, we find that neutral $p$-adic Hodge structures on $Y$ form a rigid abelian tensor category over $\mbb{Q}_p$, which we denote by $\HS^\circ(Y)$. It is neutral Tannakian: the forgetful functor $\omega_{\et}( V, \mc{L}) := V$ is a fiber functor on $\HS^\circ(Y)$.  

In this section, we define and study moduli of neutral $p$-adic Hodge structures. This is mostly straightforward, and is in part a warm-up for the more complicated case of admissible pairs treated in the next section. In \cref{ss.G-structure-hs}, we extend our definitions to allow $G$-structure. Then, in \cref{ss.period-domain-hs}, we construct moduli spaces $X_{[\mu]}$ of $p$-adic Hodge structures on the trivial $G(\QQ_p)$-local system analogous to the Hermitian symmetric domains appearing in classical Hodge theory over $\mbb{C}$. In \cref{ss.hodge-tate-locus-hs}, we study the Hodge-Tate locus of a tensor, and we apply this in \cref{ss.special-sub-hs} to establish basic properties of special subvarieties. 

\subsection{$G$-structure and Mumford-Tate groups}\label{ss.G-structure-hs}

\begin{definition}
Let $Y$ be a locally spatial diamond over $\breve{\QQ}_p$. For $G/\mbb{Q}_p$ a connected linear algebraic group, a neutral $G$-$p$-adic Hodge structure on $Y$ is an exact tensor functor $\mc{G}: \Rep\+ G \rightarrow \HS^\circ(Y)$ such that the fiber functor $\omega_\et \circ \mc{G}$ is isomorphic to the standard fiber functor $\omega_\std \colon \Rep\+ G \to \Vect_{\QQ_p}$ (the forgetful functor as in \cref{ss.notation}).
\end{definition}

\begin{remark}
An isomorphism $\omega_\std \simeq \omega_\et \circ \mc{G}$ is \emph{not} part of the data of a neutral $G$-$p$-adic Hodge structure. We will call a neutral $G$-$p$-adic Hodge structure equipped with a trivialization $\omega_\std \simeq \omega_\et \circ \mc{G}$ a $p$-adic Hodge structure on the trivial $G(\QQ_p)$-local system.
\end{remark}

Building on \Iref{ss.pahs-invariants}, we define the type of $\mc{G}$ to be the \emph{inverse} of the type of the associated lattice on the trivial $G(\mbb{B}^+_\dR)$-local system over $Y_v$ as defined in \cref{ss.filtrations-lattices-G-torsors}, if it exists (the use of the inverse here is for compatibility with the natural definition on admissible pairs given in \cref{s.moduli-admissible-pairs}). Given a neutral $G$-$p$-adic Hodge structure $\mc{G}$ on $Y$, we let $\langle \mc{G} \rangle$ be the Tannakian subcategory of $\HS^\circ(Y)$ generated by the image of $\mc{G}$. The Mumford-Tate group is $\MT(\mc{G}) = \mr{Aut}^\otimes(\omega_\et|_{\langle \mc{G} \rangle})$, a linear algebraic group over $\mathbb{Q}_p$. If $\mc{G}$ is a $G$-$p$-adic Hodge structure on the trivial $G(\QQ_p)$-local system, then $\MT(\mc{G}) \leq G$ is a closed subgroup (via the trivialization). As in \Iref{ss.canonical-G-structure}, $\mc{G}$ has canonical $\MT(\mc{G})$-structure. 

\subsection{Construction of the the period domain $X_{[\mu]}$}\label{ss.period-domain-hs}
We now construct a moduli space of neutral $G$-$p$-adic Hodge structures of a fixed type $[\mu]$ on the trivial $G$-local system. To that end, fix a Levi decomposition $G=M \ltimes U$, and let $\mc{L}_\univ$ denote the universal lattice on the trivial $G(\mbb{B}_\dR^+)$-local system over $\Gr_{[\mu^{-1}]}$. The modification $(\mc{E}_\triv)_{\mc{L}_\univ}$ as in \cref{ss.modifications} gives a $\Gr_{[\mu^{-1}]}$-point of $\Bun_G$. If the unique basic element $[b]$ of $B(M,[\mu^{-1}])$ is basic in $G$, we define $X_{[\mu]}:= \Gr_{[\mu^{-1}]} \times_{\Bun_G} \Bun_G^{[b]}$. 

\begin{lemma}\label{lemma.hodge-tate-domain-non-empty} $X_{[\mu]} \subseteq \Gr_{[\mu^{-1}]}$ is a non-empty open subdiamond. 
\end{lemma}
\begin{proof}
It is open by \cref{theorem.bun-g-b-structure}, and the existence of a geometric point in $X_{[\mu]}$ follows from \Iref{example.existence-of-G-admissible-pairs} (see also \cref{prop.b-adm-locus-open}) and the equivalence between basic admissible pairs and $p$-adic Hodge structures of \Iref{theorem.adm-pair-propeties}. 
\end{proof}
\begin{remark}
    It will follow from  \cref{prop.b-adm-locus-open} and \cref{theorem.duality}  that $X_{[\mu]}/\Spd \mathbb{Q}_p([\mu])$ is not only non-empty but in fact admits a rigid analytic point in the sense of \cref{defn:rigid-analytic-diamond}. 
\end{remark}

In this setting, let $\nu: \mathbb{G}_m \rightarrow G$ be the slope morphism associated to $[b]$ as recalled in \cref{ss.G-isocrystals}. Then, by construction, the universal lattice $\mathcal{L}_{\mr{univ}}$ with weight grading defined by $2\nu$ restricts to a $p$-adic Hodge structure on the trivial $G(\mathbb{Q}_p)$-local system over $X_{[\mu]}$ (note that the slope of a simple isocrystal is the negative of the slopes of the associated bundle on the Fargues-Fontaine curve!).

\begin{proposition}\label{prop.moduli-pahs-X}
The restriction of the lattice $\mc{L}_\mr{univ}$ on $\omega_{\std}$ to $X_{[\mu]}$ with weight grading induced by $2\nu$ is the universal $p$-adic Hodge structure of type $[\mu]$ and weight $2\nu$ on the trivial $G(\mathbb{Q}_p)$-local system. Moreover, $X_{[\mu]}$ is geometrically connected. 
\end{proposition}
\begin{proof}
Given a $p$-adic Hodge structure $\mc{L}$ of weight $2\nu$ on the trivial $G(\mathbb{Q}_p)$-local system over a geometric point, the $G$-isocrystal underlying the associated $G$-admissible pair in the equivalence of \Iref{theorem.adm-pair-propeties} is basic with slope morphism $2\nu$. Since the slope morphism is central and the type is $[\mu]$, by \Iref{theorem.hs-isocrystal-unique-basic-element}, we conclude the associated $G$-isocrystal is in the isomorphism class $[b]$, the unique basic class in $B(G,[\mu^{-1}])$, i.e. that the modification $(\mc{E}_\triv)_{\mc{L}}$ is a point of $\Bun_G^{[b]}$. Thus $\mc{L}$ is classified by a point of $X_{[\mu]}$. The case of a general base is then immediate since the classifying map to $\Gr_{[\mu^{-1}]}$ for $\mathcal{L}$ factors through $X_{[\mu]}$ if and only if it does so at each geometric point. 

That $X_{[\mu]}$ is geometrically connected will follow from the duality swapping basic admissible pairs and $p$-adic Hodge structures 
(\cref{theorem.duality}) and the corresponding more general result for all (not necessarily basic) admissible loci due to Gleason and Lourenco \cite{gleason-lourenco:connectedness} in the reductive case and extended below in \cref{prop.adm-connected}. 
\end{proof}

\begin{remark}
    We use the letter $X$ to emphasize the similarity with Hermitian symmetric domains parameterizing real Hodge structures, which are often denoted (e.g. in a Shimura datum) by $X$.
\end{remark}

\begin{remark}
    If $[b]$ is basic in $M$ but not in $G$, then any neutral $G$-$p$-adic Hodge structure of type $[\mu]$ over a geometric point can be constructed as a push-out from the centralizer of the slope morphism attached to any fixed choice of $b$ in $[b]$ (which will not be unique since the slope morphism is not central in $G$). 
\end{remark}

\subsection{The Hodge-Tate locus of a tensor}\label{ss.hodge-tate-locus-hs}

In \Iref{ss.hodge-tate-lines}, we defined Hodge-Tate lines and used them to characterize the Mumford-Tate groups of $p$-adic Hodge structures, mirroring the theory of Hodge tensors in classical complex Hodge theory. We now define Hodge-Tate loci, which are the loci where a line in a neutral $p$-adic Hodge structure is Hodge-Tate. This mirrors the theory of Hodge loci for tensors in classical complex Hodge theory (we will also see another incarnation in \cref{ss.hodge-locus}). 

\begin{definition}\label{def.Hodge-Tate-locus}
Let $S/\breve{\mbb{Q}}_p$ be a locally spatial diamond equipped with a $p$-adic Hodge structure $\mc{G}$ on the trivial $G(\QQ_p)$-local system over $S$. Let $V$ be a representation of $G$ and write $\mc{G}(V) = (V, \mc{L})$. Let $\ell \subseteq V_{2n}$ be a line in the weight $2n$ component of $V$. The \emph{Hodge-Tate locus} of $\ell$ is the subfunctor of $S$ whose value on $Y$ in $\Perfd_{\breve{\mbb{Q}}_p}$ is 
    \[ \HT(\ell)(Y) = \{f \in S(Y) \colon \ell \otimes_{\QQ_p} \Fil^n\mbb{B}_\dR \subseteq f^\ast\mc{L}\}. \]
\end{definition}
Thus the Hodge-Tate locus of $\ell$ is precisely where $\ell$ underlies a sub-$p$-adic Hodge structure (necessarily isomorphic to the weight $2n$ Tate $p$-adic Hodge structure $\QQ_p(-n) = (\QQ_p, \Fil^n\mbb{B}_\dR)$). 

\begin{lemma}\label{lem:HT-locus-closed}
With notation as in \cref{def.Hodge-Tate-locus}, the Hodge-Tate locus $\HT(\ell) \subseteq S$ is a closed subdiamond. It is the locus where $\MT(\mc{G}|_{\HT(\ell)}) \subseteq \mr{Stab}_G(\ell)$ (i.e., for any map $f:T \rightarrow S$ such that $\MT(f^*\mc{G})\subseteq \mr{Stab}_G(\ell)$, $f$ factors through $\HT(\ell)$). 
\end{lemma}
\begin{proof}
We need to show that, for an affinoid perfectoid space $Y/S$, the locus $Z$ on $Y$ where $\ell \otimes_{\QQ_p} \Fil^{n}\mbb{B}_\dR^+(Y)  \subseteq \mc{L}$ is closed in $Y$. Replacing $Y$ by an open subspace, we can assume that $\mc{L} \cong \mbb{B}_\dR^+(Y)^k$ is a free $\mbb{B}_\dR^+$-module. Then $Z$ is the locus of $Y$ where a generator $e$ (as a $\mbb{B}_\dR^+(Y)$-module) of 
    \[ \ell \otimes_{\QQ_p} \Fil^{n}\mbb{B}_\dR(Y) \subseteq V \otimes_{\QQ_p} \mbb{B}_\dR(Y) = \mc{L} \otimes_{\mbb{B}_\dR^+(Y)} \mbb{B}_\dR(Y) \cong \mbb{B}_\dR(Y)^k \]
lies in $\mc{L} \cong \mbb{B}_\dR^+(Y)^k$. By \cref{lemma.zariski-closed-affine}, this description implies $Z$ is closed. The claim on the Mumford-Tate group follows by the Tannakian formalism. 
\end{proof}

\begin{lemma}\label{lemma:pahs-rigid-ananlytic-HT-locus} 
Let $[\mu]$ be a conjugacy class of cocharacters of $G_{\overline{\mathbb{Q}}_p}$, let $L/\breve{\mathbb{Q}}_p([\mu])$ be a $p$-adic field, let $S/\Spa\+ L$ be a seminormal rigid analytic variety, and let $\mc{G}$ be a $p$-adic Hodge structure on the trivial $G(\QQ_p)$-local system over $S^\diamond$ of type $[\mu]$. Then, for $\ell$ as in \cref{def.Hodge-Tate-locus}, $\HT(\ell) \subseteq S^\diamond$, which is closed by \cref{lem:HT-locus-closed}, is also rigid analytic (equivalently, by \cref{lemma.closed-rigid-analytic-subdiamonds-zariski-closed}, it corresponds to a Zariski closed subset of $S$). 
\end{lemma}
\begin{proof}
Consider the latticed $\mbb{B}_\dR^+$-local system $(V \otimes_{\QQ_p} \mbb{B}_\dR^+, \mc{L})$, and let $\hat{\mc{V}} = V \otimes_{\QQ_p} \mc{O}$ be the associated trivial $\mc{O}$-local system. Let $\Fil^\bullet \hat{\mc{V}}$ be the filtration on $\hat{\mc{V}}$ from \cref{theorem.smooth-scholze-filtrations-lattices}-(3), which we can apply since $\mc{L}$ has constant type. From the proof of \cref{theorem.smooth-scholze-filtrations-lattices}-(3), the filtration $\Fil^\bullet \hat{\mc{V}}$ is pulled back from a filtration $\Fil^\bullet \mc{V}$ of the trivial vector bundle $\mc{V} = V \otimes_{\QQ_p} \mc{O}_{S_\et}$ by local direct summands. Choose a basis element $s$ for $\ell$, and let $\HT_f(\ell) \subseteq |S| \cong |S^\diamond|$ denote the locus where $s$ maps to zero in $\mc{V}/\Fil^{-n}\mc{V}$, which is Zariski closed. From \cref{theorem.smooth-scholze-filtrations-lattices}-(2), the filtration $\Fil^\bullet \hat{\mc{V}}$ is the trace filtration of $\mc{L}$ (the Hodge-Tate filtration):
   \[ \Fil^k \hat{\mc{V}} = (\Fil^k\mc{L} \cap (V \otimes_{\QQ_p} \mbb{B}_\dR^+))/(\Fil^k\mc{L} \cap (V \otimes_{\QQ_p} \Fil^1\mbb{B}_\dR)). \]
We claim $\HT(\ell)=\HT_f(\ell)$. First $\HT(\ell) \subseteq \HT_f(\ell)$ follows from the description of $\Fil^\bullet$ via the trace filtration. Since classical points are dense in $\HT_f(\ell)$ and $\HT(\ell)$ is closed, it suffices to prove the reverse inclusion when $S = \Spa\+ L$ (or a finite extension). In this case $\mc{V} = V_L := V \otimes_{\QQ_p} L$ and letting $C$ be the completion of an algebraic closure $\ol{L}$ of $L$, the lattice $\mc{L}$ corresponds to the $\mbb{B}_\dR^+(C)$-lattice
    \[ M_0 = \Fil^0(V_L \otimes_L \mbb{B}_\dR(C)) = \sum_{i \in \ZZ} \left(\Fil^{-i}V_L \otimes_{L} \Fil^i \mbb{B}_\dR(C)\right) \subset V \otimes_{\QQ_p} \mbb{B}_\dR(C).\]
In particular, if, $s \in \Fil^{-n}V_L$ then $s \otimes \Fil^n\mbb{B}_\dR(C) \subseteq M_0$, as required. 

\end{proof}

\subsection{Special subvarieties}\label{ss.special-sub-hs} Suppose given a closed subgroup $H \leq G$ and a conjugacy class of cocharacters $[\mu_H]$ of $H_{\overline{\mbb{Q}}_p}$ contained in $[\mu]$ such that there is an associated basic class $[b_H]$ in $B(H,[\mu^{-1}_H])$ and moreover the image $[b_G]$ of $[b_H]$ in $B(G)$ is basic. Then, $[b_G]$ is the basic conjugacy class in $B(G,[\mu^{-1}])$, and thus there is a functoriality map $X_{[\mu_H]} \hookrightarrow X_{[\mu], \Qpbreve([\mu_H])}$ that is injective since $\Gr_H \to \Gr_G$ is injective. 

\begin{proposition}\label{prop.Xmu-funct-closed}
Suppose $H=M_H \ltimes U_H$ is a Levi decomposition, and $[\mu_H]$ is $M_H$-minuscule. Then $X_{[\mu_H]}$ is a closed sub-diamond of $X_{[\mu], \breve{\mbb{Q}}_p([\mu_H])}$.
\end{proposition}
\begin{proof}

Let $V$ be a representation of $G$ and $\ell \subseteq V$ a line such that $H=\mr{Stab}_G(\ell)$. By \cref{lem:HT-locus-closed}, $\HT(\ell) \subseteq X_{[\mu]}$ is the locus where $\MT(\mc{L}_{\mr{univ}}) \subseteq H$, for $\mc{L}_{\mr{univ}}$ the lattice of the universal $G$-$p$-adic Hodge structure on the trivial $G(\QQ_p)$-local system. Consequently, the functoriality map $X_{[\mu_H]} \hookrightarrow X_{[\mu],\breve{\mbb{Q}}_p([\mu_H])}$ factors through $\HT(\ell)_{\breve{\mbb{Q}}_p([\mu_H])}$.  
By \cref{lem:HT-locus-closed}, the restriction of $\mc{L}_\univ$ to $\HT(\ell)$ comes from a neutral $H$-$p$-adic Hodge structure $\mc{L}_H$ on the trivial $H(\QQ_p)$-local system. The classifying map for the lattice $\mc{L}_H$ then provides a map $\HT(\ell) \to \Gr_{H}$, and we write $Z$ for the preimage in $\HT(\ell)_{\breve{\mbb{Q}}_p([\mu_H])}$ of $\Gr_{[\mu_H^{-1}]}$. Then $Z$ is a closed subdiamond of $\HT(\ell)_{\breve{\mbb{Q}}_p([\mu_H])}$ by \cref{theorem.bdr-aff-grass-properties}-(2) (because $[\mu_H]$ and thus $[\mu_H^{-1}]$ is $M_H$-minuscule) and thus a closed subdiamond of $X_{[\mu],\breve{\mbb{Q}}_p([\mu_H])}$ by \cref{lem:HT-locus-closed}. By construction, the functoriality map $X_{[\mu_H]} \rightarrow X_{[\mu],\breve{\mbb{Q}}_p([\mu_H])}$ factors through a map $X_{[\mu_H]} \rightarrow Z$ and the classifying map for $\mc{L}_H$ restricts to a section $Z \rightarrow X_{[\mu_H]}$. Note that $X_{[\mu_H]} \rightarrow Z$ is injective (since, as noted before the statement of the proposition, $X_{[\mu_H]} \hookrightarrow X_{[\mu], \Qpbreve([\mu_H])}$ is injective) and a section of an injective map is surjective. It is also injective since $Z \to X_{[\mu_H]} \to X_{[\mu],\breve{\QQ}_p([\mu_H])}$ is and thus $X_{[\mu_H]}$ is identified with the closed subdiamond $Z$ of $X_{[\mu],\breve{\mbb{Q}}_p([\mu_H])}$. 
\end{proof}

Note that, if $[\mu_H]$ has weights $\geq -1$ on $\Lie\+ H$, then its image in $M_H$ is minuscule, thus $X_{[\mu_H]}$ is closed in $X_{[\mu], \Qpbreve([\mu_H])}$ by the above. Moreover, in this case $X_{[\mu_H]}$ is rigid analytic (by \cref{theorem.bdr-aff-grass-properties}-(5), it is identified with an open in $\Fl_{[\mu_H]}^\diamond$). 
\begin{definition}\label{def.special-subvariety-pahs} A \emph{special subvariety of $X_{[\mu],C}$} is a closed rigid analytic subdiamond of the form $X_{[\mu_H],C}$ for $[\mu_H]$ with weights $\geq -1$ on $\Lie\+ H$ (where $C$ is an algebraically closed non-archimedean extension of $\breve{\QQ}_p$).
\end{definition}

\section{Moduli of neutral admissible pairs}\label{s.moduli-admissible-pairs}
In this section we define a category of neutral admissible pairs over a locally spatial diamond $Y/\Spd\+ \breve{\mbb{Q}}_p.$ 

\subsection{Neutral admissible pairs and associated objects}

\begin{definition}
Let $Y/\Spd\+ \breve{\mbb{Q}}_p$ be a locally spatial diamond, and consider a pair $(W, \mc{L}_\et)$ where $W$ is an isocrystal and $\mc{L}_\et$ is $\mbb{B}^+_\dR$-lattice in $W \otimes_{\breve{\mbb{Q}}_p} {\mbb{B}_\dR}$ over $Y$. We say that $(W,\mc{L}_\et)$ is a \emph{neutral admissible pair} if, for every geometric point $x: \Spa(C,C^+) \rightarrow X$, the restriction $x^*(W,\mc{L}_\et)$ is an admissible pair in the sense of Part I. This is equivalent to requiring that the vector bundle $\mc{E}(W)_{\mc{L}_\et}$, obtained by modifying the bundle $\mc{E}(W)$ on the relative Fargues-Fontaine curve over $Y$ by $\mc{L}_\et$ as in \cref{ss.modifications}, be semistable of slope 0.  

A morphism of neutral admissible pairs $f \colon (W, \mc{L}) \to (W', \mc{L}')$ is a morphism $f \colon W \to W'$ of isocrystals such that $f_{\mbb{B}_\dR}(\mc{L}) \subseteq \mc{L}'$. The latter condition can be checked at geometric points, so it is equivalent to require that it restricts to a morphism of admissible pairs at each geometric point in the sense of \Iref{s.admissible-pairs}. We write $\AdmPair^\circ(Y)$ for the category of neutral admissible pairs on $Y$, which has a tensor product and duals defined in the obvious way.
\end{definition}

\begin{theorem}\label{theorem.neutral-adm-pair-tannakian}
The category $\AdmPair^\circ(Y)$ is a neutral Tannakian category over $\QQ_p$ with fiber functor, for any geometric point $y \colon \Spd\+(C, C^+) \to Y$, 
    \[ \omega_y \colon \AdmPair^\circ(Y) \to \Vect_{\QQ_p}, \qquad (W, \mc{L}) \mapsto H^0(\FF_{C^\flat}, \mc{E}(W)_{y^\ast \mc{L}}). \]
\end{theorem}
\begin{proof}
When $Y = \Spd\+(C,C^+)$ this is \Iref{theorem.adm-pair-propeties}, and the proof carries over in the general case. For the convenience of the reader we summarize it. 

To see that it is a Tannakian category, the non-trivial point is to show that $\AdmPair^\circ(Y)$ is abelian. Following the the argument of \Iref{lemma.p-adic-Hodge-abelian}, we only need to show that images and coimages coincide. Let $f \colon (W, \mc{L}) \to (W', \mc{L}')$ be a morphism, and define $T = f(W), \mc{M}_1 = f(\mc{L}), \mc{M}_2 = \mc{L}' \cap (T \otimes_{\breve{\QQ}_p} \mbb{B}_\dR)$. Then $\mc{M}_1 \subseteq \mc{M}_2$ and we aim to show equality, which we can do after restricting to each geometric point $y$ of $Y$. For $i = 1,2$, let $\mc{V}_i = \mc{E}(T)_{y^\ast\mc{M}_i}$ be the modified vector bundle on $\FF_{C^\flat}$. Let $\lambda_i$ be the slope of $\mc{V}_i$. The following exact sequence of coherent sheaves on $\FF_{C^\flat}$
    \[ 0 \to \mc{V}_1 \to \mc{V}_2 \to \infty_{C, \ast}(y^*\mc{M}_2/y^*\mc{M}_1)  \to 0 \]
shows $\lambda_1 \leq \lambda_2$ with equality if and only if $y^\ast\mc{M}_1 = y^\ast\mc{M}_2$. The surjection $\mc{E}(W)_{y^\ast\mc{L}} \twoheadrightarrow \mc{V}_1$ shows $0 \leq \lambda_1$ and the injection $\mc{V}_2 \hookrightarrow \mc{E}(W)_{y^\ast\mc{L}}$ shows $\lambda_2 \leq 0$ and we conclude $\lambda_1 = \lambda_2 = 0$. 

Finally, we note that, for any geometric point $y \colon \Spd\+(C, C^+) \to Y$, the functor $\omega_y$ is easily seen to be an exact tensor functor, so it remains only to show it is faithful. To see that it is faithful, we first note that $y^*: \AdmPair^\circ(Y) \rightarrow \AdmPair^\circ(\Spd\+(C,C^+))$ is faithful (since any homomorphism of neutral admissible pairs is uniquely determined by the homomorphism of the underlying isocrystal). It thus suffices to show that it is faithful when $Y= \Spd\+(C,C^+)$: in this case, for $f: (W_1, \mc{L}_1) \rightarrow (W_2, \mc{L}_2)$, $\omega_y(f)$ determines $\mc{E}(W_1) \rightarrow \mc{E}(W_2)$ on $\FF_{C^\flat}\backslash \infty_C$, and thus determines $f$ itself (since the passage from isocrystals to vector bundles on $\FF_{C^\flat}$ is faithful and so is the restriction of vector bundles from $\FF_{C^\flat}$ to $\FF_{C^\flat}\backslash\infty_C$). 
\end{proof}

\begin{remark}\label{remark.neutral-admissible-pairs-non-connected-base}
    If $Y$ is not connected, then it may seem surprising at first that the functor $\omega_y$ of \cref{theorem.neutral-adm-pair-tannakian} is faithful. This is essentially because the definition of a \emph{neutral} admissible pair behaves very poorly from the perspective of descent (see also \cref{ss.intro-relative-theory}). In Part III, when we define an admissible pair in a general relative setting, the pair will consist of a local system of isocrystals and a lattice. In particular, over a non-connected base, a fiber functor $\omega_y$ as above will no longer be faithful, since it will only see the part of a morphism of a local system of isocrystals that is happening on the connected component of $y$. The reason $\omega_y$ can still be faithful in the neutral case is that, over a non-connected base, neutral admissible pairs are not a full subcategory of admissible pairs. 
\end{remark}

On $\AdmPair^\circ(Y)$, the forgetful functor
\begin{equation}\label{eqn:AP-isocrystalline-realization}
    \omega_\Isoc \colon \AdmPair^\circ(Y) \to \Isoc, \qquad (W, \mc{L}) \mapsto W
\end{equation}
is called the \emph{isocrystalline realization}. The equivalence between vector bundles on the Fargues-Fontaine curve that are semi-stable of slope zero and $\mbb{Q}_p$-local systems (\cite{kedlaya-liu:relative-foundations} in the form \cite[Corollary II.2.20]{fargues-scholze:geometrization}) gives that
\begin{equation}\label{eqn:AP-etale-realization}
    \omega_\et \colon \AdmPair^\circ(Y) \to \Loc_{\QQ_p}(Y), \qquad (W,\mc{L}) \mapsto H^0(\FF_Y, \mc{E}(W)_{\mc{L}}) 
\end{equation} 
is an exact tensor functor from $\AdmPair^\circ(Y)$ to $\Loc_{\mbb{Q}_p}(Y)$ called the \emph{\'etale realization}. By the construction using modifications, there is a canonical de Rham comparison of the \'etale and isocrystalline realizations
\begin{equation}\label{eqn:AP-de-Rham-comparison}
c_\dR \colon \omega_\et \otimes_{\ul{\mathbb{Q}_p}} \mbb{B}_\dR \xrightarrow{\sim} \omega_\Isoc \otimes_{\Qpbreve} \mbb{B}_\dR.
\end{equation}

In particular, there are canonical $\mbb{B}_\dR^+$-lattices enriching both the \'etale and isocrystalline realization:
The \'etale lattice functor
\begin{equation}\label{eqn:AP-etale-lattice-realization}
    \omega_{\mc{L}_{\et}} \colon \AdmPair^\circ(Y) \to \Loc_{\mbb{B}_\dR^+}(Y), \qquad (W, \mc{L}) \mapsto \mc{L} 
\end{equation}
is a $\mbb{B}_\dR^+$-lattice on $\omega_\Isoc \otimes_{\breve{\QQ}_p} \mbb{B}_\dR$.

The de Rham lattice functor
\begin{equation}\label{eqn:AP-de-Rham-lattice-realization}
     \mc{L}_\dR \colon \AdmPair^\circ(Y) \to \Loc_{\mbb{B}_\dR^+}(Y), \qquad (W, \mc{L}) \mapsto c_\dR^{-1}(W \otimes_{\breve{\QQ}_p} \mbb{B}_\dR^+).
\end{equation}
is a $\mbb{B}_\dR^+$-lattice on $\omega_\et \otimes_{\ul{\QQ_p}} \mbb{B}_\dR$.

\begin{example}
The Tate admissible pair $\breve{\QQ}_p(n)$ has underlying isocrystal $(\breve{\QQ}_p, p^{-n}\varphi_{\breve{\QQ}_p})$ and lattice $\Fil^n \mbb{B}_\dR$. 
\end{example}

\subsection{$G$-structure}\label{ss.G-structure-ap}
 
\begin{definition}
Let $Y/\Spd\+ \breve{\mbb{Q}}_p$ be a locally spatial diamond and let $G/\mbb{Q}_p$ be a connected linear algebraic group. A neutral $G$-admissible pair on $Y$ is an exact tensor functor $\mc{G}: \Rep\+ G \rightarrow \AdmPair^\circ(Y)$ such that, for each geometric point $y: \Spd\+(C,C^+) \rightarrow S$, $\omega_y \circ \mc{G}$ is isomorphic to $\omega_\std$. 
\end{definition}

If $A$ is a neutral admissible pair on a locally spatial diamond $Y/\Spd\+ \breve{\QQ}_p$,  let $\langle A \rangle$ be the Tannakian subcategory of $\AdmPair^\circ(Y)$ generated by $A$. For a geometric point $y \colon \Spd\+(C,C^+) \to Y$, 
the motivic Galois group with respect to $y$ is $\MG(A,y):=\ul{\Aut}^{\otimes} \omega_y|_{\langle A \rangle}$, and $A$ admits a canonical refinement to an admissible pair with $\MG(A,y)$-structure (see \Iref{ss.tannakian}). Note that $\MG(A,y)$ is linear algebraic group over $\mathbb{Q}_p$, and is naturally a subgroup of $\GL(\omega_{y}(A))$. 

More generally, suppose $\mc{G}$ is a neutral $G$-admissible pair on $Y$. The motivic Galois group with respect to $y$ is given by $\MG(\mc{G}, y)=\ul{\Aut}^{\otimes} \omega_y|_{\langle \mc{G} \rangle}$  (where $\langle \mc{G} \rangle$ is the Tannakian category generated by the image of $\Rep\+ G$ in $\AdmPair^\circ(Y)$ or, equivalently, by $\mc{G}(V)$ for $V$ any faithful representation of $G$). It is a linear algebraic group over $\mathbb{Q}_p$, and if we fix a trivialization $\omega_{y}\circ \mc{G} \simeq \omega_\std$ then, since there is a canonical identification of $G$ with $\ul{\Aut}^\otimes(\omega_\std)$, we obtain an embedding of $\MG(\mc{G},y)$ into $G$ as a closed subgroup. 

If $\mc{G}$ is a neutral $G$-admissible pair on $Y$, then may choose an isomorphism $\omega_\Isoc \circ \mc{G} \cong \omega_\std \otimes_{\QQ_p} \breve{\QQ}_p$ as in \Iref{sss.isocrystals-G-structure}. In particular, this allows us to view $\omega_{\mc{L}_\et} \circ \mc{G}$ as a $\mbb{B}_\dR^+$-lattice on the trivial $G(\mbb{B}_\dR)$-local system and gives a classifying map $Y \to \Gr_{G, \breve{\QQ}_p}$. 

\begin{definition}
For $[\mu]$ a conjugacy class of cocharacters of $G_{\overline{\mathbb{Q}}_p}$, 
 $Y/\Spd \Qpbreve([\mu])$ a locally spatial diamond, and $\mc{G}$ a neutral $G$-admissible pair on $Y$, we say that $\mc{G}$ has type $[\mu]$ if the image of the classifying map $Y\rightarrow \Gr_{G, \Qpbreve([\mu])}$ factors through $\Gr_{[\mu]}$.
\end{definition}

While the classifying map $Y \to \Gr_{G, \breve{\QQ}_p}$ depends on the choice of isomorphism $\omega_\Isoc \circ \mc{G} \cong \omega_\std \otimes \breve{\QQ}_p$, the definition above does not: any other isomorphism differs by an element of $G(\breve{\QQ}_p)$, and $\Gr_{[\mu]}$ is stable under multiplication by elements of $G(\breve{\QQ}_p)$.

Recall the conventions on types of lattices versus filtrations (\cref{s.filtrations-and-lattices}): if $\mc{G}$ has type $[\mu]$ then the filtration on $(\omega_\Isoc \circ \mc{G})\otimes_{\Qpbreve} \mc{O}$ (this is the Hodge filtration) has type $[\mu^{-1}]$. Similarly, the filtration on $\omega_\et \otimes_{\ul{\mathbb{Q}_p}} \mc{O}$ associated to the lattice $\omega_{\mc{L}_\dR} \circ \mc{G}$ (the Hodge-Tate filtration) has type $[\mu]$.

\begin{theorem}\label{theorem.rigid-an-has-type} If $L/\breve{\mbb{Q}}_p$ is a $p$-adic field and $S/\Spa\+ L$ is a geometrically connected seminormal rigid analytic variety, then any neutral $G$-admissible pair on $S^\diamond$ has a type. 
\end{theorem}
\begin{proof}

First note that, for each $V \in \Rep\+ G$, the $\QQ_p$-local system $\omega_\et \circ \mc{G}(V)$ on $S^\diamond$ descends uniquely to a $\QQ_p$-local system on $S_\et$ or a $\widehat{\QQ}_p$-local system on $S_\proet$ (in the notation of \cite{scholze:p-adic-ht}, \cite{liu-zhu:riemann-hilbert}). For ease of notation, we let $\omega_\et \circ \mc{G}(V)$ be any of these. 

We give a more precise argument for this descent. By \cite[Proposition 3.10]{mann-werner}, we may pass to an \'etale cover $U$ of $S$ so that there is a $\ZZ_p$-lattice $\Lambda \subseteq \omega_\et \circ \mc{G}(V)$. The $v$-sheaf $\mc{I}_V$ of isomorphisms $\ul{\ZZ}_p^n \simeq \Lambda$ (for $n = \dim\+ V$) is a $\ul{\GL_n(\ZZ_p)}$-torsor on $U^\diamond$. It follows from from \cite[Lemma 10.13]{scholze:etale-cohomology-of-diamonds} that, for each $r$, $\mc{I}_V \times_{\ul{\GL_n(\ZZ_p)}} (\ul{\ZZ/p^r})^n$  finite \'etale over $U^\diamond$, hence pullback along $U_\et \simeq U^\diamond_\et$, yields a $\ZZ_p$-local system on $U_\et$ (i.e.\ a lisse $\ZZ_p$-sheaf as in \cite[Definition 8.1]{scholze:p-adic-ht}). Inverting $p$ gives a $\QQ_p$-local system on $U_\et$ and by \'etale descent a $\QQ_p$-local system on $S_\et$. 

Now, it suffices to show that each $\omega_\et \circ \mc{G}(V)$ is a de Rham $\QQ_p$-local system on $S_\et$ in the sense of \cite[Definition 8.4]{scholze:p-adic-ht}. Indeed, in this case there is by definition an associated filtered vector bundle with integrable connection satisfying Griffiths transversality $(\mc{E}(V), \Fil^\bullet, \nabla)$ on $S_\et$ which satisfies, for $\nu \colon S_\proet \to S_\et$
    \[ \mc{E}(V) = \nu_\ast\left( \omega_\et \circ \mc{G}(V) \otimes_{\ul{\QQ_p}} \mc{O}\mbb{B}_\dR\right) = \nu_\ast\left( \omega_\Isoc \circ \mc{G}(V) \otimes_{{\breve{\QQ}_p}} \mc{O}\mbb{B}_\dR\right) = V \otimes_{\QQ_p} \mc{O}_{S_\et}. \]
(The first equality and last equalities by  \cite[Theorem 7.6-(i), Corollary 6.19]{scholze:p-adic-ht} respectively, and the middle equality from \cref{eqn:AP-de-Rham-comparison}.) Consequently $V \mapsto \Fil^\bullet \mc{E}(V)$ is a filtration on the trivial $G$-bundle on $\mc{O}_{S_\et}$. Moreover $\omega_{\mc{L}_{\et}} \circ \mc{G}$ is the associated lattice on the trivial $G(\mbb{B}_\dR)$-local system from \cref{theorem.bdr-aff-grass-properties} --- this can be verified, e.g., by comparing at all classical points. In particular, since $S$ is geometrically connected, the classifying map for the filtration factors through a connected component of the flag variety corresponding to some $[\mu]$, and we conclude that $\omega_{\mc{L}_\et \circ \mc{G}}$ has a type. 

To verify $\omega_\et \circ \mc{G}(V)$ on $S_\et$ is de Rham, by \cite[Theorem 3.9-(iv)]{liu-zhu:riemann-hilbert}, it suffices to check at a single classical point. But it is immediate from the definitions that the Galois representation attached to the rigid analytic admissible pair at a classical point is de Rham as it is crystalline by \Iref{theorem.full-faithful-essential-image}. 

\end{proof}

\subsection{The $b$-admissible locus} 

\begin{definition}
Let $G/\mbb{Q}_p$ be a connected linear algebraic group, let $\mc{G}$ be a neutral $G$-admissible pair on $Y$ and let $b \in G(\breve{\QQ}_p)$. A $b$-rigidification of $\mc{G}$ is an isomorphism $\rho_\Isoc \colon W_{b} \xrightarrow{\sim} \omega_\Isoc \circ \mc{G}$ with $W_{b}$ the $G$-isocrystal associated to $b$ (that is, the tensor functor $V \mapsto (V_{\breve{\QQ}_p}, b\varphi_{\breve{\QQ}_p})$ for $V \in \Rep\+ G$). 
\end{definition}
Note that a $b$-rigidification $\rho_{\Isoc}$ also gives an isomorphism of $\omega_{\std} \otimes_{\QQ_p} \breve{\QQ}_p$ with $\omega_\Isoc \circ \mc{G}$ (viewed as a functor to $\breve{\mbb{Q}}_p$-vector spaces instead of isocrystals). 

 We want to construct moduli of $b$-rigidified $G$-admissible pairs. To that end, we have the $G$-bundle $\mc{E}_b$ and the universal lattice $\mc{L}_\univ$ on the trivial $G$-bundle over $\Gr_{G,\breve{\mbb{Q}}_p}$, and we may form the modification $(\mc{E}_b)_{\mc{L}_\univ}$ to obtain a map $\Gr_{G,\breve{\mbb{Q}}_p} \rightarrow \Bun_G$. We define $\Gr_G^{b-\adm}$ to be the pre-image of the open substack (by \cref{theorem.bun-g-b-structure}) $\Bun_G^{[1]}$, i.e.\ the locus whose geometric points are those where the modification is the trivial $G$-bundle. It is immediate from the definitions that the restriction of $(b, \mc{L}_\univ)$ to $\Gr_G^{b-\adm}$ is the universal $b$-rigidified $G$-admissible pair $\mc{G}_\univ$:
 
\begin{proposition}Pullback of $\mc{G}_\univ$ induces 
\begin{align*}\Gr_{G}^{b-\adm}(S) &= \{ (\mc{G}, \rho_\Isoc)\, |\, \mc{G} \in G\text{-}\AdmPair^\circ(S), \rho_{\Isoc}: W_b \cong \omega_\Isoc \circ \mc{G} \} / \sim.
\end{align*}
\end{proposition}

For $[\mu]$ a conjugacy class of cocharacters of $G_{\overline{\mbb{Q}}_p}$, we write 
\[ \Gr_{[\mu]}^{b-\adm} := \Gr_{[\mu],\breve{\mbb{Q}}_p([\mu])} \times_{\Gr_{G,\breve{\mbb{Q}}_p([\mu])}} \Gr_{G,\breve{\mbb{Q}}_p([\mu])}^{b-\adm} \] 
In the moduli interpretation above, it is clear that $\Gr_{[\mu]}^{b-\adm}$ parameterizes $b$-rigidified admissible pairs of type $[\mu]$.

\begin{proposition}\label{prop.b-adm-locus-open}
$\Gr_{[\mu]}^{b-\adm}$ is open in $\Gr_{[\mu],\breve{\mbb{Q}}_p([\mu])}$ and non-empty if and only if $[b] \in B(G,[\mu^{-1}])$. Moreover, when it is non-empty, $\Gr_{[\mu],\breve{\mbb{Q}}_p([\mu])} / \Spd\+ \breve{\mbb{Q}}_p([\mu])$ admits a rigid analytic point (see \cref{defn:rigid-analytic-diamond}).  
\end{proposition}
\begin{proof}
Openness follows from the description $\Gr_{[\mu]}^{b-\adm}$ as the preimage under the uniformization map $\Gr_{[\mu],\breve{\mbb{Q}}_p([\mu])} \to \Bun_G, \mc{L} \mapsto (\mc{E}_b)_{\mc{L}}$ of the substack $\Bun^{[1]}_G \subseteq \Bun_G$, an open substack by \cref{theorem.bun-g-b-structure}. 

We now establish the characterization of non-emptiness. To that end, choose a Levi decomposition $G = MU$ and let $[\mu_M]$ be the $M$-conjugacy class of cocharacters induced by $[\mu]$. We are free to replace $b$ by another element in $[b]$ and can thus assume $b \in M(\breve{\QQ}_p)$ since $B(M) \cong B(G)$ (\cref{remark.Levi-decomp-BG-BM-and-basic}). The composition $\Gr_{[\mu_M]}^{b_M-\adm} \to \Gr_{[\mu]}^{b-\adm} \to \Gr_{[\mu_M]}^{b_M-\adm}$ of the functoriality maps induced by $M \to G \to M$ (\ref{eq.funct-map-adm-loc}) is the identity, which reduces the claim to the reductive case. In this case, non-emptiness implies $[b] \in B(G,[\mu^{-1}])$ by \cite[Proposition 3.5.3]{caraiani-scholze:cohomology-compact-shimura} and, in the other direction, if $[b] \in B(G,[\mu^{-1}]$), then the existence of a rigid analytic point follows from \cite[Proposition 3.1]{rapoport-viehmann} (see also \Iref{example.existence-of-G-admissible-pairs}). 

\end{proof}

The results of \cite{gleason-lourenco:connectedness} also imply
\begin{proposition}\label{prop.adm-connected}
$\Gr_{[\mu]}^{b-\adm} / \Spd \Qpbreve([\mu])$ is geometrically connected.

\end{proposition}
\begin{proof}
For $G$ reductive, the statement that $\Gr_{[\mu]}^{b-\adm}$ is geometrically connected is established in the proof of \cite[Proof of Theorem 3.2]{gleason-lourenco:connectedness}: the argument given there first establishes that the open cells $\Gr_{[\mu]}^{b-\adm}$ are connected then deduces it for the closed cells $\Gr_{\leq [\mu]}^{b-\adm}$ by dimension considerations --- note that, in \cite[Proof of Theorem 3.2]{gleason-lourenco:connectedness}, our $\Gr_{[\mu]}$ is denoted $\Gr_{G,\mu}^\circ$ and our $\Gr_{\leq[\mu]}$ is denoted $\Gr_{G,\mu}$. In the general case, we choose a Levi decomposition $G = MU$ and we will show that $\Gr_{[\mu], C}^{b-\adm}$ (for $C$ be an algebraically closed non-archimedean extension of $\breve{\QQ}_p([\mu])$) is connected by showing that it admits a map from  the connected diamond 
\[ U(\mbb{B}_\dR^+/\Fil^m\mbb{B}_\dR) \times_{\Spd\+ C} \Gr_{[\mu_M],C}^{b-\adm} \]
that is a surjection on underlying topological spaces. 

Let $s \colon M \to G$ be the inclusion and $\pi \colon G \to G/U \cong M$ the quotient. Since $B(M) \cong B(G)$, we can (and do) assume that $b \in M(\breve{\QQ}_p)$. We first establish 
\begin{enumerate}
    \item The action of $U(\mbb{B}_\dR^+)$ on $\Gr_{[\mu]}$ preserves $\Gr_{[\mu]}^{b-\adm}$; and 
    \item The action factors through $U(\mbb{B}_\dR^+/\Fil^m\mbb{B}_\dR)$ for $m$ sufficiently large. 
\end{enumerate}
Claim (1) follows from \cref{eq.bun-G-over-Levi-fiber-prod} of \cref{theorem.bun-g-b-structure}, which shows $\Bun_G^{[1]}=\Bun_M^{[1]} \times_{\Bun_M} \Bun_G$. Indeed, the action of $U(\mbb{B}_\dR^+)$ does not change the lattice after push-out to $M$, so it does not change the modified $M$-bundle. 

For (2), let $m$ be any positive integer such that the weights of $\mu$ on $\Lie\+ U$ are all less than $m$ for some (or equivalently any) cocharacter $\mu \in [\mu]$. It is enough to show that for any perfectoid space $S$ over $\Spd\, C$, 
\[ u \in U(\Fil^m\mbb{B}_\dR^+(S)):=\ker U(\mbb{B}_\dR^+(S)) \rightarrow U(\mbb{B}_\dR^+(S)/\Fil^m \mbb{B}_\dR^+(S)),\]
and $\mc{L} \in \Gr_{[\mu]}(S)$, that $u\mc{L} = \mc{L}$. As two $\mbb{B}_\dR^+$-lattices are equal if only if they are equal after restricting to all geometric points, we can assume $S = \Spd\+ C$. In this case $\mc{L} = g \xi^\mu \mc{L}_0$ for $\mc{L}_0 = \omega_\std \otimes_{\QQ_p} \mbb{B}_\dR^+$ the trivial $\mbb{B}_\dR^+$-lattice, $\xi$ a generator of $\Fil^1 \mbb{B}_\dR^+(C)$, some $\mu \in [\mu]$, and $g \in G(\mathbb{B}^+_\dR(C))$. Since $G(\mathbb{B}^+_\dR(C))$ normalizes $U(\Fil^m\mbb{B}_\dR^+(C))$, we may assume $g=1$. Then, 
    \[ u \xi^\mu \mc{L}_0 = \xi^\mu \mc{L}_0 \iff u \in U(\mbb{B}_\dR^+(C)) \cap \xi^{\mu}U(\mbb{B}_\dR^+(C))\xi^{-\mu}, \]
but the condition on $m$ ensures that $U(\Fil^m \mbb{B}_\dR^+(C)) \subseteq U(\mbb{B}_\dR^+(C)) \cap \xi^{\mu}U(\mbb{B}_\dR^+(C))\xi^{-\mu}$. 

From (1) and (2), the action of $U(\mbb{B}_\dR^+)$ and the inclusion $\Gr_{[\mu_M]}^{b-\adm} \to \Gr_{[\mu]}^{b-\adm}$ induces a map of diamonds over $\Spd\+ C$
    \[ f \colon U(\mbb{B}_\dR^+/\Fil^m \mbb{B}_\dR^+) \times_{\Spd\+ C} \Gr_{[\mu_M],C}^{b-\adm} \to \Gr_{[\mu],C}^{b-\adm}. \]
We can check $|f|$ is surjective by checking on geometric points (e.g.\ using the description of $|\Gr_{[\mu]}^{b-\adm}|$ as equivalence classes of geometric points \cite[Proposition 12.7]{scholze:etale-cohomology-of-diamonds}); surjectivity on geometric points is clear since any element of $\Gr_{[\mu]}^{b-\adm}(C)$ is 
    \[ g\xi^{\mu_M} \mc{L}_0 = u (m\xi^{\mu_M} \mc{L}_0) \]
for some $\mu_M \in [\mu_M]$ and some $g \in G(\mbb{B}_\dR^+(C))$ with Levi decomposition $g = um, u \in U(\mbb{B}_\dR^+(C)), m \in M(\mbb{B}_\dR^+(C))$.

Each of $U(\mbb{B}_\dR^+/\Fil^m \mbb{B}_\dR^+)$ and $\Gr_{[\mu_M], C}^{b-\adm}$ is connected and cohomologically smooth over $\Spd\+ C$, and so the product is connected by \cref{lemma:product-connected-diamonds}.

\end{proof}

\begin{lemma}\label{lemma:product-connected-diamonds}
Let $C$ be an algebraically closed non-archimedean extension of $\QQ_p$ and $X,Y$ locally spatial diamonds that are $\ell$-cohomologically smooth over $\Spd\+ C$ and connected. Then the product $X \times_{\Spd\+C} Y$  is connected. 
\end{lemma}
\begin{proof}
We first note that, if $f:S \rightarrow \Spd\+ C$ is cohomologically smooth, then, since $Rf^! \mathbb{F}_\ell$ is invertible, 
\[ \Hom(Rf^! \mathbb{F}_\ell, Rf^! \mathbb{F}_\ell)=\Hom(\mathbb{F}_{\ell,S}, \mathbb{F}_{\ell,S}) =  \Cont(\pi_0(S), \mathbb{F}_\ell). \]
as in the proof of \cite[Corollary 4.11]{hansen:Harris-conjecture}. On the other hand, by the defining adjunction,
\[ \Hom(Rf^! \mathbb{F}_{\ell}, Rf^! \mathbb{F}_{\ell}) = \Hom (Rf_! Rf^! \mathbb{F}_\ell, \mathbb{F}_\ell). \]
Note that $Rf_! Rf^! \mathbb{F}_\ell$ is concentrated in degree $\leq 0$ since, for any $K$ in degree $>0$, 
\begin{align*} \Hom (Rf_! Rf^! \mathbb{F}_\ell, K)=\Hom (Rf^! \mathbb{F}_\ell, Rf^!K)&=\Hom(Rf^!\mathbb{F}_\ell, Rf^!\mathbb{F}_\ell \otimes f^*K)\\
&=\Hom(\mathbb{F}_{\ell,S}, f^*K)=0 \end{align*}
where the second equality is by \cite[Proposition 23.12-(i)]{scholze:etale-cohomology-of-diamonds}, the third equality is because $Rf^!\mathbb{F}_\ell$ is invertible, and the final equality with zero is because $f^* K$ is concentrated in degree $>0$. Thus we find 
\[ \Cont(\pi_0(S), \mathbb{F}_\ell) =\Hom(Rf_! Rf^! \mathbb{F}_\ell, \mathbb{F}_\ell)= \Hom(R^0f_! Rf^! \mathbb{F}_\ell, \mathbb{F}_\ell).\]
It follows that $S$ is connected if and only if $R^0f_! Rf^! \mathbb{F}_\ell=\mathbb{F}_\ell$.

We now consider the cartesian diagram:
\[\begin{tikzcd}
	{X\times_{\Spd\+C} Y} & X \\
	Y & {\Spd\+ C}
	\arrow["p", from=1-1, to=1-2]
	\arrow["q"', from=1-1, to=2-1]
	\arrow["f"{description}, from=1-1, to=2-2]
	\arrow["a", from=1-2, to=2-2]
	\arrow["b"', from=2-1, to=2-2]
\end{tikzcd}\]

We first note that 
\[ Rf^!\mbb{F}_\ell = Rp^!Ra^!\mbb{F}_\ell = p^\ast Ra^!\mbb{F}_\ell \otimes Rp^!\mbb{F}_\ell = p^\ast Ra^!\mbb{F}_\ell \otimes  q^\ast Rb^!\mbb{F}_\ell\]
Indeed, the second equality follows from \cite[Proposition 23.12-(i)]{scholze:etale-cohomology-of-diamonds} while the third follows from  \cite[Proposition 23.12-(iii)]{scholze:etale-cohomology-of-diamonds}. 

Applying this computation and the K\"unneth formula \cite[Proposition 4.2]{hansen:Harris-conjecture}, we find
\[ Rf_! Rf^! \mathbb{F}_\ell = (Ra_! Ra^! \mathbb{F}_\ell) \otimes (Rb_! Rb^! \mathbb{F}_\ell). \]
Since, as argued above,  all of these complexes are concentrated in degree $\leq 0$, passing to $H^0$ gives 
\[ R^0 f_! Rf^! \mathbb{F}_\ell = (R^0a_! Ra^! \mathbb{F}_\ell) \otimes (R^0b_! Rb^! \mathbb{F}_\ell). \]
Since $X$ and $Y$ are connected, the right-hand side is $\mathbb{F}_\ell \otimes \mathbb{F}_\ell=\mathbb{F}_\ell$, thus we conclude that $X \times_{\Spd\+ C} Y$ is also connected. 
\end{proof}

\subsection{The Hodge locus of a tensor}\label{ss.hodge-locus}

\begin{definition}\label{def.Hodge-locus}
Let $b \in G(\breve{\mbb{Q}}_p)$ and $\mc{G} = (W_b, \mc{L})$ be a $b$-rigidified $G$-admissible pair on a locally spatial diamond $S$ over $\breve{\QQ}_p$. Suppose that $V$ is a representation of $G$ and that $\ell \subseteq V_{\breve{\mbb{Q}}_p}$ underlies a one-dimensional isocrystal $\ell\cong \Qpbreve(n)=(\breve{\mbb{Q}}_p, p^{-n} \varphi_{\breve{\QQ}_p})$ of $W_b(V)$. Then, the \emph{Hodge locus of $\ell$}, written $\Hdg(\ell)$, is the locus in $S$ where $\ell$ underlies a sub-admissible pair. More precisely, if $Y/\Qpbreve$ is a perfectoid space, then 
    \[ \Hdg(\ell)(Y) = \{ f \colon Y \to S \colon \ell \otimes_{\breve{\QQ}_p} \Fil^n \mbb{B}_\dR|_{Y_v} \subseteq f^\ast\mc{L}(V) \}, \]
where the comparison in brackets is inside of $V \otimes \mbb{B}_\dR(Y) = f^\ast\mc{L}(V) \otimes \mbb{B}_\dR(Y)$. 
\end{definition}

The condition on $\ell$ above is equivalent to requiring that $(\ell, \ell \otimes_{\breve{\QQ}_p} \Fil^n \mbb{B}_\dR|_{Y_v}) \cong (\breve{\QQ}_p(n), \Fil^n\mbb{B}_{\dR,Y_v}) \to (W_b(V), f^\ast\mc{L}(V))$ is a morphism in $\AdmPair^\circ(Y)$, and since such morphisms are strict in the category of $\mbb{B}_\dR^+$-latticed vector spaces (by the proof of \cref{theorem.neutral-adm-pair-tannakian}), we could equally well describe the Hodge locus as 
    \[ \Hdg(\ell)(Y) = \{ f\colon Y \to S \colon \left(\ell, (\ell \otimes_{\breve{\QQ}_p} \mbb{B}_\dR|_{Y_v}) \cap f^\ast\mc{L}(V)\right) \in \AdmPair^\circ(Y) \}.\]

\begin{remark}
The Hodge locus of $\ell$ is independent of the representative of $b$ of $[b]$ in the following sense: if $b' = g b \sigma(g)^{-1}$ for $g \in G(\breve{\QQ}_p)$, the isomorphism $\Gr_{G}^{b-\adm} \xrightarrow{\sim} \Gr_G^{b'-\adm}, \mc{L} \mapsto g\mc{L}$ restricts to an isomorphism $\Hdg(\ell) \xrightarrow{\sim} \Hdg(g \cdot \ell)$. 
\end{remark}

The proofs of the following two statements are almost identical to the proofs of the corresponding statements for Hodge-Tate loci in \cref{ss.hodge-tate-locus-hs}.
\begin{lemma}\label{lemma:ap-hodge-locus-closed}
The Hodge locus $\Hdg(\ell) \subseteq S$ is a closed subdiamond. 
\end{lemma}

\begin{lemma}\label{lemma:ap-rigid-analytic-Hodge-locus}
Let $[\mu]$ be a conjugacy class of cocharacters of $G_{\overline{\mathbb{Q}}_p}$, let $L/\breve{\mathbb{Q}}_p([\mu])$ be a $p$-adic field, let $S/\Spa\+ L$ be a seminormal rigid analytic variety, and let $\mc{G}$ be a $b$-rigidified admissible pair of type $[\mu]$ on $S^\diamond$. Then, for $\ell$ as in \cref{def.Hodge-locus}, $\Hdg(\ell) \subseteq S^\diamond$, which is closed by \cref{lemma:ap-hodge-locus-closed}, is also rigid analytic (equivalently, by \cref{lemma.closed-rigid-analytic-subdiamonds-zariski-closed}, it corresponds to a Zariski closed subset of $S$). 
\end{lemma}

\begin{remark}
The notions introduced in this section are modeled on analogous ideas from complex geometry: given a pure variation of $\QQ$-Hodge structures $V$ of weight $2n$ and a global section $v$ of $V$, the Hodge locus $\Hdg(v)$ is the locus where $v$ underlies a one dimensional sub-variation of Hodge structures of type $(n,n)$ (equivalently, where $v$ is in the $n$-th step of the Hodge filtration of $V$), i.e. the locus where $v$ is a Hodge tensor. The importance of Hodge tensors in this classical context is that, for the Hodge structure on the cohomology of a smooth proper variety over $\mathbb{C}$, they are expected to correspond to algebraic cycles. 
\end{remark}

\subsection{Special subdiamonds and subvarieties}

Suppose $H/\mbb{Q}_p$ and $G/\mbb{Q}_p$
are connected linear algebraic groups, $b_H \in H(\breve{\mbb{Q}}_p)$, and $b_G \in G(\breve{\mbb{Q}}_p)$. Suppose given $f: H \rightarrow G$ and $g \in G(\breve{\mbb{Q}}_p)$ such that $g f(b_H) \sigma(g^{-1})=b_G.$ Composing push-out by $f$ with the Hecke action of $g$, we obtain induced functoriality maps
\begin{equation} \label{eq.funct-map-adm-loc}\Gr_{H}^{b_H-\adm} \rightarrow \Gr_{G}^{b_G-\adm}, \qquad  \mc{L} \mapsto g\cdot(\mc{L} \times^H G). \end{equation}
Explicitly, if $\mc{L} \in \Gr_H^{b_H-\adm}(Y)$ for $Y \in \Perfd_{\breve{\QQ}_p}$ then the value of $g\cdot(\mc{L} \times^H G)$ at $V \in \Rep\+ G$ is the $\mbb{B}_\dR^+$-lattice $g\mc{L}(f^\ast V)$, i.e.\ the image of $\mc{L}(f^\ast V)$ under the action of $g \otimes 1$ on $V_{\breve{\QQ}_p} \otimes_{\breve{\QQ}_p} \mbb{B}_\dR$.

We are particularly interested in the case when $H$ is a closed subgroup of $G$. This is motivated by consideration of reduction of structure group to the motivic Galois group of an admissible pair as in \cref{ss.G-structure-ap}: Let $Y / \Spd\+ \breve{\mbb{Q}}_p$ be a locally spatial diamond, and suppose we are given a neutral $G$-admissible pair $\mc{G}$ on $Y$.  Let $y$ be a geometric point of $Y$, and fix a trivialization of the \'{e}tale fiber functor $\omega_{y}\circ \mc{G}$ on $\Rep\+ G$. This identifies $\MG(\mc{G},y)$ with a closed subgroup of $H\leq G$. The canonical (in the sense of \Iref{ss.canonical-G-structure}) reduction of structure group of $\mc{G}$ is then to $H$, and we write the resulting neutral $H$-admissible pair as $\mc{H}$. After we choose a $b$-rigidification of $\mc{G}$, the reduction $\mc{H}$ induces a factorization of the classifying map: 

\begin{lemma}\label{lemma.ap-red-sg}Notation as above, suppose $\mc{G}$ is $b$-rigidified, i.e.\ induced by a map $f:Y \rightarrow \Gr_G^{b-\adm}$. Then
\begin{enumerate}
    \item There is an an element $b_H \in H(\breve{\mbb{Q}}_p)$ such that $[b_H] \subseteq [b]$ and $f$ factors through a functoriality map $f': Y \rightarrow \Gr_H^{b_H-\adm}$ as above.
    \item If $Y$ is furthermore a geometrically connected rigid analytic diamond over a $p$-adic field $\Spd\+ L$ then any $f'$ as in (1) factors through an open Schubert cell $\mr{Gr}^{b_H-\adm}_{[\mu_H]}$ with $\mu_H$ a cocharacter of $H$ contained in $[\mu]$.
\end{enumerate}
The pull-back of the universal neutral $H$-admissible pair by such an $f'$ is $\mc{H}$. 
\end{lemma}
\begin{proof}
The isomorphism class of $\omega_{\Isoc} \circ \mc{H}$ can be $b_H$-rigidifed for some $b_H$.  Since the push-out to $G$ is isomorphic to the $G$-isocrystal $b$, we deduce that $[b_H] \subseteq [b]$ and we may choose a $g$ that $\sigma$-conjugates $b_H$ to $b$ to obtain the factorization $f'$ in (1). The factorization through a Schubert cell in the geometrically connected rigid analytic case (2) follows from \cref{theorem.rigid-an-has-type}. 
\end{proof}

\begin{proposition}\label{prop:funct-closed-subdiamond}
Let $H \leq G$ be a closed subgroup, $b_H$, $b_G$, and $g$ as at the start of the subsection. Then the functoriality map identifies $\Gr_{H}^{b_H-\adm}$ with a closed subdiamond of $\Gr_{G}^{b_G-\adm}$. 
\end{proposition}

\begin{remark}
\newcommand{\diag}{\mr{diag}}
The functoriality map is induced by a map 
\[ \Gr_{H, \breve{\mbb{Q}}_p} \rightarrow \Gr_{G, \breve{\mbb{Q}}_p}. \]
When $H$ is a subgroup, this map is an injection and, taking a faithful representation of $G$, we find that the pre-image of the $b_G$-admissible locus is the $b_H$-admissible locus. If $H$ is reductive, then the map on affine Grassmannians is a closed immersion, so that one obtains \cref{prop:funct-closed-subdiamond} immediately in this case. The most interesting case is thus when $H$ is not reductive. A good example to keep in mind is  $G=\GL_2$, $H$ the upper-triangular Borel, $b_H=b_G=\diag(p,1)$, and $g=1$. If we restrict to the closed Schubert cell for $\mu=\diag(t^{-1},1)$ in $\Gr_G$ and its preimage in $\Gr_H$, then the map on Grassmannians is identified with the open inclusion $\mbb{A}^1 \rightarrow \mbb{P}^1$, but the restriction to the admissible loci is the isomorphism $\mbb{A}^1 \rightarrow \mbb{A}^1$. 
\end{remark}
\begin{proof}
It suffices to prove the proposition when $g = 1$. Fix a line $\ell$ in a representation $W$ of $G$ such that $H=\Stab(\ell)$, and consider $\ell_{\breve{\QQ}_p}$ as a one dimensional isocrystal via $ b_H$. Consider the commutative diagram
\[\begin{tikzcd}
	{\Gr_{H}^{b_H-\adm}} && {\Gr_{G}^{b_G-\adm}} \\
	& {\Hdg(\ell_{\breve{\QQ}_p})}
	\arrow[hook, from=1-1, to=1-3]
	\arrow["f"', from=1-1, to=2-2]
	\arrow[hook, from=2-2, to=1-3]
\end{tikzcd}\]
The top horizontal map is the functoriality map; this factors over the locus of $\Gr_{G}^{b_G-\adm}$ where $\ell_{\breve{\QQ}_p}$ underlies a sub-neutral admissible pair of $\mc{G}(W)$, which is precisely $\Hdg(\ell_{\breve{\mbb{Q}}_p})$. We claim there is a section $s \colon \Hdg(\ell_{\breve{\QQ}_p}) \to \Gr_{H}^{b_H-\adm}$ of $f$ in the category of diamonds over $\Gr_{G}^{b_G-\adm}$. Given this claim, the map $s$ is a surjection (since it is a section of an injection) and it is an injection since the composition to $\Gr_{G}^{b_G-\adm}$ is. This finishes the proof since then $\Gr_H^{b_H-\adm}$ is identified with the subdiamond $\Hdg(\ell_{\breve{\QQ}_p})$, which is closed by \cref{lemma:ap-hodge-locus-closed}.

To construct $s$, observe that the neutral $G$-admissible pair $\mc{G}|_{\Hdg(\ell_{\breve{\QQ}_p})}$ admits a reduction of structure group from $G$ to $H$ which agrees with the canonical reduction of structure of $\mc{G}|_{\Gr_{H}^{b_H-\adm}} = \mc{H} \times^H G$; the classifying map $s \colon \Hdg(\ell_{\breve{\QQ}_p}) \to \Gr_H$  for the neutral $H$-admissible pair induced by this reduction of structure group has image in $\Gr_{H}^{b_H-\adm}$ and gives the section $s$. 
\end{proof}

\begin{definition}\label{defn.special-subdiamond} A \emph{special subdiamond} of $\Gr_{G}^{b_G-\adm}$ is a closed subdiamond in $\Gr_{G}^{b_G-\adm}$ identified with $\Gr_{H}^{b_H-\adm}$ as above. 
\end{definition}

Within certain special subdiamonds, we can isolate a natural rigid analytic part. 

\begin{lemma}\label{lem.special-subvariety-closed} Suppose $[\mu_H]$ is a conjugacy class of cocharacters of $H$ such that the weights of $[\mu_H]$ on $\Lie\+ H$ are $\leq 1$. Then $\Gr_{[\mu_H]}^{b_H-\adm}$ is a geometrically connected rigid analytic diamond over $\breve{\mbb{Q}}_p([\mu_H])$, closed in $\Gr_{H,\breve{\mbb{Q}}_p([\mu_H])}^{b_H-\adm}$. \end{lemma}
\begin{proof}
    $\Gr_{H,[\mu_H]}^{b_H-\adm}$ is rigid analytic as an open subdiamond of the rigid analytic diamond $\Gr_{[\mu_H]}=\Fl_{[\mu_H^{-1}]}^\diamond$ (the equality by \cref{theorem.bdr-aff-grass-properties}-(5)). It is geometrically connected by \cref{prop.adm-connected} and closed by \cref{theorem.bdr-aff-grass-properties}-(2) since the image of $[\mu_H]$ in the quotient $M_H$ of $H$ by its unipotent radical is minuscule.  
\end{proof}

In the setting of \cref{lem.special-subvariety-closed}, the functoriality map \cref{eq.funct-map-adm-loc} allows us to view $\Gr_{[\mu_H]}^{b_H-\adm}$ as a closed subdiamond of $\Gr_{G, \breve{\mbb{Q}}_p([\mu_H])}^{b_G-\adm}$ --- indeed, it is closed in $\Gr_{H,  \breve{\mbb{Q}}_p([\mu_H])}^{b_H-\adm}$, which is closed in $\Gr_{G,  \breve{\mbb{Q}}_p([\mu_H])}^{b_G-\adm}$ by \cref{prop:funct-closed-subdiamond}. It is thus a closed geometrically connected rigid analytic subdiamond of $\Gr_{G,  \breve{\mbb{Q}}_p([\mu_H])}^{b_G-\adm}$, which even factors through $\Gr_{[\mu_G],\breve{\mbb{Q}}_p([\mu_H])}^{b-\adm}$ where $[\mu_G]$ denotes the induced conjugacy class of cocharacters of $G$. 

\begin{definition}\label{def.special-subvariety-ap} A \emph{special subvariety} of $\Gr_{G, C}^{b_G-\adm}$ or $\Gr_{[\mu_G], C}^{b_G-\adm}$ (where $C$ is an algebraically closed non-archimedean extension of $\breve{\QQ}_p$) is a closed connected rigid analytic subdiamond identified as above with some $\Gr_{[\mu_H], C}^{b_H-\adm}$ for $[\mu_H]$ such that the weights of the adjoint representation on $\Lie\+ H$ are $\leq 1$.
\end{definition}

Each special subvariety is naturally a closed subset of a special subdiamond, but let us emphasize that a special subvariety is typically much smaller than this special subdiamond except when $H$ is a split torus 
(see \cref{example.torus-special-sub}).

\begin{example}\label{example.torus-special-sub}
If $T$ is a torus then everything can be described explicitly in terms of the cocharacter group $X_\ast(T)$. Kottwitz constructs an isomorphism $\kappa_T \colon B(T) \to X_\ast(T)_{\mr{Gal}(\ol{\QQ}_p/\QQ_p)}$ \cite[\S 2.5]{kottwitz:isocrystals}. If $\mu \in X_\ast(T)$ has image $\mu^\# \in X_\ast(T)_{\mr{Gal}(\ol{\QQ}_p/\QQ_p)}$ then $B(T,[\mu])$ consists of the single class represented by any $b_\mu$ with $\kappa_T([b_{\mu}]) = \mu^\#$. Thus 
    \[ \Gr_{[\mu]} = \Gr_{[\mu]}^{b_{\mu^{-1}}-\adm} \cong \Spd\, \QQ_p([\mu]). \]
It follows that, for $b \in T(\breve{\QQ}_p)$,
    \[ \Gr_T^{b-\adm} = \bigsqcup_{\{\mu \colon \mu^\# = \kappa_T([b])\}} \Gr_{[\mu^{-1}]}. \]
For a general $G$, the special subvarieties of $\Gr_{G,C}^{b_G-\adm}$ arising from a torus $T \leq G$ are called special points. 
\end{example}

\begin{remark}
The collection of special subvarieties attached to a fixed pair $(H,[\mu_H])$ only depends on the conjugacy class of this data under $G(\mbb{Q}_p)$. However, the special subvarieties attached to a fixed $(H,[\mu_H])$ typically vary non-trivially in locally profinite families by varying the element $g$ that $\sigma$-conjugates $b_H$ to $b_G$, i.e.\ by the action of $G_b(\mbb{Q}_p)$. For example, in the height two Lubin-Tate case, where the admissible locus is all of $\mathbb{P}^1_{\Qpbreve}$, and for $H=F^\times$ for $F/\mathbb{Q}_p$ a degree two extension, the associated special subvarieties are special points corresponding to admissible pairs (or, in this case, height two formal groups up to isogeny) with endomorphisms by $F$. The set of all such can be identified with $\mathbb{P}^1(F)$ by taking the orbit of any one under $G_b(\mathbb{Q}_p)$, the units in the non-split quaternion algebra over $\mathbb{Q}_p$. 
\end{remark}

\begin{remark}
It is natural to consider, for $[\mu]$ any conjugacy class of cocharacters of $G_{\overline{\mbb{Q}}_p}$, the intersection of $\Gr_H^{b_H-\adm}$ (via the functoriality map) with $\Gr_{[\mu]}^{b_G-\adm}$. The special subvarieties of $\Gr_{[\mu]}^{b_G-\adm}$ are precisely the connected components of such an intersection which are rigid analytic. In general, if $H$ is reductive, then the intersection is simple: the finitely many conjugacy classes of cocharacters of $H_{\overline{\mbb{Q}}_p}$ refining $[\mu]$ are not comparable in the Bruhat order, so we have (after a finite extension so the right hand side makes sense)
\[ \Gr_{H,\Qpbreve([\mu])}^{b_H-\adm} \cap \Gr_{[\mu]}^{b_G-\adm} = \bigsqcup_{[\mu_H] \subseteq [\mu]} \Gr_{[\mu_H]}^{b_H-\adm} \]
However, when $H$ is not reductive, the situation is not so nice. In particular, it can happen that $\Gr_{[\mu_H]}^{b_H-\adm}$ is not closed in $\Gr_H^{b_H-\adm} \cap \Gr_{[\mu_G]}^{b_G-\adm}$ (see \cref{example.non-closed-special-in-schubert-cell}). This explains why we do not define a special subdiamond within an open cell outside of the minuscule case giving special subvarieties. 
\end{remark}

\begin{example}\label{example.non-closed-special-in-schubert-cell} 
Let $G=\GL_3$, let $H$ be the standard parabolic subgroup preserving the subspace $\langle e_1, e_2 \rangle$, and let $M=\GL_2 \times \GL_1$ be its standard Levi subgroup. Consider the following lattices in $(\mbb{B}_\dR)^3$ parameterized by $\mbb{D}_C = \Spa\, C\langle z \rangle$ (with $C$ an algebraically closed non-archimedean field) 
\[ \mc{L}_z = \langle e_1, e_2 + \frac{z}{\xi}e_1, e_3 + \frac{1}{\xi}e_1 \rangle=\langle \xi e_3, e_2 - z e_3, \frac{1}{\xi}(\xi e_3 + e_1) \rangle\]
where $\xi$ is a generator of $\Fil^1 \mbb{B}^+_\dR(C)$ and we construct the lattice as in \cref{example:non-constant-type-family}. First, we observe that $\mc{L}_z \in \Gr_{G}(\mbb{A}^1_C)$ is admissible for $b=1$, i.e.\ that $(\breve{\QQ}_p^3, \mc{L}_z)$ is a neutral admissible pair on $\mbb{A}^1_C$. Equivalently, we need to check that after pulling back to a geometric point $\Spa\+ C \to \mbb{D}_C$, i.e.\ specializing to $z = \alpha \in C$, the modification $(\mc{O}^3)_{\mc{L}_\alpha}$ of the trivial rank 3 vector bundle on $\FF_{C^\flat}$ by $\mc{L}_\alpha$ is again trivial. But the first presentation of $\mc{L}_\alpha$ shows that the standard flag on $\breve{\QQ}_p^3$ induces a filtration on the $\mbb{B}_\dR^+(C)$-latticed isocrystal $(\breve{\QQ}_p^3,\mc{L}_\alpha)$ whose graded pieces (isomorphic to $(\breve{\QQ}_p, \mbb{B}_\dR^+(C))$) are admissible; thus $(\mc{O}^3)_{\mc{L}_\alpha}$ is trivial as it is an iterated extension of trivial vector bundles on $\FF_{C^\flat}$. Second, observe that $\mc{L}_z \in \Gr_{[\mu]}(C)$ for $\mu=(t, 1, t^{-1})$ as is immediate from the second presentation. 

Now the admissible pair $(\breve{\QQ}_p^3, \mc{L}_\alpha)$ has motivic Galois group contained in $H \leq G$ (as $\breve{\QQ}_p^2 \subseteq \breve{\QQ}_p^3$ induces a rank 2 sub-admissible pair) and so $\mc{L}_\alpha \in \Gr_{H}^{1-\adm}(\mbb{D}_C)$.  When $\alpha \neq 0$, we also have 
\[  \mathcal{L}_\alpha = \langle \xi e_2, \frac{1}{\xi}(\xi e_2 + \alpha e_1), e_3 - \alpha^{-1}e_2 \rangle, \]
which shows $\mc{L}_\alpha \in \Gr_{[\mu_H],H}^{1-\adm}(C)$ for $\mu_H=(t, t^{-1}) \times 1$; while for $\alpha=0$ the image of $\mc{L}_0$ under $\Gr_H(C) \to \Gr_M(C)$ lands in the Schubert cell of the trivial cocharacter. In particular, $\mc{L}_0 \not\in \Gr_{[\mu_H],H}(C)$, and thus $\Gr_{[\mu_H]}^{1-\adm}$ is not closed in $\Gr_H^{1-\adm} \cap \Gr_{[\mu]}^{1-\adm}.$

\end{example}

\section{Infinite level moduli spaces}\label{s.moduli-infinite-level}

In this section we define the infinite level moduli spaces appearing in our Ax-Lindemann theorem, and discuss their basic properties as well as their relation to the moduli spaces of admissible pairs and $p$-adic Hodge structures we have introduced already via period maps. In \cref{ss.inf-level-special-sv}, we define the two notions of special subvarieties that arise in our Ax-Lindemann theorem and, in \cref{ss.duality}, we give the duality theorem for infinite level spaces in the basic case (the novelty in the latter, if any, is the extension to the non-reductive case).

\subsection{The infinite level moduli problem}\label{ss.infinite-level-moduli}
Let $G$ be a connected linear algebraic group over $\mbb{Q}_p$, and let $b \in G(\breve{\mbb{Q}}_p)$. We consider the infinite level moduli problem $\mc{M}_b/\Spd\+ \breve{\mbb{Q}}_p$ classifying $b$-rigidified admissible pairs equipped also with a trivialization of the induced $G(\QQ_p)$-local system:

\[ \mc{M}_b(S)  = \left\{ (\mc{G}, \rho_\Isoc, \rho_\et) \ \middle\vert \begin{array}{c} \mc{G}\in G\text{-}\AdmPair^\circ(S),\, \rho_{\Isoc}: W_b \cong \omega_\Isoc \circ \mc{G}, \\ \textrm{ and } \rho_\et: \omega_\std \otimes \ul{\QQ_p} \cong \omega_\et \circ \mc{G} \end{array} \right\} / \sim.  \]

The natural forgetful map $\pi_{\mc{L}_\et}: \mc{M}_b \rightarrow \Gr_G^{b-\adm}$ is canonically identified with the $G(\mbb{Q}_p)$-torsor $\ul{\Isom}(\omega_\std \otimes \ul{\QQ_p}, \omega_\et \circ \mc{G})$. In particular, $\mc{M}_b$ is a locally spatial diamond. For $[\mu]$ a conjugacy class of cocharacters of $G_{\overline{\mbb{Q}}_p}$, we write 
\[ \mc{M}_{b,[\mu]} := \mc{M}_{b} \times_{\Gr_{G, \breve{\mbb{Q}}_p}^{b-\adm}} \Gr_{[\mu]}^{b-\adm}, \]
a locally closed subdiamond of $\mc{M}_{b, \Qpbreve([\mu])}$.

\subsection{Interpretation in terms of modifications}\label{ss.modification-moduli}
The space $\mc{M}_b$ also has an interpretation in terms of $G$-bundles on the Fargues-Fontaine curve. If $S$ is a locally spatial diamond over $\QQ_p$, and $\mc{E}, \mc{E}' \in \Bun_G(S)$, a modification (at $\infty$) $\alpha \colon \mc{E} \dashrightarrow \mc{E}'$ is an isomorphism $\alpha \colon \mc{E}_{\FF\setminus \infty} \xrightarrow{\sim} \mc{E}'_{\FF\setminus \infty}$ with locally bounded poles. Now suppose $\alpha \colon \mc{E}_b \dashrightarrow \mc{E}_\triv$ is a modification. Restriction of $\alpha$ to the `punctured formal neighborhood' of $\infty \in \FF$ gives an isomorphism $\alpha_{\mbb{B}_\dR} \colon \mc{E}_b |_{\mbb{B}_\dR} \cong \mc{E}_\triv|_{\mbb{B}_\dR}$. We define a lattice 
    \[ \pi_{\mc{L}_{\et}}(\alpha) := \alpha_{\mbb{B}_\dR}^{-1}(\mc{E}_\triv|_{\mbb{B}_\dR^+}) \subseteq \mc{E}_{b}|_{\mbb{B}_\dR} = W_b \otimes_{\breve{\mbb{Q}}_p} \mbb{B}_\dR. \]
The pair $\mc{G}_\alpha = (W_b, \pi_{\mc{L}_{\et}}(\alpha))$ is then a $b$-rigidified admissible pair on $S$. Further, $\mc{G}_\alpha$ has canonical \'etale and isocrystalline trivializations since $\omega_\Isoc \circ \mc{G}_\alpha = W_b$, and $\omega_\et \circ \mc{G}_\alpha = H^0(\mc{E}_\triv) = \omega_\std  \otimes \ul{\QQ_p}$. Thus we see that 
    \[ \mc{M}_b(S) = \{\text{Modifications }\mc{E}_b \dashrightarrow \mc{E}_\triv \text{ on } \FF \textrm{ (over $S$)}\}. \]

\subsection{Period maps}\label{ss.period-maps}
We have a natural diagram of period maps
\begin{equation}\label{eq.first-period-diagram}\begin{tikzcd}
	& {\mc{M}_{b}} \\
	\\
	{\Gr_{G}} && {\Gr_{G}}
	\arrow["{\pi_{\mc{L}_\dR}}"{description}, from=1-2, to=3-3]
	\arrow["{\pi_{\mc{L}_\et}}"{description}, from=1-2, to=3-1]
\end{tikzcd}\end{equation}
where, in the moduli interpretation in terms of admissible pairs,
\begin{itemize}
    \item $\pi_{\mc{L}_\et}$ sends $(\mc{G}, \rho_\et, \rho_\Isoc)$ to the $\mbb{B}^+_\dR$-lattice on $\omega_\std$
    \[ \rho_\Isoc^* \mc{L}_\et = \rho_\Isoc^* ((\omega_\et \circ \mc{G}) \otimes_{\ul{\mbb{Q}_p}} \mbb{B}^+_\dR) \subseteq W_b \otimes_{\Qpbreve} \mbb{B}_\dR = \omega_\std \otimes_{\mbb{Q}_p} \mbb{B}_\dR.\]
    \item $\pi_{\mc{L}_\dR}$ sends $(\mc{G}, \rho_\et, \rho_\Isoc)$ to the $\mbb{B}^+_\dR$-lattice on $\omega_\std$
    \[ \rho_\et^* \mc{L}_\dR=\rho_\et^* ((\omega_\Isoc \circ \mc{G}) \otimes_{\Qpbreve} \mbb{B}^+_\dR) \subseteq \omega_\std \otimes_{\mbb{Q}_p} \mbb{B}_\dR. \]
\end{itemize}
As we have indicated in \cref{ss.infinite-level-moduli}, $\pi_{\mc{L}_\et}$ is a $G(\mbb{Q}_p)$-torsor over $\Gr_G^{b-\adm}$. 

If we fix a conjugacy class of cocharacters $[\mu]$ and restrict to $\mc{M}_{b,[\mu]}$ then there are Hodge and Hodge-Tate filtrations (of type $[\mu^{-1}]$ and $[\mu]$, respectively) induced by the Bialynicki-Birula map, and the period diagram can be refined and extended to:

\begin{equation}\label{eq.second-period-diagram}\begin{tikzcd}
	& {\mc{M}_{b,[\mu]}} \\
	\\
	{\Gr_{[\mu]}} && {\Gr_{[\mu^{-1}]}} \\
	\\
	{\Fl_{[\mu^{-1}]}} && {\Fl_{[\mu]}}
	\arrow["{\pi_{\mc{L}_\dR}}"{description}, from=1-2, to=3-3]
	\arrow["{\pi_{\mc{L}_\et}}"{description}, from=1-2, to=3-1]
	\arrow["\BB"{description}, from=3-1, to=5-1]
	\arrow["{\pi_\Hdg}"{description}, curve={height=-12pt}, from=1-2, to=5-1]
	\arrow["{\pi_\HT}"{description}, curve={height=12pt}, from=1-2, to=5-3]
	\arrow["\BB"{description}, from=3-3, to=5-3]
\end{tikzcd}\end{equation}

Of course $\pi_{\mc{L}_\et}$ factors through its image $\Gr_{[\mu]}^{b-\adm}$. In the basic case it is useful to keep track of an extended diagram accounting for the image also of $\pi_{\mc{L}_\dR}$: indeed, in this case the definitions and equivalence of basic admissible pairs and $p$-adic Hodge structures over a geometric point (see \Iref{s.admissible-pairs}) imply that $\rho_\et^*\mc{L}_\dR$ is a $p$-adic Hodge structure on the trivial $G(\QQ_p)$-local system over $\mc{M}_{b,[\mu]}$ (with weight grading determined by the slope morphism as usual), and then $\pi_{\mc{L}_\dR}$ is the classifying map so factors through $X_{[\mu]}$. Thus in the basic case the diagram can be refined to 

\[\begin{tikzcd}
	&& {\mc{M}_{b,[\mu]}} \\
	\\
	{\Gr_{[\mu]}} & {\Gr_{[\mu]}^{b-\adm}} && {X_{[\mu]}} & {\Gr_{[\mu^{-1}]}} \\
	&&&&& {} \\
	& {\Fl_{[\mu^{-1}]}} && {\Fl_{[\mu]}}
	\arrow["{\pi_{\mc{L}_\dR}}"{description}, from=1-3, to=3-5]
	\arrow["{\pi_{\mc{L}_\et}}"{description}, from=1-3, to=3-1]
	\arrow["\BB"{description}, from=3-1, to=5-2]
	\arrow["{\pi_\Hdg}"{description}, curve={height=-12pt}, from=1-3, to=5-2]
	\arrow["{\pi_\HT}"{description}, curve={height=12pt}, from=1-3, to=5-4]
	\arrow["\BB"{description}, from=3-5, to=5-4]
	\arrow[from=1-3, to=3-2]
	\arrow[hook', from=3-2, to=3-1]
	\arrow[from=1-3, to=3-4]
	\arrow[hook, from=3-4, to=3-5]
	\arrow[from=3-4, to=5-4]
	\arrow[from=3-2, to=5-2]
\end{tikzcd}\]

\begin{lemma}
    For $b$ basic, $\pi_{\mc{L}_\dR}$ is a $G_b(\mbb{Q}_p)$-torsor over $X_{[\mu]}$.
\end{lemma}
\begin{remark}
Implicit in this claim is the statement that the action of the group $G_b(\mbb{Q}_p)$ on $\mc{M}_{b,[\mu]}$ by change of trivialization of $\rho_\Isoc$ extends to an action of the topological constant sheaf $\ul{G_b(\mbb{Q}_p)}$. This is not completely clear via the moduli interpretation we have used via neutral admissible pairs, but follows immediately from the interpretation in terms of modifications of $G$-bundles on the Fargues-Fontaine curve in \cref{ss.modification-moduli} or from the formalism we will develop in Part III.  
\end{remark}
\begin{proof}
This follows from the definitions and \cref{eq.isomorphism-b-loc-classifying-stack}.
\end{proof}

\begin{remark} On geometric points, we have described all of these period maps explicitly in \Iref{ss.periods}. 
\end{remark}

\subsection{Special subvarieties}\label{ss.inf-level-special-sv}
Let $G/\mbb{Q}_p$ be a connected linear algebraic group, let $[\mu]$ be a conjugacy class of cocharacters of $G_{\overline{\mbb{Q}}_p}$ admitting a basic class $[b] \in B(G,[\mu^{-1}])$, and fix $b \in G(\breve{\mbb{Q}}_p)$ representing this class. Let $L/\breve{\mbb{Q}}_p([\mu])$ be a $p$-adic field, and let $C=\overline{L}^\wedge.$

For $\mc{M}:=\mc{M}_{b,[\mu]}$, we will now define some notions of special subvarieties of $\mc{M}_C$. Suppose $\iota: H \hookrightarrow G$ is a closed connected subgroup, $[\mu_H]$ is a conjugacy class of cocharacters of $H_{\overline{\mbb{Q}}_p}$ contained in $[\mu]$, $b_H \in H(\breve{\mbb{Q}}_p)$ represents the basic class in $B(H,[\mu_H^{-1}])$, and $g \in \breve{\mbb{Q}}_p$ is such that $g^{-1} b \sigma(g)=b_H$. Then push-out combined with the isomorphism $g: \mc{E}_{\iota(b_H)} \rightarrow \mc{E}_b$ induces functoriality maps
\[ \mc{M}_{b_H, [\mu_H]} \rightarrow \mc{M}_{b,[\mu], \breve{\mbb{Q}}_p([\mu_H])} \textrm{ and } \Gr_{[\mu_H]}^{b_H-\adm} \rightarrow \Gr_{[\mu], \breve{\mbb{Q}}_p([\mu_H])}^{b-\adm}, \]
and, if $b$ is basic,
\[  X_{[\mu_H]}\rightarrow X_{[\mu], \mbb{Q}_p([\mu_H])}. \]
These are compatible with all group actions and period maps. 

\begin{definition}[Special subvarieties at infinite level]\label{defn.special-subvarieties-inf-level} Let $\mc{M}=\mc{M}_{b,[\mu]}$.
\begin{enumerate}
\item A \emph{Hodge special subvariety} of $\mc{M}_{C}$ is a connected component of a closed subdiamond $\mc{M}_{b_H, [\mu_H], C}$ for $[\mu_H]$ such that $[\mu_H]$ has weights $\leq 1$ on $\Lie\+ H$. Equivalently, it a connected component of the preimage under $\pi_{\mc{L}_\et}$ of a special subvariety  $\Gr_{[\mu_H],C}^{b_H-\adm}$ of $\Gr_{[\mu],C}^{b-\adm}$ (see \cref{def.special-subvariety-ap}).
\item If $b$ is basic, a \emph{Hodge-Tate special subvariety} of $\mc{M}_{C}$ is a connected component  of a closed subdiamond $\mc{M}_{b_H, [\mu_H], C}$ for $[\mu_H]$ such that $[\mu_H]$ has weights $\geq -1$ on $\Lie\+ H$. Equivalently, it a connected component of the preimage of a special subvariety (see \cref{def.special-subvariety-pahs}) $X_{[\mu_H],C}$ of $X_{[\mu], C}$ under $\pi_{\mc{L}_\dR}$. 
\item A \emph{special subvariety} of $\mc{M}_{C}$ is a connected component of a closed subdiamond $\mc{M}_{b_H, [\mu_H], C}$ for $[\mu_H]$ minuscule (i.e.\ with weights in $[-1, 1]$ on $\Lie\+ H$). 
\end{enumerate}
\end{definition}

\begin{lemma}
If $b$ is basic, a closed connected subdiamond of $\mc{M}_C$ is a special subvariety if and only if it is both a Hodge special subvariety and Hodge-Tate special subvariety.  
\end{lemma}
\begin{proof}
The only subtlety is that a priori a subdiamond $Y$ could be a Hodge/Hodge-Tate special subvariety for two different $H$. However, for any fixed geometric point $y$ of $Y$, we can always take $H$ to be the motivic Galois group (corresponding to $\omega_y$) of the universal neutral $G$-admissible pair restricted to $Y$, or equivalently the Mumford-Tate group of the universal neutral $G$-$p$-adic Hodge structure restricted to $Y$ (both of which are identified with the same closed subgroup of $G$ using $\rho_\et$). 
\end{proof}

\begin{remark}If $[\mu]$ is minuscule (the local Shimura case as in the introduction), then Hodge special subvarieties and special subvarieties are the same; if $b$ is furthermore basic then these also agree with Hodge-Tate special subvarieties. Moreover, in this case it is not necessary to impose any condition on the weights of $[\mu_H]$ on $\Lie\+ H$ in the definitions since they hold already for the weights on $\Lie\+ G \supseteq \Lie\+ H$. The three definitions also agree when the subgroup $H$ appearing is reductive since then the three conditions on the weights of the adjoint representation are equivalent. In general, however, there are Hodge special subvarieties which are not Hodge-Tate special and, in this case, only one of the lattice period domains for $H$ is a rigid analytic variety instead of both --- see \cref{example.special-distinction} for an explicit example, which also illustrates how this can occur even when the larger group $G$ is reductive. 
\end{remark}

\begin{example}\label{example.special-distinction}
Let $E/\mbb{Q}_p$ be a degree 3 field extension, let $G$ be the restriction of scalars $\mr{Res}_{E/\mbb{Q}_p} \mr{GL}_{2}$, and let $H \leq G$ be the mirabolic subgroup, i.e.\ the subgroup consisting of matrices of the form 
\[ \begin{bmatrix} * & * \\ 0 & 1 \end{bmatrix}. \]
Abusing notation by writing $E^\times$ for the restriction of scalars of $\mbb{G}_{m,E}$ and $E$ for the additive group $E\otimes_{\mbb{Q}_p} \mbb{G}_{a,\mbb{Q}_p}$, we have $H=E^\times \ltimes E$. 
Fixing a trivialization $E^\times_{\overline{\mbb{Q}}_p}=\mbb{G}_m^3$, we take $\mu_H=\mu$ to be the cocharacter $t \mapsto (t, t, t^{-2})$. The basic elements are then $b_H=b=1$, and it is immediate that the associated Hodge special subvariety is not Hodge-Tate special.
\end{example}

\subsection{Basic duality}\label{ss.duality}
Let $G$ be a connected linear algebraic group over $\QQ_p$ and let $b \in G(\breve{\mbb{Q}}_p)$ be basic. Recall from \cref{ss.G-isocrystals} that we have the group $G_b$ of automorphisms of the associated $G$-isocrystal, and recall from the end of \cref{ss.BunG} the twisting isomorphism $F_b: \Bun_G \rightarrow \Bun_{G_b}$ sending $\mc{E} \in \Bun_G(S)$ to $\ul{\Isom}_{\FF_S}(\mc{E}_b,\mc{E})$. Given a modification $\mc{E}_\triv \dashrightarrow \mc{E}_b$ of $G$-bundles on $\FF_S$, we may apply this twisting isomorphism to obtain a modification of $G_b$-bundles on $\FF_S$, $\ul{\Isom}_{\FF_S}(\mc{E}_b, \mc{E}_\triv) \dashrightarrow \ul{\Isom}_{\FF_S}(\mc{E}_b, \mc{E}_b)$. Using the isomorphisms of \cref{lemma.BunG-twisting}, this is a modification of $G_b$-bundles 
\[ \mc{E}_{b_{G_b}^{-1}} \dashrightarrow \mc{E}_{\triv} \]
where here we write $b_{G_b}$ as in \cref{lemma.BunG-twisting} when we are viewing $b$ as an element of $G_b(\Qpbreve)$ via \cref{eq.Qpbreve-points-of-G-and-twist}. Using the moduli interpretation in \cref{ss.modification-moduli}, we find the above construction induces a map $\mc{M}_{b} \rightarrow \mc{M}_{b^{-1}_{G_b}}$. 

\begin{theorem}\label{theorem.duality} With notation as above, $b_{G_b}^{-1}$ is basic in $G_b(\Qpbreve)$, and the twisting construction is a $G(\mbb{Q}_p) \times G_b(\mbb{Q}_p)$-equivariant isomorphism $\mc{M}_{b}\cong \mc{M}_{b_{G_b}^{-1}}$ such that the de Rham lattice and \'{e}tale lattice  period maps are exchanged. Moreover, for each $[\mu]$ such that $[b]$ is the basic element in $B(G,[\mu^{-1}])$, this restricts to an isomorphism $\mc{M}_{b,[\mu]}\cong\mc{M}_{b_{G_b}^{-1},[\mu^{-1}]}$ that exchanges the Hodge and Hodge-Tate filtration period maps and exchanges Hodge and Hodge-Tate special subvarieties. 
\end{theorem}
\begin{proof}
It is more or less immediate that this is an isomorphism (the inverse is the opposite twisting map!). At the level of periods, this replaces an element of $G(\mbb{B}_\dR)$ with its inverse and thus exchanges the two lattice period maps (cf. \Iref{ss.periods}). This also explains why it exchanges the Schubert cell $[\mu]$ with $[\mu^{-1}]$, and the exchange of the Hodge and Hodge-Tate period maps is a consequence of the exchange of the lattice period maps. Finally, everything is naturally compatible with functoriality maps, thus the isomorphism exchanges Hodge and Hodge-Tate special subvarieties. 
\end{proof}

\begin{remark}
In the PEL Rapoport-Zink case as in, e.g., \cite{fargues:two-towers}, the isomorphism at the level of periods matrices is typically given as a transpose, i.e.\ $g \mapsto g^t$, instead of $g \mapsto g^{-1}$ as above. This is obtained by composing with the dual of the standard representation (so $g\mapsto g^{-1} \mapsto ((g^{-1})^t)^{-1}=g^t$), which is necessary in order to compare the homological normalization on both sides arising from the interpetation via the moduli of $p$-divisible groups.  
\end{remark}

\section{The bi-analytic Ax-Lindemann theorem}\label{s.bi-an-ax--lind-theorem}

In this section we state and prove the bi-analytic Ax-Lindemann theorem in full generality (i.e.\ without the minuscule assumption in \cref{main.ax-lindemann}). In \cref{ss.ax-lindemann-full-statement}, we introduce the necessary terminology for Zariski closures and analytic sets, then state the result (\cref{theorem.ax-lindemann-general}). The main subtlety compared to the minuscule case stated as \cref{main.ax-lindemann} in the introduction is that we need to allow for the two different notions of special subvariety introduced in \cref{ss.inf-level-special-sv} --- both appear because, for $H$ a non-reductive group and $[\mu_H]$ a conjugacy class of characters of $H_{\overline{\mbb{Q}}_p}$ admitting a basic element $b_H \in B(H,[\mu_H^{-1}])$, it is possible that only one of the two period domains $X_{[\mu_H]}$ and $\Gr_{[\mu_H]}^{b_H-\adm}$ is rigid analytic. 

Once the definitions are in place, much of the proof is straightforward using the Tannakian formalism we have set up and the results of Part I. However, there is a key technical ingredient we have not yet established: \cref{lemma:Hodge-generic-locus-for-rigid-analytic}, which gives the existence of a dense set of classical points in the Hodge generic locus for an admissible pair on the trivial $G$-isocrystal over a rigid analytic variety. 

In \cref{ss.proof-main-results}, we assume \cref{lemma:Hodge-generic-locus-for-rigid-analytic} and use it to prove the bi-analytic Ax-Lindemann theorem (\cref{theorem.ax-lindemann-general}) and \cref{main.hodge-tate-image} (the argument for the latter is essentially a subset of the argument for the former). We conclude in \cref{ss.generic-lemma} with the proof of \cref{lemma:Hodge-generic-locus-for-rigid-analytic}.

\subsection{Statement of the bi-analytic Ax-Lindemann theorem}\label{ss.ax-lindemann-full-statement}
Let $G/\mbb{Q}_p$ be a connected linear algebraic group, let $[\mu]$ be a conjugacy class of cocharacters of $G_{\overline{\mbb{Q}}_p}$ admitting a basic class $[b] \in B(G,[\mu^{-1}])$, and fix $b \in G(\breve{\mbb{Q}}_p)$ representing this class. Let $\mc{M}=\mc{M}_{b,[\mu]}$, the infinite level moduli space defined in \cref{ss.infinite-level-moduli}. Let $L/\breve{\mbb{Q}}_p([\mu])$ be a $p$-adic field, $C=\overline{L}^\wedge$, and $\mf{I} = \mr{Gal}(\ol{L}/L) = \mr{Aut}_{\mr{cont}}(C/L)$. 

We will use the notions of special subvarieties given in \cref{defn.special-subvarieties-inf-level}. We also need the following definition of Zariski closures and analytic sets.

\begin{definition}[Hodge/Hodge-Tate analytic Zariski closures and irreducible sets]\hfill
\begin{itemize}
    \item for $A \subset |\mc{M}_{C}|$, the \emph{Hodge (resp. Hodge-Tate) analytic Zariski closure} of $A$ is the intersection of all pre-images of closed rigid analytic subdiamonds\footnote{In everything that follows, a rigid analytic subdiamond of $\mc{M}_C/K$ refers to a rigid analytic subdiamond over $\Spd\+ C^{K}$, in the sense of \cref{defn:rigid-analytic-diamond}, and closed refers to the diamond topology (note that, since $C^K/L([\mu])$ is a finite extension, a diamond is rigid analytic over $C^K$ if and only it is rigid analytic over $L([\mu])$). In the minuscule case, $\mc{M}_{C}/K$ is the diamond associated to a smooth rigid analytic variety, and \cref{lemma.closed-rigid-analytic-subdiamonds-zariski-closed} says that a closed rigid analytic subdiamond of $\mc{M}_C/K$ is the same thing as a Zariski closed set in the traditional sense of rigid analytic geometry.} of $\mc{M}_{C}/K$, for any $K \leq G(\mbb{Q}_p) \times \mf{I}$ (resp. $K \leq G_b(\mbb{Q}_p) \times \mf{I}$) compact open, that contain $A$.  
    \item an \emph{irreducible Hodge (resp. Hodge-Tate) analytic set} is a connected component of the preimage in $\mc{M}_{C}$ of an irreducible locally closed rigid analytic subdiamond of $\mc{M}_{C}/K$ for $K \leq G(\mbb{Q}_p) \times \mf{I}$ (resp. $K \leq G_b(\mbb{Q}_p) \times \mf{I}$) compact open. 
    \item A locally closed subdiamond of $\mc{M}_{C}$ is \emph{bi-analytic} if it is both an irreducible Hodge analytic set and an irreducible Hodge-Tate analytic set.    
\end{itemize}
\end{definition}

\begin{remark} In the minuscule case, the Hodge/Hodge-Tate Zariski closures are equivalent to the topological closures in the inverse limit of the Zariski topologies in the Hodge/Hodge-Tate towers used in the introduction. In the non-minuscule case, already at finite level it is not clear to us whether a finite union of closed rigid analytic subdiamonds is rigid analytic. Thus, this definition may not correspond to the closure in the inverse limit of the topologies at finite level whose closed sets are generated by closed rigid analytic subdiamonds. However, the direct definition given above captures the key property we are interested in, and the question of whether it corresponds to a genuine topological closure will not arise in the proof of our main results. 
\end{remark}

We can now state the bi-analytic Ax-Lindemann theorem in full generality.
\begin{theorem}[Bi-analytic Ax-Lindemann]\label{theorem.ax-lindemann-general} Let $G/\mbb{Q}_p$ be a connected linear algebraic group,  let $[\mu]$ be a conjugacy class of cocharacters of $G_{\overline{\mbb{Q}}_p}$ admitting a basic class $[b] \in B(G,[\mu^{-1}])$, and fix $b \in G(\breve{\mbb{Q}}_p)$ representing this class. Let $\mc{M}=\mc{M}_{b,[\mu]}$. 

\begin{enumerate}
    \item Hodge (resp. Hodge-Tate) special subvarieties are Hodge (resp. Hodge-Tate) Zariski closed irreducible Hodge (resp. Hodge-Tate) analytic sets in $\mc{M}_{C}$.
    \item The Hodge-Tate (resp. Hodge) analytic Zariski closure of an irreducible Hodge (resp. Hodge-Tate) analytic set is a Hodge-Tate (resp. Hodge) special subvariety unless the defining intersection is indexed by the empty set (in which case it is trivially equal to all of $\mc{M}_{C}$).
\end{enumerate}
In particular, special subvarieties are the only closed bi-analytic sets in $\mc{M}_{C}$. 
\end{theorem}

\begin{remark} Without the basic assumption, special subvarieties may not be Hodge-Tate analytically Zariski closed in any reasonable sense. Indeed, this can be seen already in the simplest case of the generic fiber of the Serre-Tate moduli of deformations of $\mu_{p^\infty} \times \mbb{Q}_p/\mbb{Z}_p$ (corresponding to $G$ the group of invertible upper triangular matrices, $\mu=\mr{diag}(z^{-1},1)$ and $b=\mr{diag}(p,1)\sigma$). Indeed, in this case the image of the Hodge-Tate period map is a single point, and because of this there is no way to cut out the special point corresponding to the diagonal torus given by the canonical lift using only conditions on the Hodge-Tate filtration.  
\end{remark}

\subsection{Proof of \cref{theorem.ax-lindemann-general} and \cref{main.hodge-tate-image}}\label{ss.proof-main-results}

\begin{proof}[Proof of \cref{theorem.ax-lindemann-general}]
We first verify part (1) of \cref{theorem.ax-lindemann-general}. A Hodge special subvariety $\tilde{Z}$ of $\mc{M}_C:=\mc{M}_{b,[\mu],C}$ is a Hodge irreducible analytic set: by definition, $\tilde{Z}$ is a connected component of the pre-image in $\mc{M}_C$ of a special subvariety $\Gr_{[\mu_H], C}^{b_H-\adm}$ of $\Gr_{[\mu],C}^{b-\adm}$ (for $[\mu_H]$ having weights $\leq 1$ on $\Lie H$), and $\Gr_{[\mu_H]}^{b_H-\adm}$ is identified with an open (by \cref{prop.b-adm-locus-open}) subset of $\Gr_{[\mu_H]} \cong \Fl_{[\mu_H^{-1}]}^\diamond$. Thus $\Gr_{[\mu_H]}^{b_H-\adm}$ is associated to a smooth rigid analytic variety, so any connected component $T$ of its preimage in $\mc{M}_C/K$ for any compact open $K\leq G(\mathbb{Q}_p)\times \mf{I}$ is the diamond associated to a smooth and connected and thus irreducible rigid analytic variety. Since $\tilde{Z}$ is also equal to a connected component of the preimage of such a $T$, we find it is a Hodge irreducible analytic set. 

That $\tilde{Z}$ is Hodge Zariski closed follows from the description (denoting by $\pi_K \colon \mc{M}_C \to \mc{M}_C/K$ the quotient by a compact open subgroup $K \leq G(\QQ_p) \times \mf{I}$)
    \[ \tilde{Z} = \bigcap_{K \leq G(\QQ_p) \times \mf{I}} \pi_K^{-1}(\pi_K(\tilde{Z})), \]
once we show that $\pi_K(\tilde{Z}) \subseteq \mc{M}_C/K$ is closed and rigid analytic. We may assume $K \leq G(\QQ_p) \times \Gal(\overline{L}/L([\mu_H])$. Then that $\pi_K(\tilde{Z})$ is closed and rigid analytic follows as $\pi_K(\tilde{Z}) \subseteq \mc{M}_C/K$ is a connected component  of the preimage (under the \'{e}tale map $\mc{M}_C/K \to \Gr_{[\mu],L([\mu_H])}^{b-\adm}$) of the rigid analytic subdiamond $\Gr_{[\mu_H]}^{b_H-\adm}$ of $\Gr_{[\mu],L([\mu_H])}^{b-\adm}$, which is closed by \cref{prop:funct-closed-subdiamond} and \cref{lem.special-subvariety-closed}. We see similarly that a Hodge-Tate special subvariety is a Hodge-Tate Zariski closed irreducible Hodge-Tate analytic set (using either the combination of \cref{prop.moduli-pahs-X} and \cref{prop.Xmu-funct-closed} to prove it directly in the same way, or appealing to duality in \cref{theorem.duality} to reduce to the statement for Hodge special subvarieties). 

We now verify part (2) of \cref{theorem.ax-lindemann-general} (the statement following it at the end of the theorem about bi-analytic sets being an immediate consequence). By the duality of \cref{theorem.duality}, it suffices to verify only the statement about the Hodge-Tate analytic Zariski closure of an irreducible Hodge analytic set. Suppose $\tilde{Z}$ is an irreducible Hodge analytic set, i.e.\  $\tilde{Z}$ is a connected component of the pre-image in $\mc{M}_C$ of an irreducible rigid analytic subdiamond $Z$ in $\mc{M}_C/K$ for a compact open $K\leq G(\mbb{Q}_p) \times \mf{I}$. Since $\tilde{Z}$ is connected, by shrinking the Galois part of $K$ we may assume that $Z$ is geometrically connected (as a rigid analytic diamond over $C^K$).

Writing $\pi_{K_b} \colon \mc{M}_C \to \mc{M}_C/{K_b}$ for the quotient by a compact open $K_b \leq G_b(\QQ_p) \times \mf{I}$,  the Hodge-Tate Zariski  closure can be written as
\begin{equation}\label{eqn:HT-Zar-clos-intersection}
\ol{\tilde{Z}}^{\HT} = \bigcap_{ \pi_{K_b}(\tilde{Z}) \subseteq Y \subseteq \mc{M}_C/K_b} \tilde{Y}
\end{equation}
where $Y$ ranges over the closed rigid analytic subdiamonds of $\mc{M}_C/K_b$ containing $\pi_{K_b}(\tilde{Z})$ (for varying $K_b$), and $\tilde{Y} \subseteq \pi_{K_b}^{-1}(Y)$ is the connected component containing $\tilde{Z}$. It is possible that $\pi_{K_b}(\tilde{Z})$ is not contained in a closed rigid analytic subdiamond of $\mc{M}_{C}/K_b$ for any $K_b$ compact open in $G_b(\mbb{Q}_p) \times \mf{I}$: in this case, by definition the Hodge-Tate Zariski closure is an empty intersection, so equal to $\mc{M}_{C}$. 

Otherwise, fix a geometric point $x$ of $Z$ and a lift $\tilde{x}$ to $\tilde{Z}$, and let $\mc{G}$ denote the restriction of the universal $G$-admissible pair to $Z$. The choice of $\tilde{x}$ induces an isomorphism $\omega_{x}\circ \mc{G}\cong\omega_\std$; let $H =\MG(\mc{G}, x) \leq G$ denote the corresponding motivic Galois group of $\mc{G}$. We have the canonical reduction of structure group $\mc{H}$ of $\mc{G}$ to $H$ that, by \cref{lemma.ap-red-sg}, is pulled back via a factorization of $Z$ through some $\Gr_{[\mu_H]}^{b_H-\adm}$. By construction, $\tilde{Z}$ then factors through $\mc{M}_{b_H,[\mu_H]}$; let $\mc{M}^\circ$ be the connected component of $\mc{M}_{b_H,[\mu_H],C}$ containing $\tilde{Z}$. 

We will show that, if the intersection in \cref{eqn:HT-Zar-clos-intersection} defining the closure is not indexed by the empty set, then in fact the weights of $[\mu_H]$ on $\Lie\+ H$ are $\geq -1$ and any $Y$ appearing in the intersection contains the image of $\mc{M}^\circ$. It then follows from part (1) that in fact this intersection is equal to $\mc{M}^\circ$ and is Hodge-Tate special. 

To this end, we change views slightly and now suppose $Y$ is a seminormal rigid analytic variety over $\Spa\+ C^{K_b}$ such that $Y^\diamond$ appears in the index set of the intersection in \cref{eqn:HT-Zar-clos-intersection}, i.e.\ $Y^\diamond$ is a closed rigid analytic subdiamond of $\mc{M}_{b,[\mu]}/ K_b$ containing the image of $\tilde{Z}$. Since $\tilde{Z}$ is connected, we may assume that $Y^\diamond$ (and hence $Y$) is connected. The locus where the Mumford-Tate group of the canonical neutral $p$-adic
Hodge structure over $Y^\diamond$ (from the map $Y^\diamond \to X_{[\mu]}$) factors through $H$ is a closed rigid analytic subdiamond of $Y^\diamond$: writing $H$ as the stabilizer of a line $\ell$ in a representation of $G$, then this locus is the Hodge-Tate locus of $\ell$, which is Zariski closed by \cref{lemma:pahs-rigid-ananlytic-HT-locus}. Therefore restricting the intersection \eqref{eqn:HT-Zar-clos-intersection} to only include the $Y^\diamond$ that factor through $\mc{M}_{b_H, [\mu_H]}/K_b'$, where $K'_b=H(\mbb{Q}_p) \times \mf{I} \cap K_b$, does not change the Hodge-Tate Zariski closure. Thus from here on out we may simply assume $G=H$ is the Mumford-Tate group of the canonical $p$-adic Hodge structure on $Y^\diamond$.


By \cref{lemma:Hodge-generic-locus-for-rigid-analytic}, we can choose a Hodge generic classical point $z \in Z(C)$ and a lift $\tilde{z}$ in $\tilde{Z}(C)$. Let $y\in Y(C)=Y^\diamond(C)$ denote the image of $\tilde{z}$. Note $Y(C) = Y^\diamond(C)$ also contains the orbit of $y$ under a sufficiently small open subgroup of $\mf{I}$, which, because $z$ is Hodge generic, by \Iref{corollary.galois-rep-dense-open} and \cref{lemma:hodge-generic-MGG}, is equal to the orbit of $y$ by a compact open subgroup $K_0 \leq G(\mbb{Q}_p)$. Note that we may stratify $Y$ by finitely many smooth locally closed rigid analytic subvarieties. One of these, call it $V$, satisfies $V(C) \cap (K_0 \cdot y) \subseteq K_0 \cdot y$ is a non-empty open subset of $K_0 \cdot y$. Consequently $V(C)$ contains $K_0' \cdot y'$ for some $y'=k \cdot y$, $k \in K_0$, and $K_0'$ an open subgroup of $K_0$. Replacing $\tilde{z}$ with $k \cdot \tilde{z}$ and replacing $K_0$ with $K_0'$, we may assume $y=y'$ and $K_0'=K_0$, so that $K_0 \cdot y \subseteq V(C) \subseteq Y(C)$. With $E = C^{K_b}$ (a $p$-adic subfield of $C$), let $f \colon V \rightarrow \Fl_{[\mu],E}$ be the unique map of smooth rigid analytic varieties over $E$ such that $f^\diamond$ is the composition 
\[ V^\diamond \rightarrow Y^\diamond \rightarrow \mathcal{M}_{C}/K_b \rightarrow \Gr_{[\mu^{-1}],E} \xrightarrow{\BB} \Fl_{[\mu],E}^\diamond. \]
Note that $V$ contains a non-empty Zariski open subvariety where the rank of $df$ is constant equal to its maximum. If the intersection of this open set with $K_0 \cdot y$ is non-empty, we can replace $V$ with this open subvariety and again change $\tilde{z}$ and shrink $K_0$ to assume $K_0 \cdot y \subseteq V(C)$ and the rank of $df$ is constant on $V$. Otherwise $K_0 \cdot y$ is contained in the Zariski closed complement, and we may stratify this complement by smooth subvarieties and possibly change $\tilde{z}, y$, and $K_0$ to obtain a new $V$ of smaller dimension containing $K_0 \cdot y$. Repeating this process, we eventually obtain a smooth subvariety $V \subseteq Y$, a compact open subgroup $K_0 \leq G(\mathbb{Q}_p)$, and a choice of lift $\tilde{z}$ of $z$ with image $y$ in $Y(C)$ such that $K_0 \cdot y \in V(C)$ and such that, for $f: V \rightarrow \Fl_{[\mu],E}$, the rank of $df$ is constant equal to an integer $k$. 

Consider now the induced map $f_C: V(C) \rightarrow \Fl_{[\mu]}(C)$ of $p$-adic analytic manifolds (in the sense of \cite[Chapter III of Part II]{serre.lie-algebras-and-lie-groups}). We have $df_C=df|_{T_{V(C)}}$, thus we can apply the constant rank theorem \cite[Theorem on subimmersions on p.86]{serre.lie-algebras-and-lie-groups} to find that there are local charts around $y$ and $f(y)$ such that $f_C$ can be written as $(x_1,\ldots,x_m) \rightarrow (x_1, \ldots, x_k, 0, \ldots ,0)$, where $m$ is the dimension of $V$, and there are $n-k$ zeroes for $n$ the dimension of $\Fl_{[\mu]}$. Shrinking $K_0$, we may assume $K_0 \cdot y$ is contained in the domain of the chart around $y$. 

We now show $n=k$, arguing by contradiction. Otherwise, $n>k$, and taking one of the last $n-k$ coordinates in the local chart around $f(y)$ gives a non-constant function on a neighborhood of $f(y)$ whose pullback to a small neighborhood of the identity in $G(C)$ along the orbit map $G(C) \rightarrow \Fl_{[\mu]}(C)$ for $G(C)$ acting on $f(y)$ is identically zero on $K_0 \subseteq G(C)$. But, since any function which is zero on $K_0$ is zero on a neighborhood of the identity in $G(C)$ and the orbit map is a submersion, this coordinate function is constant on an open neighborhood of $f(y)$, a contradiction. 

It follows, in particular, that $df_y$ is surjective. Since Griffiths transversality for the trivial connection holds over $V$ by \cref{theorem.bdr-aff-grass-properties}-(4), we conclude that Griffiths transversality holds for the trivial $G$-bundle in all of the directions in $T_{\Fl_{[\mu]}, f(y)}$. As in the proof of \cref{theorem.bdr-aff-grass-properties}-(5), by a standard computation this implies the weights of $[\mu]$ for the adjoint action on $\Lie\+ G$ are $\geq -1$ and thus, by \cref{theorem.bdr-aff-grass-properties}-(5), that $\Gr_{[\mu^{-1}]}=\Fl_{[\mu]}^\diamond$. In particular, since $\mathcal{M}_{C}/K_b$ is \'etale over $\Fl_{[\mu],E}^\diamond$, it is a rigid analytic diamond over $E$. In a slight abuse of notation, we consider $\mc{M}_C/K_b$ as a smooth rigid analytic variety over $E$ in what follows. By \cref{lemma.closed-rigid-analytic-subdiamonds-zariski-closed}, $Y$ corresponds to a Zariski closed subvariety of $\mc{M}_C/K_b$. We have the following comparison of dimensions of tangent spaces
\[ \dim_y Y \geq \dim_y V \geq \dim_{f(y)}\Fl_{[\mu],E} = \dim_y \mathcal{M}_{C}/K_b.\]
Thus $Y$ corresponds to a connected Zariski closed subvariety of $\mathcal{M}_{C}/K_b$ with dimension at $y$ equal to the dimension of $\mathcal{M}_{C}/K_b$, and thus $Y$ is the connected component of $\mathcal{M}_{C}/K_b$ containing $y$. In particular, its preimage in $\mathcal{M}_{C}$ contains $\mathcal{M}^\circ$, since the image of $\mathcal{M}^\circ$ is a connected set containing $y$. 

\end{proof}

\begin{proof}[Proof of \cref{main.hodge-tate-image}]
\cref{theorem.rigid-an-has-type} implies the existence of the lattice invariant $[\mu]$. Then, the argument given above for \cref{theorem.ax-lindemann-general}-(ii) establishes the remainder of \cref{main.hodge-tate-image} (we have not used the basic hypothesis except to apply the duality of \cref{theorem.duality}). 
\end{proof}

\subsection{Proof of \cref{lemma:Hodge-generic-locus-for-rigid-analytic}}\label{ss.generic-lemma}

It remains to establish: 

\begin{lemma}\label{lemma:Hodge-generic-locus-for-rigid-analytic} 
Let $L/\breve{\mbb{Q}}_p$ be a $p$-adic field, $Z/\Spa\+ L$ a geometrically irreducible rigid analytic variety, and $\mc{G}$ a neutral $G$-admissible pair over $Z^\diamond$. Then the Hodge generic locus of $\mc{G}$ (see \cref{defn.Hodge-generic}) in $|Z|=|Z^\diamond|$ contains a dense set of classical points. 
\end{lemma} 

We return first to the setting where $\mc{G}$ is a neutral admissible pair on an arbitrary locally spatial diamond $S$ over $\Spd\+ \breve{\QQ}_p$. 

\begin{definition}\label{defn.Hodge-generic}
The Hodge exceptional locus of $\mc{G}$ is the union of all strictly contained Hodge loci $|\Hdg(\ell)|\subsetneq |S|$ as $\ell$ runs over rank one sub-isocrystals of $\omega_\Isoc(V)$ for all representations $V$ of $G$. The Hodge generic locus is the complement of the Hodge exceptional locus. 
\end{definition}

The following lemma gives a description of the Hodge generic locus in terms of the motivic Galois group, and is also used in the proof of \cref{theorem.ax-lindemann-general} (but not in the proof of \cref{lemma:Hodge-generic-locus-for-rigid-analytic}). 

\begin{lemma}\label{lemma:hodge-generic-MGG}
A subdiamond $T \subseteq S$ intersects the Hodge generic locus of $\mc{G}$ if and only if $\MG(\mc{G}|_T, y) = \MG(\mc{G}, y)$ for some (equivalently any) geometric point $y$ of $T$. 
\end{lemma}
\begin{proof}
Without loss of generality, we may assume $\MG(\mc{G},y) = G$ (fixing $\omega_y \circ \mc{G} \cong \omega_\std$). Now $H := \MG(\mc{G}|_T,y) \leq G$ is a proper closed subgroup if and only if there is a representation $V \in \Rep\+ G$ and a line $\ell_\et \subseteq V \cong \omega_y \circ \mc{G}(V)$ that is stabilized by $H$ but not by $G$. By the Tannakian formalism, this happens if and only if there is a rank one isocrystal $\ell \subseteq \omega_\Isoc \circ \mc{G}(V)$ that underlies a sub-admissible pair of $\mc{G}|_T$ but not of $\mc{G}$. Equivalently, $T \subseteq \Hdg(\ell) \subsetneq S$.
\end{proof}

Note that, without imposing further conditions, the Hodge generic locus may be empty. Still, we have the following useful statement, which we will invoke later to apply the Baire category theorem:

\begin{proposition}\label{prop:hodge-exceptional-countable-union}
Let $S$ be a locally spatial diamond over $\Spd\+ \breve{\QQ}_p$, and let $\mc{G}$ be a $b$-rigidified neutral $G$-admissible pair on $S$. The Hodge exceptional locus of $\mc{G}$ is a countable union of closed subsets of $|S|$.
\end{proposition}
\begin{proof}
Consider a geometric point $x \in S(C, C^+)$. By replacing $\mc{G}$ with its canonical reduction of structure group, we may assume $G=\MG(\mc{G}, x)$. Then $x$ lies in the Hodge exceptional locus if and only if the motivic Galois group $\MG(x^*\mc{G}, x)$ of $x^\ast \mc{G}$ is a proper subgroup of $G$. It thus suffices to show the following two statements:
\begin{enumerate}
    \item There exists a countable collection $\{V_\alpha\}$ of representations of $G$ such that any closed subgroup $H \leq G$ is the stabilizer of a line in some $V_\alpha$; and
    \item For any representation $V$ of $G$, let $\mc{I}$ be the set of all lines $\ell_{\breve{\QQ}_p}$ in $V \otimes_{\QQ_p} {\breve{\QQ}_p}$ such that $\ell_{\breve{\QQ}_p}$ underlies a rank one sub-isocrystal such that $\Hdg(\ell_{\breve{\QQ}_p}) \subsetneq |S|$ is a proper subset. Then $\mr{Exc}(V) := \bigcup_{\ell \in \mc{I}} \Hdg(\ell_{\breve{\QQ}_p})$ is a countable union of closed subsets of $|S|$. 
\end{enumerate}
Indeed, if these hold, then by (1) the Hodge exceptional locus is $\bigcup_{\alpha} \mr{Exc}(V_\alpha)$, and this can be rewritten as a countable union of closed subsets of $|S|$ since the index set is countable and, by (2), each $\mr{Exc}(V_\alpha)$ is a countable union of closed subsets of $|S|$. Assertion (1) follows \cite[Proof of Theorem 5.1]{borel:linear-algebraic-groups} whose argument we recall for convenience. Write the regular representation $\QQ_p[G] = \bigcup_i V_i$ as a countable filtered union of finite dimensional subrepresentations $V_i \subseteq \QQ_p[G]$ and take the collection $\{\wedge^j V_i\}$ of all exterior powers of the $V_i$. Now let $H \leq G$ be a closed subgroup. The ideal sheaf $I \subseteq \QQ_p[G]$ of $H$ is finitely generated so there exists some $V_i$ containing a set of generators of $I$. Then $W = V_i \cap I$ is a subspace of $V_i$, say of dimension $d$, and $H$ is the stabilizer of the line $\wedge^d W \subseteq \wedge^d V_i$.  

We now prove (2). Let $V \in \Rep\+ G$. For $k \in \ZZ$ define $V_k := W_b(V)^{\varphi = p^{-k}} \subseteq V_{\breve{\QQ}_p}$, a finite dimensional $\QQ_p$-vector space. Given a line $\ell \subseteq V_k$, the extension $\ell_{\breve{\QQ}_p} \subseteq V \otimes_{\QQ_p} \breve{\QQ}_p$ underlies a rank one sub-isocrystal, and any rank one sub-isocrystal of $W_b(V)$ arises from such a line in $V_k$ for some $k \in \ZZ$. Consider the incidence subdiamond
\[ I_{V_k} \subset {\mbb{P}_{V_k}(\QQ_p)} \times S \]
consisting of pairs $(\ell, s)$ such that $\ell \in \mbb{P}_{V_k}(\QQ_p)$ and $s \in \Hdg(\ell) := \Hdg(\ell_{\breve{\QQ}_p})$. This a closed subdiamond of ${\mbb{P}_{V_k}(\QQ_p)} \times S$ by essentially the same argument as in the proof of \cref{lem:HT-locus-closed}. Define $\mr{Exc}(V,k)$ to be the union of $\Hdg(\ell)$ for lines $\ell \subseteq V_k$ such that $\Hdg(\ell) \subsetneq |S|$. Then clearly $\mr{Exc}(V) = \bigcup_{k\in \ZZ} \mr{Exc}(V,k)$ and it suffices to show $\mr{Exc}(V,k)$ for a fixed $k$ is a countable union of closed subsets of $|S|$. 

The collection of lines $\ell \subseteq V_k$ such that $\Hdg(\ell) = S$ is the intersection in $\mbb{P}_{V_k}(\QQ_p)$ of all fibers of $|I_{V_k}|$ over $|S|$. Each fiber is closed in $\mbb{P}_{V_k}(\QQ_p)$ since $|I_{V_k}|$ is closed in $|{\mbb{P}_{V_k}(\QQ_p)} \times S|={\mbb{P}_{V_k}(\QQ_p)} \times |S|,$ and thus the collection of lines $\ell \subseteq V_k$ with $\Hdg(\ell) = S$ forms a closed subset of $\mbb{P}_{V_k}(\QQ_p)$. The open complement $U \subseteq \mbb{P}_{V_k}(\QQ_p)$ consists of the lines $\ell \subseteq V_k$ with $\Hdg(\ell) \subsetneq S$. Because $\mbb{P}_{V_k}(\QQ_p)$ admits a countable neighborhood basis consisting of compact open subsets, $U$ can be written as a countable union of compact (thus closed) open subsets $K_i$. Let $I_{K_i}:=I_{V_k} \cap K_i \times S$. Then $\mr{Exc}(V,k)$ is the countable union of the images of $|I_{K_i}|$ in $|S|$. As the projection map is proper, these are closed sets.
\end{proof}

The following implies \cref{lemma:Hodge-generic-locus-for-rigid-analytic} by taking $S$ to be a desingularization of $Z$ (which will be connected since $Z$ is irreducible). 

\begin{lemma}
Suppose $S$ is a smooth connected rigid analytic variety over $L$, and $\mc{G} = (W_b, \mc{L})$ is a $b$-rigidified neutral $G$-admissible pair on $S^\diamond$. Then the Hodge generic locus of $\mc{G}$ is dense in $|S| = |S^\diamond|$; in fact, it contains all weakly Shilov points (c.f.\ \cite[Definition 2.4]{bhatt-hansen:six-functors-Zariski-constructible}) and a dense set of classical points. 
\end{lemma} 
\begin{proof}

Suppose $\ell \subseteq \omega_\Isoc\circ \mc{G}(V)$ is a rank one sub isocrystal with $\Hdg(\ell) \neq S^\diamond$. Then $|\Hdg(\ell)| \subseteq |S^\diamond| = |S|$ is Zariski closed by \cref{lemma:ap-rigid-analytic-Hodge-locus} and nowhere dense by \cite[Lemma 2.1.4]{conrad:irreducible-components}. By \cite[Corollary 2.10]{bhatt-hansen:six-functors-Zariski-constructible}, we conclude that $|\Hdg(\ell)|$ does not contain any weakly Shilov points. As a consequence, the Hodge generic locus contains all weakly Shilov points of $S$ so is dense in $S$. It remains to establish the existence of a dense set of classical points. 

To that end, we first observe that the classical points in $S$ are themselves dense; thus it suffices to show that every classical point is in the closure of the set of Hodge generic classical points. Thus suppose given a point $x \in S(L')$ for some finite extension $L'/L$. To show $x$ is in the closure, it suffices to show this after replacing $S$ with $S_{L'}$. Then, there is a neighborhood basis for $x$ consisting of $n$-dimensional disks over $L'$. Thus it will suffice to show that any $n$-dimensional disk $D_n/{L'}$ in $S_{L'}$ contains a Hodge generic ${L'}$-point. Since $D_n(L') \subseteq |S_{L'}|$ with its subspace topology is naturally homeomorphic to $\mc{O}_{L'}^n$ (with the $p$-adic topology), it has the topology of a complete metric space. By the Baire category theorem it will thus suffice to show that the Hodge exceptional locus intersected with $D_n(L')$ is a countable union of closed subsets with no interior. To that end, we first observe that for any closed subset $Z$ of the Hodge exceptional locus, $Z \cap D_n(L')$ has no interior; indeed, if $Z$ contained a disk in $D_n(L')$, then it would also contain the corresponding Gauss point (e.g.\ by Raynaud's description; here we use that the residue field is algebraically closed!), but we have already established that the Hodge exceptional locus contains no weakly Shilov points. We then conclude by \cref{prop:hodge-exceptional-countable-union}. 
\end{proof}

\section{Other Ax-Lindemann theorems}\label{s.other-ax-lindemann}

The minuscule case of our main result, \cref{main.ax-lindemann}, is a bi-analytic Ax-Lindemann theorem for infinite level basic local Shimura varieties. In this setting, one analytic structure comes from the finite level local Shimura varieties, while the other analytic structure comes from the Hodge-Tate period map. In \cref{main.hodge-tate-image}, we have also given a weaker result for more general rigid analytic varieties carrying a neutral admissible pair that applies, in particular, to \emph{any} local Shimura variety (i.e., not just the \emph{basic} local Shimura varieties treated by \cref{main.ax-lindemann}). Note that, outside of the basic case, the Hodge-Tate period map is no longer a pro-\'{e}tale covering space of an open subset of the flag variety, so that it provides a much weaker notion of an analytic structure than in the basic case treated by \cref{main.ax-lindemann}. 
 
It is natural to ask whether analogous results hold in more general settings where one has a Hodge-Tate period map defined on a profinite \'{e}tale cover of a smooth rigid analytic variety. For example, one can consider infinite level global Shimura varieties, or even more general trivializing torsors for de Rham local systems. In this context, we expect our methods to be most relevant in the case of \emph{variations of $p$-adic Hodge structure}, which are a subclass of de Rham local systems. We give a preliminary discussion of this notion in \cref{ss.vpahs}, and plan to return to it in more detail in Part III. In the case of a Hodge-type global Shimura variety, the locus where the infinite level global Shimura variety is a variation of $p$-adic Hodge structure over the finite level global Shimura variety is precisely the basic locus. The basic locus can be related to a local Shimura variety by a uniformization theorem of Cerednik \cite{cerednik.algberaic-curves-that-can-be-uniformized}/Drinfeld \cite{drinfeld.Coverings-of-p-adic-symmetric-domains}/Rapoport-Zink \cite{rapoport-zink:period-spaces}/Caraiani-Scholze \cite{caraiani-scholze:cohomology-compact-shimura}/Daniels-van Hoften-Kim-Zhang \cite{daniels-van-hoften-kim-zhang.igusastackscohomologyshimura}, and thus this case is essentially covered by \cref{main.ax-lindemann}. 

We note that, over the basic locus of a Hodge-type global Shimura variety, the $\mathbb{Q}$-rational Tannakian structure group of the abelian variety\footnote{Here the Tannakian structure group of an abelian variety over a field of characteristic zero can be defined by choosing an embedding of its field of definition, which is finitely generated over $\mathbb{Q}$, into $\mathbb{C}$ and then taking the Mumford-Tate group of the associated Hodge structure.} parameterized by a geometric point is invisible to our theory. Instead, we are only seeing the $\mathbb{Q}_p$-rational Tannakian structure group of its $p$-divisible group (or rather, the associated admissible pair, in order to obtain an object living in a Tannakian category). In particular, the special loci that appear are much more abundant than the usual special loci of the theory of global Shimura varieties --- for example, it is much weaker to require that the $p$-divisible group of an abelian variety has complex multiplication than it is to require that the abelian variety itself has complex multiplication! 

In principle, to obtain a $p$-adic theory that can see the Tannakian structure group of the universal abelian variety over a global Hodge-type Shimura variety, one should develop a theory where the underlying isocrystal coefficients are global, i.e.\ live in a (necessarily non-neutral) Tannakian category over $\mathbb{Q}$ (analogous to the $\mathbb{Q}$-coefficients of singular cohomology in classical Hodge theory); this is precisely what is proposed by Scholze \cite[Conjecture 9.5]{scholze:icm2018} via a conjectural cohomology theory for varieties over $\overline{\mathbb{F}}_p$ valued in Kottwitz's \cite{kottwitz:bg} category of global isocrystals. This underlying coefficient system should exist already over the Igusa stack, and be pulled back along the morphism from the Shimura variety to the Igusa stack of \cite{zhang.peltypeigusastackpadic, daniels-van-hoften-kim-zhang.igusastackscohomologyshimura} (at least over the good reduction locus). The map to the Igusa stack alone, however, is not sufficient to develop a theory from this perspective since abelian varieties up to isogeny in characteristic $p$ (the objects parameterized by the Igusa stack) do not yet have a motivic or cohomological realization in a suitable Tannakian category over $\mathbb{Q}$.

However, even if these difficulties can be sorted, {\bf there is no reason to expect the Hodge-Tate period map to play any role in such a theory with global coefficients}. The reason is that the the underlying rational structure of a global isocrystal is completely invisible to the $p$-adic \'{e}tale cohomology, so that one cannot use the Hodge-Tate filtration to detect when a map from the global crystalline Tate object to the global coefficient object (in the category of Kottwitz's global isocrystals) extends, after specializing to the $p$-adic isocrystal, to a map from the Tate admissible pair. Instead, this is detectable only using the Hodge filtration, thus, just as in the complex case, it is only the Hodge period map that is relevant. 

Now, working over the basic locus of a finite level global Hodge-type Shimura variety, uniformization shows that the Hodge period map is defined on a natural rigid analytic covering space given by a disjoint union of basic local Shimura varieties of finite level (the story is not as simple over other Newton strata where the associated Igusa variety has a more complicated structure). Because both the Hodge period map and the uniformization map are rigid analytic, here there is only a single analytic structure in play. The Ax-Lindemann theorem one expects in this case is rather a bi-\emph{algebraic} Ax-Lindemann theorem comparing the \emph{algebraic} structure of the global Shimura variety to the \emph{algebraic} structure of the Hodge flag variety and realizing Zariski closures as (weakly) special subvarieties. In particular, one expects to be able to use this structure to characterize special subvarieties in the traditional sense of the global Shimura variety as the bi-algebraic subvarieties. This mirrors the bi-algebraic theory that one sees in the classical archimedean setting over $\mathbb{C}$. 

It would involve a considerable notational detour in order to make the ideas of the previous paragraph precise at this level of generality. Since the motivation and potential results are also largely conjectural in this generality,  we focus instead in \cref{ss.products-of-Shimura-curves} on developing in detail the concrete case of a product of Shimura curves over $\mathbb{Q}$. In this case, the basic locus is the entire Shimura variety, and a bi-algebraic Ax-Lindemann theorem has already been established by Chambert-Loir and Loeser \cite[Theorem 3.4]{chambert-loir-loeser:nonarchimedean-ax-lindemann} (as a special case of a more general result on products of Mumford curves, \cite[Theorem 2.7]{chambert-loir-loeser:nonarchimedean-ax-lindemann}). Thus we can give, in this setting, an unconditional comparison of three Ax-Lindemann theorems: the bi-analytic local Shimura Ax-Lindemann-treated in \cref{main.ax-lindemann}, an equivalent (by uniformization) bi-analytic global Shimura Ax-Lindemann, and the (very much non-equivalent) bi-algebraic global Ax-Lindemann of \cite{chambert-loir-loeser:nonarchimedean-ax-lindemann}. This discussion illustrates how the bi-analytic theory of the present work and the bi-algebraic theory of \cite{chambert-loir-loeser:nonarchimedean-ax-lindemann} are essentially disjoint. 

\subsection{Variations of $p$-adic Hodge structure}\label{ss.vpahs}

Let $L$ be a $p$-adic field and let $S/L$ be a smooth rigid analytic variety. We say a $\mathbb{Q}_p$-local system $\mathbb{L}$ on $S$ is de Rham if $L \otimes_{\ul{\mathbb{Q}_p}} \mbb{B}^+_\dR$ is in the essential image of the functor $\mathbb{M}$ from filtered vector bundles with integrable connections satisfying Griffiths transversality on $S$ to $\mbb{B}^+_\dR$-local systems on $S$ described in \cref{ss.filtered-vector-bundles-scholzes-functor-M}. By the discussion in \cref{ss.filtered-vector-bundles-scholzes-functor-M}, this is equivalent to the definition given in \cite{scholze:p-adic-ht} and the associated filtered vector bundle with integrable connection satisfying Griffiths transversality can be recovered from $\mathbb{L}$. We recall from \cite{scholze:p-adic-ht} that the cohomology local systems of smooth proper families are de Rham, so that this is a very natural condition on a local system. 

As in \cref{ss.bilatticed-equivalence-rigid-analytic}, any de Rham local system $\mathbb{L}$ is automatically equipped with a second $\mathbb{B}^+_\dR$-lattice $\mathbb{M}_0 \subseteq \mathbb{L} \otimes_{\ul{\mathbb{Q}_p}} \mathbb{B}_\dR$, coming from the trivial filtration on the associated vector bundle with integrable connection. 

\begin{definition} A
\emph{variation of $p$-adic Hodge structure} over $S$ is a $\mathbb{Q}$-graded de Rham local system $\mathbb{L}=\bigoplus_{\lambda} \mathbb{L}_\lambda$ over $S$ such that, for each $\lambda$, the restriction of $(\mathbb{L}_{\lambda}, \mathbb{M}_{0,\lambda})$ to any geometric point is a $p$-adic Hodge structure of weight $\lambda$ (in the sense of \Iref{ss.p-adic-hodge-definition}). 
\end{definition}

We note that these are not neutral $p$-adic Hodge structures over $S$ except in the case $\mathbb{L}$ is concentrated in weight $0$ and equal to a trivial local system; more general families where we allow a nontrivial $\mathbb{Q}_p$-local system will be considered in Part III. 

For $G/\mathbb{Q}_p$ a connected linear algebraic group, a variation of $p$-adic Hodge structure with $G$-structure on $S$ is an exact tensor functor from $\Rep\+ G$ to the category of variations of $p$-adic Hodge structure over $S$. These arise naturally in the following way: suppose given an open subgroup $K \leq G(\mathbb{Q}_p)$. If $\tilde{S}/S$ is a pro-\'{e}tale $K$-torsor, then, we obtain by push-out a $G(\mathbb{Q}_p)$-local system on $S$, $V \mapsto \tilde{S} \times^K V$. We say $\tilde{S}$ is de Rham if this factors through de Rham local systems, or equivalently if there is a faithful representation $V$ of $G$ such that this push-out is de Rham. 

In particular, in this case we can view the $G(\mathbb{Q}_p)$-local system as a map $\mc{E}_{\tilde{S}}: S^\diamond \rightarrow \Bun_G^{[1]}$ (by \cref{eq.isomorphism-b-loc-classifying-stack}) and then use the lattice $\mathbb{M}_0$ to make a modified bundle $\mc{E}_{\tilde{S}}^{\dR}: S^\diamond \rightarrow \Bun_G$ as in \cref{ss.modifications}. If $\mathcal{E}_{\tilde{S}}^{\dR}$ factors through $\Bun_{G}^{[b]}$ for $b$ basic, then the slope morphism for $b$ induces a grading  that allows us to upgrade the $G(\mathbb{Q}_p)$-local system to a variation of $p$-adic Hodge structure with $G$-structure over $S$ (this is essentially due to a relative version of the equivalence between basic admissible pairs and $p$-adic Hodge structure treated in \Iref{theorem.adm-pair-propeties} and will be discussed further in Part III). 

Pulling back the universal $G_b(\mathbb{Q}_p)$-torsor of \cref{eq.isomorphism-b-loc-classifying-stack} from $\Bun_G^{[b]}$, we obtain a $G_b(\mathbb{Q}_p)$-torsor $\tilde{S}^\cris/S$. By construction, it carries a canonically neutralized good $b$-admissible pair with $G$-structure; in particular, it admits a classifying \'{e}tale lattice period map $\pi_{\mathcal{L}_\et}: \tilde{S}^\cris \rightarrow \Gr_{[\mu]}^{b-\adm}$ (and associated Hodge period map by composing with the Bialynicki-Birula map) for a conjugacy class of cocharacters $[\mu]$ of $G_{\overline{\mathbb{Q}}_p}$ (the Hodge cocharacter of the associated filtered vector bundle with $G$-structure on $S$). The pullback of the universal $G(\mathbb{Q}_p)$-torsor from $\Gr_{[\mu]}^{b-\adm}$ is naturally identified with $\tilde{S} \times^K G(\mathbb{Q}_p)$, so that the reduction of structure group from $G(\mathbb{Q}_p)$ to $K$ furnished by $\tilde{S}$ can be viewed as a lift of $\pi_{\mc{L}_\et}$ to a map
\begin{equation}\label{eq.variation-of-p-adic-hs-shtuka-period-map} \tilde{S}^\cris \rightarrow \mathcal{M}_{b,[\mu]}/K. \end{equation}

We \emph{expect} that the cover $\tilde{S}^\cris / S$ is potentially unramified (we will make a precise conjecture in Part III). This is suggested, for example, by the fact that it is de Rham with Hodge-Tate weights zero. In particular, we expect that $\tilde{S}^\cris$ is represented by an adic space over $L$ that behaves, at least for differential purposes, as nicely as a smooth rigid analytic variety over $L$. Thus we hope that results analogous to \cref{theorem.bdr-aff-grass-properties}-(4) and \cref{lemma:Hodge-generic-locus-for-rigid-analytic} hold in this setting, so that the methods applied to prove \cref{main.ax-lindemann} and \cref{main.hodge-tate-image} can also be applied on $\tilde{S}^\cris$. 

In this generality, one cannot expect to do better than a potentially unramified cover. In particular, one cannot expect in general to refine the map in \cref{eq.variation-of-p-adic-hs-shtuka-period-map} by replacing $\tilde{S}^\cris$ with a rigid analytic variety that is, e.g., an \'{e}tale de Jong covering space of $S$. For example, let $S$ be the ordinary locus of a connected finite level $p$-adic modular curve with no level structure at $p$, and let $\tilde{S}/S$ to be the Katz-Igusa tower, i.e.\ the $\mathbb{Z}_p^\times$ local system trivializing the canonical subgroup of the Tate module of the universal ordinary elliptic curve over $S$. In this case, we have $\tilde{S}^\cris=\tilde{S}$ (e.g., because trivializing the canonical subgroup is the same as trivializing the unit root isocrystal). Now, $\tilde{S}$ is represented by the generic fiber of a tower of formal schemes finite \'{e}tale over a formal model of $S$, thus $\tilde{S}$ is represented by a very nice adic space, but it is not a rigid analytic space over $L$ because it is not of finite type. Moreover, because the monodromy representation is surjective onto $\mathbb{Z}_p^\times$ (i.e. the Katz-Igusa tower is connected), it is not possible to reduce the cover to a rigid analytic variety! 

However, there is one important case where $\tilde{S}^\cris$ can be replaced with a rigid analytic variety in a very satisfactory way: Suppose given a Hodge-type Shimura datum $(G,X)$ and a neat level subgroup 
\[ K=K_pK^p \subseteq G(\mathbb{A}_f)=G(\mathbb{Q}_p) \times G(\mathbb{A}_f^{(p)}). \]
Let $L$ be a $p$-adic completion of the reflex field, let $S$ be the basic locus of the rigid generic fiber of the associated Shimura variety over $L$, and let $\tilde{S}/S$ be the restriction of the the universal torsor for $K_p \leq G(\mathbb{Q}_p)$ (i.e.\ $\tilde{S}$ is the restriction of the associated infinite level Shimura variety to the basic locus). Then, the uniformization of the basic locus\footnote{In this generality, the necessary result can be deduced from the version of the Igusa stacks fiber product conjecture established in \cite{daniels-van-hoften-kim-zhang.igusastackscohomologyshimura}.} implies that $S$ can be written as a disjoint union of components of the form $\mathcal{M}_{b,[\mu]}/K_p \times \Gamma$ where $\Gamma$ is a discrete subgroup of $G_b(\mathbb{Q}_p)$ acting on the finite level local Shimura variety $\mathcal{M}_{b,[\mu]}/K_p$. In other words, $S$ can be covered by connected components where $\tilde{S}^\cris$ admits a reduction of structure group to an \'{e}tale covering space for a discrete group, and that \'{e}tale covering space is the local Shimura variety $\mathcal{M}_{b,[\mu]}/K_p$. In particular, the diagram 
\[\begin{tikzcd}
	& {\tilde{S}} \\
	S && {\Fl_{[\mu]}}
	\arrow[from=1-2, to=2-1]
	\arrow["{\pi_{\HT}}"{description}, from=1-2, to=2-3]
\end{tikzcd}\]
is, locally in the \'{e}tale topology on $S$, isomorphic to a diagram 
\[\begin{tikzcd}
	& {\mathcal{M}_{b,[\mu]}} \\
	{\mathcal{M}_{b,[\mu]}/K_p} && {\Fl_{[\mu]}}
	\arrow[from=1-2, to=2-1]
	\arrow["{\pi_{\HT}}"{description}, from=1-2, to=2-3]
\end{tikzcd}\]
In particular, \cref{main.ax-lindemann} can be applied \'{e}tale locally to $\tilde{S}$ and thus it can be viewed as a bi-analytic Ax-Lindemann theorem also for $\tilde{S}$. In the next section, we discuss this more explicitly for products of Shimura curves over $\mathbb{Q}$.

\subsection{Products of Shimura curves over $\mathbb{Q}$}\label{ss.products-of-Shimura-curves}

We now consider the theory of the previous section in the explicit example of products of Shimura curves over $\mathbb{Q}$, and also compare our bi-analytic Ax-Lindemann theorem to the bi-algebraic Ax-Lindemann theorem of \cite{chambert-loir-loeser:nonarchimedean-ax-lindemann}, which applies in this setting. 

Let $D/\mathbb{Q}$ be a quaternion algebra unramified at $\infty$ and ramified at $p$. Let $K^p$ be a neat compact open subgroup of $D^\times(\mathbb{A}_f^{(p)})$ and let $\mathcal{O}_p$ be the unique maximal order in $D \otimes \mathbb{Q}_p$ so that $\mathcal{O}_p^\times$ is a maximal compact open subgroup of $D^\times(\mathbb{Q}_p)$. Attached to this data there is an associated Shimura curve $\Sh_{\mathcal{O}_p^\times K^p}$, a smooth proper algebraic curve over $\mathbb{Q}$. 

Let $S$ denote the rigid analytic generic fiber of the base change of $\Sh_{\mathcal{O}_p^\times K^p}$ to $\breve{\mathbb{Q}}_p$.  In this case, the basic locus is all of $S$, and thus we have a $p$-adic uniformization of $S$. Explicitly, for $B$ a quaternion algebra that is ramified at $\infty$ and split at $p$ (and otherwise ramified at the same places as $D$), there is a Cerednik-Drinfeld uniformization
\[ B^\times(\mathbb{Q}) \backslash B^\times(\mathbb{A}_f) \times_{\GL_2(\mathbb{Q}_p)} \mathcal{M}_{b,[\mu]} / \mathcal{O}_{p}^\times K^p \xrightarrow{\sim} S. \]
For details, we refer the reader to \cite{boutot-carayol.UniformisationpAdiqueDesCourbesDeShimura}. 
Note that we can identify $\mathcal{M}_{b,[\mu]}/\mathcal{O}_p^\times= \bigsqcup_{p^{\mathbb{Z}}} \Omega$; the action of $\GL_2(\mathbb{Q}_p)$ on components is through the valuation of the determinant, so it is transitive on the components and the stabilizer of the component $\Omega$ corresponding to $1=p^0$ is the subgroup $\GL_2(\mathbb{Q}_p)^\circ \leq \GL_2(\mathbb{Q}_p)$ of elements of determinant in $\mathbb{Z}_p^\times$. We thus can rewrite the uniformization as 
\[ B^\times(\mathbb{Q}) \backslash B^\times(\mathbb{A}_f) \times_{\GL_2(\mathbb{Q}_p)^\circ} \Omega / \mathcal{O}_{p}^\times K^p \xrightarrow{\sim} S. \]

It follows that $S$ is a finite disjoint union of spaces $\Omega/\Gamma$ indexed by the finite set $B^\times(\mathbb{Q})\backslash p^\mathbb{Z} \times B^\times(\mathbb{A}_f^{(p)})/K^p$, and where the discrete subgroup $\Gamma$ can be chosen as the subgroup of $B^\times(\mathbb{Q})$ stabilizing a representative in $p^\mathbb{Z} \times B^\times(\mathbb{A}_f^{(p)})/K^p.$ Moreover, in this case $\Gamma$ acts properly discontinuously on $\Omega$, so that the map $\Omega \rightarrow \Omega/\Gamma$ is a covering space already in the analytic topology; in particular, it is an isomorphism locally in the analytic topology.

More generally, we can take $S=\prod S_i$ to be a product of Shimura curves as above. In this setting, we write $\pi_{\mathcal{O}_{p}^\times, \mathrm{glob}}:\tilde{S} \rightarrow S$ for the associated for the infinite level Shimura variety over $S$. We fix a connected component $S^\circ$ of $S$ that, by applying the Cerednik-Drinfeld uniformization described above to each term of the product, can be written as $S^\circ=\Omega^n/\Gamma$ for $\Gamma=\prod_{i=1}^n \Gamma_i$. We write $\mathcal{M}_{b,[\mu]}$ for the associated infinite level local Shimura variety, which in this case is the product of $n$ copies of the infinite level Drinfeld space, or, equivalently by the basic duality in this case, the infinite level Lubin-Tate space. In particular, $\mathcal{M}_{b,[\mu]}$ is a perfectoid space by results of \cite{scholze-weinstein:moduli-of-p-divisible-groups}. Similarly, the infinite level Shimura variety $\tilde{S}/S$ is also a perfectoid space by results of \cite{scholze:torsion}. Both $\tilde{S}$ and $\mathcal{M}_{b,[\mu]}$ have Hodge-Tate period maps, which we denote by $\pi_{\HT,\mr{glob}}$ and $\pi_{\HT,\mr{loc}}$, and these are compatible, as described in \cite{caraiani-scholze:cohomology-compact-shimura}. We thus have the following commutative diagram of adic spaces over $\breve{\mathbb{Q}}_p$: 
\[\begin{tikzcd}
	&&& {\mathcal{M}_{b,[\mu]}} \\
	{(\mathbb{P}^1)^n} & {(\Omega)^n} & {\left(\bigsqcup_{p^\mathbb{Z}} \Omega\right)^n} & {\tilde{S}|_{S^\circ}=\mathcal{M}_{b,[\mu]}/\Gamma'} & {(\mathbb{P}^1)^n} \\
	& {\Omega^n/\Gamma} & {S^\circ=\left(\bigsqcup_{p^\mathbb{Z}} \Omega\right)^n/\Gamma'} && {(\mathbb{P}^1)^n}
	\arrow["{\pi_{\mathcal{O}_p^\times,\mr{loc}}}"{description}, color={rgb,255:red,92;green,92;blue,214}, from=1-4, to=2-3]
	\arrow[from=1-4, to=2-4]
	\arrow["{\pi_{\HT,\mr{loc}}}"{description}, color={rgb,255:red,92;green,92;blue,214}, from=1-4, to=2-5]
	\arrow[draw={rgb,255:red,214;green,92;blue,92}, hook', from=2-2, to=2-1]
	\arrow[hook, from=2-2, to=2-3]
	\arrow[color={rgb,255:red,214;green,92;blue,92}, from=2-2, to=3-2]
	\arrow[from=2-3, to=3-3]
	\arrow["{\pi_{\mathcal{O}_p^\times,\mr{glob}}}"{description}, color={rgb,255:red,214;green,92;blue,214}, from=2-4, to=3-3]
	\arrow["{\pi_{\HT,\mathrm{glob}}}"{description}, color={rgb,255:red,214;green,92;blue,214}, from=2-4, to=3-5]
	\arrow["{=}"{description}, no head, from=2-5, to=3-5]
	\arrow[from=3-2, to=3-3]
\end{tikzcd}\]
where $\Gamma'$ is the variant of $\Gamma$ where we have not fixed the determinants to have $p$-adic valuation $0$. The top triangle on the right-hand side, whose arrows are dark purple, is the realm of \cref{main.ax-lindemann}. The triangle below it, whose arrows are pink, is the realm of a global bi-analytic Ax-Lindemann deduced from \cref{main.ax-lindemann} as described in \cref{ss.vpahs} (and here we are in the even nicer situation where the diagrams are isomorphic locally already in the analytic topology on $S^\circ$). Finally, the left-most triangle, whose arrows are red, is the realm of a global bi-algebraic Ax-Lindemann theorem as in \cite[Theorem 2.7]{chambert-loir-loeser:nonarchimedean-ax-lindemann}. In particular, note that the two different \emph{algebraic} structures appearing in the latter result (the one coming from $(\mathbb{P}^1)^n$ and the one coming from $S^\circ$) both underlie the same \emph{analytic} structure on $\Omega^n$, and that the infinite level Shimura variety and its second analytic structure corresponding to the Hodge-Tate period map play no role whatsoever in the bi-algebraic theory.

\bibliographystyle{plain}
\bibliography{refs}

\end{document}